\documentclass[11pt,reqno]{amsart}

\usepackage{amssymb,amscd,amsxtra}
\usepackage{xypic}

\usepackage{hyperref}
\hypersetup{colorlinks=false}

\theoremstyle{plain}

\newtheorem{thm}{Theorem}[section]
\newtheorem{lem}[thm]{Lemma}
\newtheorem{cor}[thm]{Corollary}
\newtheorem{prop}[thm]{Proposition}
\newtheorem{mainthm}{Theorem}

\theoremstyle{definition}

\newtheorem{defn}[thm]{Definition}

\newtheorem{ex}[thm]{Example}
\newtheorem{prob}[thm]{Problem}
\newtheorem{quest}[thm]{Question}

\theoremstyle{remark}

\newtheorem{rem}{Remark}
\newtheorem{claim}{Claim}

\newcommand{\N}{\mathbb{N}}
\newcommand{\Z}{\mathbb{Z}}
\newcommand{\R}{\mathbb{R}}

\newcommand{\CC}{\mathcal{C}}

\newcommand{\FF}{\mathcal{F}}
\newcommand{\GG}{\mathcal{G}}
\newcommand{\HH}{\mathcal{H}}

\newcommand{\KK}{\mathcal{K}}

\newcommand{\MM}{\mathcal{M}}
\newcommand{\NN}{\mathcal{N}}

\newcommand{\RR}{\mathcal{R}}

\newcommand{\UU}{\mathcal{U}}
\newcommand{\VV}{\mathcal{V}}
\newcommand{\WW}{\mathcal{W}}

\newcommand{\im}{\operatorname{im}}

\newcommand{\id}{\operatorname{id}}

\DeclareMathOperator{\diam}{diam}
\DeclareMathOperator{\vol}{vol}

\DeclareMathOperator{\dom}{dom}

\DeclareMathOperator{\card}{card}
\DeclareMathOperator{\germ}{\boldsymbol{\gamma}}

\DeclareMathOperator{\Cl}{Cl}

\newcommand{\const}{\operatorname{const}}
\newcommand{\Loct}{\operatorname{Loct}}
\newcommand{\Par}{\operatorname{Par}}
\newcommand{\Paro}{\operatorname{Paro}}
\newcommand{\pr}{\operatorname{pr}}
\newcommand{\rK}{\mathring{K}}
\newcommand{\ol}{\overline}

\newdimen\theight
\def\TeXref#1{%
             \leavevmode\vadjust{\setbox0=\hbox{{\tt
                     \quad\quad  {\small \textrm #1}}}%
             \theight=\ht0
             \advance\theight by \lineskip
             \kern -\theight \vbox to
             \theight{\rightline{\rlap{\box0}}%
             \vss}%
             }}%

\title{Topological Molino's theory}

\author[J.A. \'Alvarez L\'opez]{Jes\'us A. \'Alvarez L\'opez}
\address{Departamento de Xeometr\'{\i}a e Topolox\'{\i}a\\
         Facultade de Matem\'aticas\\
         Universidade de Santiago de Compostela\\
         15782 Santiago de Compostela\\ Spain}
\email{jesus.alvarez@usc.es}
\thanks{The first author is partially supported by the MICINN grant MTM2011-25656}

\author[M. Moreira Galicia]{Manuel F. Moreira Galicia}
\address{Departamento de Xeometr\'{\i}a e Topolox\'{\i}a\\
         Facultade de Matem\'aticas\\
         Universidade de Santiago de Compostela\\
         15782 Santiago de Compostela\\ Spain}
\email{manuel.moreira@usc.es, morgal2002@gmail.com}
\thanks{The second author was partially supported by the grant ERASMUS MUNDUS LOTE~20B, for 
the period 2010--2012}

\date{\today}



\begin{document}

\maketitle

\begin{abstract}
	Molino's description of Riemannian foliations on compact manifolds is generalized to the setting of compact equicontinuous foliated spaces, in the case where the leaves are dense. In particular, a structural local group is associated to such a foliated space. As an application, we obtain a partial generalization of results by Carri\`ere and Breuillard-Gelander, relating the structural local group to the growth of the leaves.
\end{abstract}

\tableofcontents

\section{Introduction}\label{s: intro}

Riemannian foliations were introduced by Reinhart \cite{Reinhart1959} by requiring isometric transverse dynamics. It was pointed out by Ghys in \cite[Appendix E]{Molino1988} (see also Kellum's paper \cite{Kellum1993}) that equicontinuous foliated spaces should be considered as the ``topological Riemannian foliations,'' and therefore many of the results about Riemannian foliations should have versions for equicontinuous foliated spaces. Some steps in this direction were given by \'Alvarez and Candel \cite{AlvCandel2009,AlvCandel2010}, showing that, under reasonable conditions, their leaf closures are minimal foliated spaces, and their generic leaves are quasi-isometric to each other, like in the case of Riemannian foliations. In the same direction, Matsumoto \cite{Matsumoto2010} proved that any minimal equicontinuous foliated space has a non-trivial transverse invariant measure, which is unique up to scaling if the space is compact---observe that this unicity implies ergodicity. The magnitude of the generalization from Riemannian foliations to equicontinuous foliated spaces was made precise by \'Alvarez and Candel \cite{AlvCandel2010} (see 
also Tarquini's paper \cite{Tarquini2004}),
 giving a topological description of Riemannian foliations within the class of equicontinuous foliated spaces.

Most of the known properties of Riemannian foliations follow from a description due to Molino \cite{Molino1982,Molino1988}. However, so far, there was no version of Molino's description for equicontinuous foliated spaces---the indicated properties of equicontinuous foliated spaces were obtained by other means. The goal of our work is to develop such a version of Molino's theory, and use it to study the growth of their leaves, following the study of the growth of Riemannian foliations by Carri\`ere \cite{Carriere1988} and Breuillard-Gelander \cite{BreuillardGelander2007}. To understand our results better, let us briefly recall Molino's theory.

\subsection{Molino's theory for Riemannian foliations}

The needed basic concepts from foliation theory can be seen in \cite{HectorHirsch1981-A,HectorHirsch1983-B,CandelConlon2000-I}.

Let $\FF$ be a (smooth) foliation of codimension $q$ on a manifold $M$. Let $T\FF\subset TM$ denote the vector subbundle of vectors tangent to the leaves, and $N\FF=TM/T\FF$ its normal bundle. Recall that there is a natural flat leafwise partial connection on $N\FF$ so that any local normal vector field is leafwise parallel if and only if it is locally projectable by the distinguished submersions; terms like ``leafwise flat,'' ``leafwise parallel'' and ``leafwise horizontal" will refer to this partial connection. It is said that $\FF$ is:
	\begin{description}
	
		\item[Riemannian] if $N\FF$ has a leafwise parallel Riemannian structure;
		
		\item[transitive] if the group of its foliated diffeomorphisms acts transitively on $M$;
		
		\item[transversely parallelizable (TP)] if there is a leafwise parallel global frame of $N\FF$, called {\em transverse parallelism\/}; and a
		
		\item[Lie foliation] if moreover the transverse parallelism is a basis of a Lie algebra with the operation induced by the vector field bracket.
	
	\end{description}
These conditions are successively stronger. Molino's theory describes Riemannian foliations on compact manifolds in terms of minimal Lie foliations, and using TP foliations as an intermediate step:
	\begin{description}
	
		\item[1st step] If $\FF$ is Riemannian and $M$ compact, then there is an $\operatorname{O}(q)$-principal bundle, $\hat\pi:\widehat M\to M$, with an $\operatorname{O}(q)$-invariant TP foliation,  $\widehat\FF$, such that $\widehat\pi$ is a foliated map whose restrictions to the leaves are the holonomy covers of the leaves of $\FF$.
		
		\item[2nd step] If $\FF$ is TP and $M$ compact, then there is a fiber bundle $\pi:M\to W$ whose fibers are the leaf closures of $\FF$, and the restriction of $\FF$ to each fiber is a Lie foliation.
	
	\end{description}
Since the structure of Lie foliation is unique in the minimal case, we end up with a Lie algebra associated to $\FF$, called {\em structural Lie algebra\/}. The proofs of the above statements strongly use the differential structure of $\FF$. In the first step, $\hat\pi:\widehat M\to M$ is the $\operatorname{O}(q)$-principal bundle of orthonormal frames for some leafwise parallel metric on $N\FF$, and $\widehat\FF$ is given by the corresponding flat leafwise horizontal distribution. Then $\widehat\FF$ is TP by a standard argument. In the second step, foliated flows are used to produce the fiber bundle trivializations whose fibers are the leaf closures; this works because there are foliated flows in any transverse direction since $\FF$ is TP. 

When $\FF$ is minimal (the leaves are dense), any leaf closure $\widehat M_0$ of $\widehat\FF$ is a principal subbundle of $\hat\pi:\widehat M\to M$, obtaining the following:
	\begin{description}
	
		\item[Minimal case] If $\FF$ is minimal and Riemannian, and $M$ is compact, then, for some closed subgroup $H\subset \operatorname{O}(q)$, there is an $H$-principal bundle, $\hat\pi_0:\widehat M_0\to M$, with an $H$-invariant minimal Lie foliation, $\widehat\FF_0$, such that $\widehat\pi_0$ is a foliated map whose restrictions to the leaves are the holonomy covers of the leaves of $\FF$.
	
	\end{description}

A useful description of Lie foliations was also given by Fedida \cite{Fedida1971,Fedida1973}, but it will not be considered here.

The differential structure cannot be used in our generalization; instead, we use the holonomy pseudogroup. Thus let us briefly indicate the holonomy properties of Riemannian foliations that will play an important role.

\subsection{Holonomy of Riemannian foliations}

Recall that a {\em pseudogroup\/} is a maximal collection of local transformations of a space, which contains the identity map, and is closed under the operations of composition, inversion, restriction and combination. It can be considered as a generalized dynamical system, and all basic dynamical concepts have pseudogroup versions. They are relevant in foliation theory because the holonomy pseudogroup of a foliation $\FF$ describes the transverse dynamics of $\FF$. Such a pseudogroup is well determined up to certain {\em equivalence\/} of pseudogroups introduced by Haefliger \cite{Haefliger1985,Haefliger1988}. We may say that $\FF$ is {\em transversely modeled\/} by a class of local transformations of some space if its holonomy pseudogroup can be generated by that type of local transformations. Riemannian, TP and Lie foliations can be respectively characterized by being transversely modeled by local isometries of some Riemannian manifold, local parallelism preserving diffeomorphisms of some parallelizable manifold, and local left translations of a Lie group. In this sense, Riemannian foliations are the transversely rigid ones, and TP foliations have a stronger type of transverse rigidity.

When the ambient manifold $M$ is compact, Haefliger \cite{Haefliger2002} has observed that the holonomy pseudogroup $\HH$ of $\FF$ satisfies a property called compact generation. If moreover $\FF$ is Riemannian, then Haefliger \cite{Haefliger1988,Haefliger2002} has also strongly used the following properties of $\HH$: completeness, quasi-analyticity, and existence of a closure $\ol{\HH}$, which is also complete and quasi-analytic. Here, $\ol\HH$ is defined by taking the closure of the set of $1$-jets of maps in $\HH$ in the space of $1$-jets.

For a compactly generated pseudogroup $\HH$ of local isometries of a Riemannian manifold $T$, Salem has given a version of Molino's theory \cite{Salem1988}, \cite[Appendix~D]{Molino1988} (see also \cite{AlvMasa2008}). In particular, in the minimal case, it turns out that there is a Lie group $G$, a compact subgroup $K\subset G$ and a dense finitely generated subgroup $\Gamma\subset G$ such that $\HH$ is equivalent to the pseudogroup generated by the action of $\Gamma$ on the homogeneous space $G/K$ (this was also observed by Haefliger \cite{Haefliger1988}).

\subsection{Growth of Riemannian foliations}

Molino's theory has many consequences for a Riemannian foliation $\FF$ on a compact manifold $M$: classification in particular cases, growth, cohomology, tautness, tenseness and global analysis. In all of them, Molino's theory is used to reduce the study to the case of Lie foliations with dense leaves, where it usually becomes a problem of Lie theory. We concentrate on the consequences about growth of the leaves and their holonomy covers. This study was begun by Carri\`ere \cite{Carriere1988}, and recently continued by Breuillard-Gelander, as a consequence of their study of a topological Tits alternative \cite{BreuillardGelander2007}. Their results state the following, where $\mathfrak g$ is the structural Lie algebra of $\FF$:
	\begin{description}
	
		\item[Carri\`ere's theorem] The holonomy covers of the leaves are F\o lner if and only if $\mathfrak g$ is solvable, and of polynomial growth if and only if $\mathfrak g$ is nilpotent. In the second case, the degree of their polynomial growth is bounded by the nilpotence degree of $\mathfrak g$.
		
		\item[Breuillard-Gelander's theorem] The growth of the holonomy covers of the leaves is either polynomial or exponential.
	
	\end{description}

\subsection{Equicontinuous foliated spaces}

A {\em foliated space\/} $X\equiv(X,\FF)$ is a topological space $X$ equipped with a partition $\FF$ into connected manifolds ({\em leaves\/}), which can be locally described as the fibers of topological submersions. It will be assumed that $X$ is locally compact and Polish. A foliated space should be considered as a ``topological foliation''. In this sense, all topological notions of foliations have obvious versions for foliated spaces. In particular, the {\em holonomy pseudogroup\/} $\HH$ of $X$ is defined on a locally compact Polish space $T$. Many basic results about foliations also have straightforward generalizations; for example, $\HH$ is compactly generated if $X$ is compact. Even leafwise differential concepts are easy to extend. However this task may be difficult or impossible for transverse differential concepts. For instance, the normal bundle of a foliated space does not make any sense in general: it would be the tangent bundle of 
a topological space in the case of a space foliated by points. Thus the concept of Riemannian foliation cannot be extended by using the normal bundle; instead, this can be done via the holonomy pseudogroup as follows.

The transverse rigidity of a Riemannian foliation can be translated to the foliated space $X$ by requiring equicontinuity of $\HH$; in fact, the equicontinuity condition is not compatible with combinations of maps, and therefore it is indeed required for some generating subset $S\subset\HH$ which is closed by the operations of composition and inversion; such an $S$ is called a pseudo$*$group with the terminology of Matsumoto \cite{Matsumoto2010}. This gives rise to the concept of equicontinuous foliated space.

In the topological setting, the quasi-analyticity of $\HH$ does not follow from the equicontinuity assumption. Thus it will be required as an additional assumption when needed. Indeed, it does not work well enough when $T$ is not locally connected. So we use a property called strong quasi-analyticity, which is stronger than quasi-analyticity only when $T$ is not locally connected.
  
\'Alvarez and Candel \cite{AlvCandel2009} have proved that, if $\HH$ is compactly generated, equicontinuous and strongly quasi-analytic, then it is complete and has a closure $\ol\HH$. Here, $\ol\HH$ is the pseudogroup generated by the homeomorphisms on small enough open subsets $O$ of $T$ that are limits in the compact-open topology of maps in $\HH$ defined on those sets $O$.

Transitive and Lie foliations have the following topological versions. It is said that the foliated space $X$ is:
	\begin{description}
	
		\item[homogeneous] if its group of foliated transformations acts transitively on $X$.
		
		\item[$G$-foliated space] if it is transversely modeled by local left translations in some locally compact Polish local group $G$ (if $X$ is minimal).
		
	\end{description}

\subsection{Topological Molino's theory}

Our first main result is the following topological version of the minimal case in Molino's theory.

\begin{mainthm}\label{mt: topological Molino}
	Let $X\equiv(X,\FF)$ be a compact Polish foliated space, and $\HH$ its holonomy pseudogroup. Suppose that $X$ is minimal and equicontinuous, and $\ol\HH$ is strongly quasi-analytic. Then there is a compact Polish minimal foliated space $\widehat X_0\equiv(\widehat X_0,\widehat\FF_0)$, an open continuous foliated map $\hat\pi_0:\widehat X_0\to X$,  and a locally compact Polish local group $G$ such that $\widehat X_0$ is a $G$-foliated space, the fibers of $\hat\pi_0$ are homeomorphic to each other, and the restrictions of $\hat\pi_0$ to the leaves of $\widehat\FF_0$ are the holonomy covers of the leaves of $\FF$.
\end{mainthm}

The proof of Theorem~\ref{mt: topological Molino} is different from Molino's proof in the Riemannian foliation case because there may not be the normal bundle of $\FF$. To define $\widehat X_0$, we first construct what should be its holonomy pseudogroup, $\widehat\HH_0$ on a space $\widehat T_0$. To some extent, this was achieved by  \'Alvarez and Candel \cite{AlvCandel2010}, proving that, with the assumptions of Theorem~\ref{mt: topological Molino}, there is a locally compact Polish local group $G$, a compact subgroup $K\subset G$ and a dense finitely generated sub-local group $\Gamma\subset G$ such that $\HH$ is equivalent to the pseudogroup generated by the local action of $\Gamma$ on $G/K$, like in the Riemannian foliation case. Hence $\widehat\HH_0$ should be the pseudogroup generated by the local action of $\Gamma$ on $G$. This may look like a big step towards the proof, but the realization of compactly generated pseudogroups as holonomy pseudogroups of compact foliated spaces is impossible in general, as shown by Meigniez \cite{Meigniez2010}. This difficulty is overcome as follows.

Take a ``good'' cover of $X$ by distinguished open sets, $\{U_i\}$, with corresponding distinguished submersions $p_i:U_i\to T_i$, and elementary holonomy transformations $h_{ij}:T_{ij}\to T_{ji}$, where $T_{ij}=p_i(U_i\cap U_j)$. Let $\HH$ denote the corresponding representative of the holonomy pseudogroup on $T=\bigsqcup_iT_i$, generated by the maps $h_{ij}$. Then the construction of $\widehat\HH_0$ must be associated to $\HH$ in a natural way, so that it becomes induced by some ``good'' cover by distinguished open sets of a compact foliated space. In the Riemannian foliation case, the good choices of $\widehat T_0$ and $\widehat\HH_0$ are the following ones:
	\begin{itemize}
	
		\item Let $P$ be the bundle of orthonormal frames for any $\HH$-invariant metric on $T$. Fix $x_0\in T$ and $\hat x_0\in P_{x_0}$. Then, as subspace of $P$,
			\begin{align}
				\widehat T_0&=\ol{\{\,h_*(\hat x_0)\mid h\in\HH,\ x_0\in\dom h\,\}}\notag\\
				&=\{\,g_*(\hat x_0)\mid g\in\ol\HH,\ x_0\in\dom g\,\}\;.\label{widehat T_0}
			\end{align}
		
		\item $\widehat\HH_0$ is generated by the differentials of the maps in $\HH$.
			
	\end{itemize}
These differential concepts can be modified in the following way. In~\eqref{widehat T_0}, each $g_*(\hat x_0)$ determines the germ of $g$ at $x_0$, $\germ(g,x_0)$, by the strong quasi-analyticity of $\ol\HH$. Therefore it also determines $\germ(f,x)$, where $f=g^{-1}$ and $x=g(x_0)$---this little change, using $\germ(f,x)$ instead of $\germ(g,x_0)$, is not really necessary, but it helps to simplify the notation in some involved arguments. So
			\begin{equation}\label{germs}
				\widehat T_0\equiv\{\,\germ(f,x)\mid f\in\ol\HH,\ x\in\dom f,\ f(x)=x_0\,\}\;.
			\end{equation}
The projection $\hat\pi_0:\widehat T_0\to T$ corresponds via~\eqref{germs} to the source map $\germ(f,x)\mapsto x$. The differentials of maps $h\in\HH$, acting on orthonormal references, correspond via~\eqref{germs} to the maps $\hat h$ defined by 
			$$
				\hat h(\germ(f,x))=\germ(fh^{-1},h(x))\;.
			$$
Let us describe the topology of $\widehat T_0$ using~\eqref{germs}. Let $\ol S$ be a pseudo$*$group generating $\ol\HH$ and satisfying the equicontinuity and strong quasi-analyticity conditions. Endow $\ol S$ with the compact-open topology on partial maps with open domains, as defined by Abd-Allah-Brown \cite{Abd-AllahBrown1980}, and consider the subspace
			$$
				\ol S*T=\{\,(f,x)\in\ol S\mid x\in\dom f\,\}\subset\ol S\times T\;.
			$$
		Then the topology of $\widehat T_0$ corresponds via~\eqref{germs} to the quotient topology by the germ map $\germ:\ol S*T\to\germ(\ol S*T)\equiv\widehat T_0$, which is of course different from the sheaf topology on germs.
This point of view, replacing orthonormal frames by germs, can be readily translated to the foliated space setting, obtaining good choices of $\widehat T_0$ and $\widehat \HH_0$ under the conditions of Theorem~\ref{mt: topological Molino}. 

Now, consider triples $(x,\gamma,i)$, where $x\in U_i$, $\gamma\in\widehat T_{i,0}:=\hat\pi_0^{-1}(T_i)$ and $p_i(x)=\hat\pi_0(\gamma)$. Declare $(x,\gamma,i)$ equivalent to $(y,\delta,j)$ if $x=y$ and $\delta=\widehat{h_{ij}}(\gamma)$. Then $\widehat X_0$ is defined as the corresponding quotient space. Let $[x,\gamma,i]$ denote the equivalence class of each triple $(x,\gamma,i)$. The foliated structure $\widehat\FF_0$ on $\widehat X_0$ is determined by requiring that, for each fixed index $i$, the elements of the type $[x,\gamma,i]$ form a distinguished open set $\widehat U_{i,0}$, with distinguished submersion $\hat p_{i,0}:\widehat U_{i,0}\to\widehat T_{i,0}$ given by $\hat p_{i,0}([x,\gamma,i])=\gamma$. The projection $\hat\pi_0:\widehat X_0\to X$ is defined by $\hat\pi_0([x,\gamma,i])=x$. The properties stated in Theorem~\ref{mt: topological Molino} are satisfied with these definitions. 

It is also proved that, up to foliated homeomorphisms (respectively, local isomorphisms), $\widehat X_0$ (respectively, $G$) is independent of the choices involved. Hence $G$ can be called the {\em structural local group\/} of $\FF$.

\subsection{Growth of equicontinuous foliated spaces}

Our second main result, is the following weak topological version of the above theorems of Carri\`ere and Breuillard-Gelander.

\begin{mainthm}\label{mt: growth}
	Let $X$ be a foliated space satisfying the conditions of Theorem~\ref{mt: topological Molino}, and let $G$ be its structural local group. Then one of the following properties holds:
		\begin{itemize}
		
			\item $G$ can be approximated by nilpotent local Lie groups; or
			
			\item the holonomy covers of all leaves of $X$ have exponential growth.
		
		\end{itemize}
\end{mainthm}

(The precise definition of {\em approximation\/} of a local group is given in Definition~\ref{d: approximated}.) Like in the case of Riemannian foliations, Theorem~\ref{mt: topological Molino} reduces the proof of Theorem~\ref{mt: growth} to the case of minimal $G$-foliated spaces, where it becomes a problem about local groups. Then, since any locally compact Polish local group can be approximated by local Lie groups in the above sense, the result follows by applying the same arguments as Breuillard-Gelander.

The paper is finished indicating some examples where Theorems~\ref{mt: topological Molino} and~\ref{mt: growth} may have interesting applications, and proposing some open problems.

We thank the referee for suggestions to improve the paper.

\section{Preliminaries on equicontinuous pseudogroups}\label{s: prelim equicont pseudogr}

\subsection{Compact-open topology on partial maps with open domains}

(See \cite{Abd-AllahBrown1980}.) Given spaces $X$ and $Y$, let $C(X,Y)$ be the space of all continuous maps $X\to Y$; the notation $C_{\text{\rm c-o}}(X,Y)$ may be used to
indicate that $C(X,Y)$ is equipped with the compact-open topology.
Let $Y^*=Y\cup \{\omega\}$, where $\omega\notin Y$, endowed with the topology in which $U\subset Y^{*}$ is open if and only if  $U=Y^{*}$ or $U$ is open in $Y$. A {\em partial map\/} $X\rightarrowtail Y$ is a continuous map of a subset of $X$ to $Y$; the set of all partial maps $X\rightarrowtail Y$ is denoted by $\Par(X,Y)$. A partial map $X\rightarrowtail Y$ with open domain is called a {\em paro map\/}, and the set of all paro maps $X\rightarrowtail Y$ is denoted by $\Paro(X,Y)$. There is a bijection $\mu:\Paro(X,Y)\to C(X,Y^*)$ defined by 
	$$
		\mu(f)(x)=
			\begin{cases}
				f(x) & \text{if $ x\in\dom f$}\\
				\omega & \hbox{if  $x \notin\dom f$\;.}
			\end{cases}
	$$
The topology on $\Paro(X,Y)$ which makes $\mu:\Paro(X,Y)\to C_{\text{\rm c-o}}(X,Y^*)$ a homeomorphism is called the {\em compact-open topology\/}, and the notation $\Paro_{\text{\rm c-o}}(X,Y)$ may be used for the corresponding space. This topology has a subbasis of open sets of the form
	$$
		\NN(K,O)=\{\,h\in\Paro(X,Y)\mid K\subset\dom h,\ h(K)\subset O\,\}\;,
	$$
where $K\subset X$ is compact and $O\subset Y$ is open.

\begin{prop}\label{p: Paro_c-o(X,Y) is 2nd countable}
	If $X$ is second countable and locally compact, and $Y$ is second countable, then $\Paro_{\text{\rm c-o}}(X,Y)$ is second countable.
\end{prop}

\begin{proof}\setcounter{claim}{0}
	By hypothesis, there are countable basiss of open sets, $\VV$ of $X$ and $\WW$ of $Y$, such that $\overline V$ is compact for all $V\in\VV$. Then the sets $\NN(\overline V,W)$ ($V\in\VV$ and $W\in\WW$) form a countable subbasis of open sets of $\Paro_{\text{\rm c-o}}(X,Y)$.
\end{proof}

The following result is elementary.

\begin{prop}\label{p:C(U,Y)}
	For any open $U\subset X$, the restriction of the topology of $\Paro_{\text{\rm c-o}}(X,Y)$ to the subset $C(U,Y)$ is its usual compact-open topology.
\end{prop}

Since paro maps are not globally defined, let us make precise the definition of their composition. Given spaces $X$, $Y$ and $Z$, the 
{\em composition\/} of two paro maps, $f\in\Paro(X,Y)$ and $g\in\Paro(Y,Z)$, is the paro map $gf\in\Paro(X,Z)$ defined as the usual 
composition of the maps
	$$
		\begin{CD}
			f^{-1}(\dom g) @>f>> \dom g @>g>> Z\;.
		\end{CD}
	$$

\begin{prop}[{Abd-Allah-Brown \cite[Proposition~3]{Abd-AllahBrown1980}}]\label{p: g_*, h^*}
	The following properties hold:
	\begin{itemize}
		
		\item[{\rm(}i{\rm)}] Let $h:T\rightarrowtail X$ and $g:Y\rightarrowtail Z$ be paro maps. Then the maps
			\begin{gather*}
				g_*:\Paro_{\text{\rm c-o}}(X,Y)\to\Paro_{\text{\rm c-o}}(X,Z)\;,\quad f\mapsto gf\;,\\
				h^*:\Paro_{\text{\rm c-o}}(X,Y)\to\Paro_{\text{\rm c-o}}(T,Y)\;,\quad f\mapsto fh\;,
			\end{gather*}
		are continuous.
			
		\item[{\rm(}ii{\rm)}] Let $X'\subset X$ and $Y'\subset Y$ be subspaces such that $X'$ is open in $X$. Then the map
			$$
				\Paro_{\text{\rm c-o}}(X',Y')\to\Paro_{\text{\rm c-o}}(X,Y)\;,
			$$
		mapping a paro map $X'\rightarrowtail Y'$ to the paro map $X\rightarrowtail Y$ with the same graph, is an embedding.
		
	\end{itemize}
\end{prop}

\begin{prop}[{Abd-Allah-Brown \cite[Proposition~7]{Abd-AllahBrown1980}}]\label{p: evaluation}
	If  $Y$ is locally compact, then the evaluation partial map
		\[
  			\operatorname{ev}:\Paro_{\text{\rm c-o}}(Y,Z)\times Y\rightarrowtail Z\;,
			\quad\operatorname{ev}(f,y)=f(y)\;,
		\]
	is a paro map; in particular, its domain is open.  
\end{prop}

\begin{prop}[{Abd-Allah-Brown \cite[Proposition~9]{Abd-AllahBrown1980}}]\label{p:composition}
	If $X$ and $Y$ are locally compact, then the composition mapping 
		\[
  			\Paro_{\text{\rm c-o}}(X,Y)\times\Paro_{\text{\rm c-o}}(Y,Z)\to\Paro_{\text{\rm c-o}}(X,Y)\;,
			\quad(f,g)\mapsto gf\;,
		\]
	is continuous.  
\end{prop}

Let $\Loct(T)$ be the family of all homeomorphisms between open subsets of a space $T$, which are called {\em local transformations\/}.  For $h,h'\in\Loct(T)$, the composition $h'h\in\Loct(T)$ is the composition of maps
	$$
		\begin{CD}
			h^{-1}(\im h\cap\dom h') @>h>>\im h\cap\dom h' @>{h'}>> h'(\im h\cap\dom h')\;.
		\end{CD}
	$$
Each $h\in \Loct(T)$ can be identified with the paro map $T\rightarrowtail T$ with the same graph. This gives rise to a canonical injection $\Loct(T) \to \Paro(T,T)$ compatible with composition. The corresponding restriction of the compact-open topology of $\Paro(T,T)$ to $\Loct(T)$ is also called {\em compact-open topology\/}, and the notation $\Loct_{\text{\rm c-o}}(T)$ may be used for the corresponding space. The {\em bi-compact-open topology\/} is the smallest topology on $\Loct(X)$ so that the identity and inversion maps 
	$$
		\Loct(T)\to\Loct_{\text{\rm c-o}}(T)\;,\quad f\mapsto f^{\pm1}\;,
	$$
are continuous, and the notation $\Loct_{\text{\rm b-c-o}}(T)$ will be used for the corresponding space. The following result is elementary.

\begin{prop}[{Abd-Allah-Brown \cite[Proposition~10]{Abd-AllahBrown1980}}]\label{p: c-o top for pseudogroups}
 If $T$ is locally compact, then the composition and inversion maps,
 	\begin{gather*}
		\Loct_{\text{\rm b-c-o}}(T)\times\Loct_{\text{\rm b-c-o}}(T)\to\Loct_{\text{\rm b-c-o}}(T)\;,\quad(g,f)\mapsto gf\;,\\
		\Loct_{\text{\rm b-c-o}}(T)\to\Loct_{\text{\rm b-c-o}}(T)\;,\quad f\mapsto f^{-1}\;,
	\end{gather*}
 are continuous.
\end{prop}

\subsection{Pseudogroups}
 
\begin{defn}[Sacksteder \cite{Sacksteder1965}, Haefliger \cite{Haefliger2002}]
   A \textit{pseudogroup} on a space $T$ is a collection $\HH\subset\Loct(T)$ such that 
\begin{itemize}
 \item  the identity map of $T$ belongs to $\HH$ ($\id_T\in\HH$);
 \item if $h, h' \in \HH$, then the composite  $h'h$ is in $\HH$ ($\HH^2\subset\HH$);
 \item $h\in \HH$ implies that $h^{-1}\in \HH$ ($\HH^{-1}\subset\HH$); 
 \item if $h\in\HH$ and $U$ is open in $\dom h$, then the restriction $h:U\to h(U)$ is in $\HH$; and,
 \item  if a combination (union) of maps in $\HH$ is defined and is a homeomorphism, then it is in $\HH$. 

\end{itemize}
\end{defn}

\begin{rem}
	The following properties hold:
		\begin{itemize}
		
			\item $\id_{U}\in \HH$ for every open subset $U\subset T$.
			
			\item A local transformation $h\in\Loct(T)$ belongs to $\HH$ if and only if it locally belongs to $\HH$ (any point $x\in\dom h$ has a neighborhood $V_x\subset\dom h$ such that $h|_{V_x}\in\HH$). 
			
			\item Any intersection of pseudogroups on $T$ is a pseudogroup on $T$. 
		
		\end{itemize}
\end{rem}

\begin{ex}
 	$\Loct(T)$ is the pseudogroup that contains any other pseudogroup on $T$.
\end{ex}

\begin{defn}
 A {\em sub-pseudogroup\/} of a pseudogroup $\HH$ on $T$ is a pseudogroup on $T$ contained in $\HH$. The {\em restriction\/} of $\HH$ to an open subset $U\subset T$ is the pseudogroup 
 	$$
		\HH|_U=\{\,h\in\HH\mid\dom h\cup\im h\subset U\,\}\;.
	$$
The pseudogroup {\em generated\/} by a set $S\subset\Loct(T)$ is the intersection of all pseudogroups that contain $S$ (the smallest pseudogroup on $T$ containing $S$).
\end{defn}

\begin{defn}
  	Let $\HH$ be a pseudogroup on $T$. The {\em orbit\/} of each $x\in T$ is the set
 		$$
 			\HH(x)=\{\,h(x)\mid h\in \HH,\ x\in \dom h\,\}\;.
 		$$
	The orbits form a partition of $T$. The space of orbits, equipped with the quotient topology, 
is denoted by $T/\HH$. It is said that $\HH$ is
	\begin{itemize}
	
		\item ({\em topologically\/}) {\em transitive\/} if some orbit is dense; and
		
		\item {\em minimal\/} when all orbits are dense.
		
	\end{itemize}
\end{defn}

The following notion, less restrictive than the concept of pseudogroup, is useful to study some properties of pseudogroups.

\begin{defn}[Matsumoto \cite{Matsumoto2010}]
	A {\em pseudo$*$group\/} on a space $T$ is a family $S\subset\Loct(T)$ that is closed by the operations of composition and inversion. 
\end{defn}

\begin{rem}
	Any intersection of pseudo$*$groups on $T$ is a pseudo$*$group.
\end{rem}

\begin{defn}
	Any  pseudo$*$group contained in another pseudo$*$group is called a {\em sub-pseudo$*$group\/}. The pseudo$*$group {\em generated\/} by a set $S_0\subset\Loct(T)$ is the intersection of all pseudo$*$groups containing $S_0$ (the smallest pseudo$*$group containing $S_0$).
\end{defn}

\begin{rem}\label{r: localization}
	Let $S$ be a pseudo$*$group on $T$, and let $S_1$ be the collection of restrictions of all maps in $S$ to all open subsets of their domains. Then $S_1$ is also a pseudo$*$group on $T$, and $S$ is a sub-pseudo$*$group of $S_1$.
\end{rem}

\begin{defn}\label{d: localization}
	In Remark~\ref{r: localization}, it will be said that $S_1$ is the {\em localization\/} of $S$. If $S=S_1$, then the pseudo$*$group $S$ is called {\em local\/}.
\end{defn}

\begin{rem}
	Let $S_0\subset\Loct(T)$. The pseudo$*$group $S$ generated by $S_0$ consists of all compositions of maps in $S_0$ and their inverses. The pseudogroup $\HH$ generated by $S_0$ consists of all $h\in\Loct(T)$ that locally belong to the localization of $S$.
\end{rem}

\begin{rem}\label{r: S_1 cap S_2}
	If two local pseudo$*$groups, $S_1$ and $S_2$, generate the same pseudogroup $\HH$, then $S_1\cap S_2$ is also a local pseudo$*$group that generates $\HH$
\end{rem}

 Let $\HH$  and $\HH'$ be pseudogroups on respective spaces $T$ and $T'$.

\begin{defn}[Haefliger \cite{Haefliger1985,Haefliger1988}]\label{d: morphism}
	A {\em morphism\/}\footnote{This is usually called {\em \'etal\'e morphism\/}. We simply call it morphism because no other type of morphism will be considered here.} $\Phi\colon\HH\rightarrow \HH'$ is a maximal 
collection of homeomorphisms of open sets of $T$ to open sets of $T'$ such that
 		\begin{itemize}
 
 			\item if $\phi \in \Phi$, $h\in \HH$  and $h'\in \HH'$, then $h'\phi h\in \Phi$ ($\HH'\Phi\HH\subset\Phi$);

 		 	\item the family of the domains of maps in $\Phi$ cover $T$; and

  			\item if $\phi, \phi' \in \Phi$, then $\phi'\phi^{-1} \in \HH' $ ($\Phi\Phi^{-1}\subset\HH' $).

	 	\end{itemize}
	A morphism $\Phi$ is called an {\em equivalence\/} if the family $\Phi^{-1}=\{\,\phi^{-1}\mid\phi\in\Phi\,\}$ is also a morphism.
\end{defn}

\begin{rem}
	An equivalence $\Phi:\HH\to\HH'$ can be characterized as a maximal family of homeomorphisms of open sets of $T$ to open sets of $T'$ such that $\HH'\Phi\HH\subset\Phi$, and $\Phi^{-1}\Phi$ and $\Phi\Phi^{-1}$ generate $\HH'$ and $\HH'$, respectively.
\end{rem}

\begin{rem}
	Any morphism $\Phi:\HH\to\HH'$ induces a map between the corresponding orbit spaces, $T/\HH\to T/\HH'$. This map is a homeomorphism if $\Phi$ is an equivalence.
\end{rem}

\begin{defn}
	Let $\Phi_0$ be a family of homeomorphisms of open subsets of $T$ to open subsets of $T'$ such that
		\begin{itemize}
		
			\item the union of domains of maps in $\Phi_0$ meet all $\HH$-orbits; and
			
			\item $\Phi_0\HH\Phi_0^{-1}\subset\HH' $.
		
		\end{itemize}
	Then there is a unique morphism $\Phi:\HH\to\HH'$ containing $\Phi_0$, which is said to be {\em generated\/} by $\Phi_0$. If moreover:
		\begin{itemize}
		
			\item the union of images of maps in $\Phi_0$ meet all $\HH'$-orbits; and
			
			\item $\Phi_0^{-1}\HH\Phi_0\subset\HH $;
		
		\end{itemize}
	then $\Phi$ is an equivalence.
\end{defn}

\begin{defn}[Haefliger \cite{Haefliger2002}]
A pseudogroup $\HH$ on a locally compact space $T$ is said to be 
{\em compactly generated\/} if
\begin{itemize}

 \item there is a relatively compact open subset $U\subset T$ meeting
 all $\HH$-orbits;
  
\item  there is a finite set $S= \{h_{1},\dots,h_{n}\}\subset\HH|_U$ that generates $\HH|_U$; and
 
\item each $h_{i}$ is the restriction of some $\tilde h_i\in\HH$ with $\overline{\dom h_i}\subset\dom\tilde h_i$. 
\end{itemize}

\end{defn}

\begin{rem}
	Compact generation is very subtle \cite{Ghys1985,Meigniez1995}. Haefliger asked when compact generation implies realizability as a holonomy pseudogroup of a compact foliated space. The answer is not always affirmative \cite{Meigniez2010}.
\end{rem}

\begin{defn}[Haefliger \cite{Haefliger1985}]
A pseudogroup $\HH$ is called {\em quasi-analytic\/} if every $h\in
\HH$ is the identity around some $x\in\dom h$ whenever $h$ is the
identity on some open set whose closure contains $x$.
\end{defn}

If a pseudogroup $\HH$ on a space $T$ is quasi-analytic,  then every  $h\in\HH$ with connected domain 
is  the identity on $\dom h$ if it is the identity on some non-empty open set.  Because of this, quasi-analyticity is 
interesting when $T$ is locally connected, but local connectivity is too restrictive in our setting.  Then, instead of requiring local connectivity, the following stronger version of quasi-analyticity will be used.

\begin{defn}[\'Alvarez-Candel \cite{AlvCandel2009}]\label{d: strongly quasi-analytic}
 A pseudogroup $\HH$ on a space $T$ is said to be {\em strongly quasi-analytic\/}
if it is generated by some sub-pseudo$*$group $S\subset\HH$ such that any transformation in $S$ is the identity on its domain if it is the identity on some non-empty open subset of its domain. 
\end{defn}

\begin{rem}
 	In \cite{AlvCandel2009}, the term used for the above property is ``quasi-effective''. However the term 
``strongly quasi-analytic'' seems to be more appropriate.
\end{rem}

\begin{rem}\label{r: strongly quasi-analytic}
 	If the condition on $\HH$ to be strongly quasi-analytic is satisfied with a sub-pseudo$*$group $S$, it is also satisfied with the localization of $S$. It follows that this property is hereditary by taking sub-pseudogroups and restrictions to open subsets.
\end{rem}

\begin{defn}[Haefliger \cite{Haefliger1985}]
A pseudogroup $\HH$ on a space $T$ is said to be {\em complete\/} if, for all $x,y\in T$, there are relatively compact open neighborhoods, $U_{x}$ of $x$ and $V_{y}$ of $y$, such that, for all $h\in\HH$ and $z\in U_x\cap\dom h$ with $h(z)\in V_y$, there is some $g\in\HH$ such that $\dom g=U_{x}$ and with the same germ as $h$ at $z$.
\end{defn}

Since any pseudo$*$group $S$ on $T$ is a sub-pseudo$*$group of $\Loct(T)$, it can be endowed with the restriction of the (bi-)compact-open topology, also called the ({\em bi-\/}){\em compact-open topology\/} of $S$, and the notation $S_{\text{\rm (b-)c-o}}$ may be used for the corresponding space. In this way, according to Proposition~\ref{p: c-o top for pseudogroups}, if $T$ is locally compact, then $S_{\text{\rm b-c-o}}$ becomes a {\em topological pseudo$*$group\/} in the sense that the composition and inversion maps of $S$ are continuous. In particular, this applies to a pseudogroup $\HH$ on $T$, obtaining $\HH_{\text{\rm (b-)c-o}}$; thus $\HH_{\text{\rm b-c-o}}$ is a {\em topological pseudogroup\/} in the above sense if $T$ is locally compact.

\begin{rem}
	$S_{\text{\rm (b-)c-o}}\hookrightarrow S'_{\text{\rm (b-)c-o}}$ is continuous for pseudo$*$groups $S\subset S'$.
\end{rem}

The pseudogroups considered from now on will be assumed to act on locally compact Polish\footnote{Recall that a space is called {\em Polish\/} if it is separable and completely metrizable.} spaces; i.e., locally compact, Hausdorff and second countable spaces \cite[Theorem~5.3]{Kechris1991}.

\subsection{Groupoid of germs of a pseudogroup}\label{ss: groupoid}

\begin{defn}
 	A {\em groupoid\/} $\mathfrak G$ is a small category
 where every morphism is an isomorphism. This means that $\mathfrak G$ is a set {\rm(}of {\em morphisms\/}{\rm)} equipped with the structure defined by an 
additional set $T$ {\rm(}of {\em objects\/}{\rm)}, and the following {\em structural\/} maps:
 		\begin{itemize}
		
			\item the {\em source\/} and {\em target\/} maps $s,t:\mathfrak G\to T$;
			
			\item the {\em unit\/} map $T\to\mathfrak G$, $x\mapsto1_x$;
			
			\item the {\em operation\/} (or {\em multiplication\/}) map $\mathfrak G\times_T\mathfrak G\to\mathfrak G$, $(\delta,\gamma)\mapsto\delta\gamma$, where
				$$
  					\mathfrak G\times_T\mathfrak G=\{\,(\delta,\gamma)\in\mathfrak G\times\mathfrak G\mid t(\gamma)=s(\delta)\,\}					\subset\mathfrak G\times\mathfrak G\;;
				$$
			\item and the {\em inversion\/} map $\mathfrak G\to\mathfrak G$, $\gamma\mapsto\gamma^{-1}$; 
			
		\end{itemize}
	such that the following conditions are satisfied:
		\begin{itemize}

			\item $s(\delta\gamma)=s(\gamma)$ and $t(\delta\gamma)=t(\delta)$ for all $(\delta,\gamma)\in\mathfrak G\times_T\mathfrak G$;

			\item for all $\gamma,\delta,\epsilon\in\mathfrak G$ with $t(\gamma)=s(\delta)$ and $t(\delta)=s(\epsilon)$, we have  $\epsilon(\delta\gamma)=(\epsilon\delta)\gamma$ {\rm(}associativity{\rm)};

			\item  $1_{t(\gamma)}\gamma=\gamma1_{s(\gamma)}=\gamma$ {\rm(}units or identity elements{\rm)}; and

			\item    $s(\gamma)=t(\gamma^{-1})$, $t(\gamma)=s(\gamma^{-1})$, $\gamma^{-1}\gamma=1_{s(\gamma)}$ and $\gamma\gamma^{-1}=1_{t(\gamma)}$ for all $\gamma\in\mathfrak G$ {\rm(}inverse elements{\rm)}.

		\end{itemize}
	If moreover $\mathfrak G$ and $T$ are equipped with topologies so that all of the above structural maps are continuous, then $\mathfrak G$ is called a {\em topological groupoid\/}.
\end{defn}

\begin{rem}
	For a groupoid $\mathfrak G$, observe that $s(1_x)=t(1_x)=x$ for all $x\in T$, and therefore the source and target maps 
$s,t:\mathfrak G\to T$ are surjective, and the unit map $T\to\mathfrak G$ is injective. If moreover $\mathfrak G$ is a topological groupoid, then the unit map $T\to\mathfrak G$ is a topological embedding, and therefore the topology of $T$ is determined by the topology of $\mathfrak G$; indeed, we can consider $T$ as a 
subspace of $\mathfrak G$ if desired.
\end{rem}

\begin{defn}
 A topological groupoid is called {\em \'etal\'e\/} if the source and target maps are local
 homeomorphisms.
\end{defn}

Let $\HH$ be a pseudogroup on a space $T$. Note that the domain of the evaluation partial map $\operatorname{ev}:\HH\times T\rightarrowtail T$ is
	$$
		\HH*T=\{\,(h,x)\in \HH\times T\mid x\in\dom h\,\}\subset\HH\times T\;.
	$$
Define an equivalence relation on $\HH*T$ by setting $(h,x) \sim (h',x')$ if $x=x'$ and $h=h'$ on some neighborhood of $x$ in $\dom h\cap\dom h'$. The equivalence class of each $(h,x)\in\HH*T$ is called the {\em germ\/} of $h$ at $x$, which will be denoted by $\germ(h,x)$. The corresponding quotient set is denoted by $\mathfrak G$, and the quotient map, $\germ:\HH*T\to\mathfrak G$, is called the 
{\em germ map\/}. It is well
 known that $\mathfrak G$ is a groupoid with set of units $T$, where the source and target maps $s,t:\mathfrak G\to T$ are given by
 $s(\germ(h,x))=x$ and $t(\germ(h,x))=h(x)$, the unit map $T\to\mathfrak G$ is defined by $1_x=\germ(\id_T,x)$, the operation map $\mathfrak G\times_T\mathfrak G\to\mathfrak G$ is given by $\germ(g,h(x))\,\germ(h,x)=\germ(gh,x)$, and the inversion map is defined by $\germ(h,x)^{-1}=\germ(h^{-1},h(x))$. 
 
 For $x,y\in T$, let us use the notation $\mathfrak G_x=s^{-1}(x)$, $\mathfrak G^y=t^{-1}(y)$ and $\mathfrak G_x^y=\mathfrak G_x\cap\mathfrak G^y$; in particular, the group $\mathfrak G_x^x$ will be called the {\em germ group\/} of $\HH$ at $x$. Points in the same $\HH$-orbit have isomorphic germ groups (if $y\in\HH(x)$, an isomorphism $\mathfrak G_y^y\to\mathfrak G_x^x$ is given by conjugation with any element in $\mathfrak G_x^y$); hence the germ groups of the orbits make sense up to isomorphism. Under pseudogroup equivalences, corresponding orbits have isomorphic germ groups. The set $\mathfrak G_x$ will be called the {\em germ cover\/} of the orbit $\HH(x)$ with basis point $x$. The target map restricts to a surjective map $\mathfrak G_x\to\HH(x)$ whose fibers are bijective to $\mathfrak G_x^x$ (if $y\in\HH(x)$, a bijection $\mathfrak G_x^x\to\mathfrak G_x^y$ is given by left product with any element in $\mathfrak G_x^y$); thus $\mathfrak G_x$ is finite if and only if both $\mathfrak G_x^x$ and $\HH(x)$ 
are finite. Moreover germ covers basisd on points in the same orbit are also bijective (if $y\in\HH(x)$, a bijection $\mathfrak G_y\to\mathfrak G_x$ is given by right product with any element in $\mathfrak G_x^y$); therefore the germ covers of the orbits make sense up to bijections. 

\begin{defn}
 It is said that $\HH$ is:
 	\begin{itemize}
		
		\item {\em locally free\/} if all of its germ groups are trivial; and
		
		\item {\em strongly locally free\/} if $\HH$ is generated by a sub-pseudo$*$group $S\subset\HH$ such that, for all $h\in S$ and $x\in \dom h$, if $h(x)=x$ then $h=\id_{\dom h}$.
	
	\end{itemize}		
\end{defn}

\begin{rem}
	The condition of being (strongly) locally free is stronger than the condition of being (strongly) quasi-analytic. If $\HH$ is locally free and satisfies the condition of strong quasi-analyticity with a sub-pseudo$*$group $S\subset\HH$, generating $\HH$, then $\HH$ also satisfies the condition of being strongly locally free with $S$.
\end{rem}

\begin{rem}
	If the condition on $\HH$ to be strongly locally free is satisfied with a sub-pseudo$*$group $S$, then it is also satisfied with the localization of $S$. It follows that this property is hereditary by taking sub-pseudogroups and restrictions to open subsets
\end{rem}

The {\em sheaf topology\/} on $\mathfrak G$ has a basis consisting of the sets $\{\,\germ(h,x)\mid x\in\dom h\,\}$ for $h\in \HH$. Equipped with the sheaf topology, $\mathfrak G$ is an \'etal\'e groupoid.

Let us define another topology on $\mathfrak G$. Suppose that $\HH$ is generated by some sub-pseudo$*$group $S\subset\HH$. The set $S*T=(\HH*T)\cap(S\times T)$ is open in $S_{\text{\rm (b-)c-o}}\times T$ by Proposition~\ref{p: evaluation}.  It will be denoted by $S_{\text{\rm (b-)c-o}}*T$ when endowed with the restriction of the topology of $S_{\text{\rm (b-)c-o}}\times T$. The induced quotient topology on $\mathfrak G$, via the germ map $\germ:S_{\text{\rm (b-)c-o}}*T\to\mathfrak G$, will be also called the ({\em bi-\/}){\em compact-open topology\/}. The corresponding space will be denoted by $\mathfrak G_{\text{\rm (b-)c-o}}$, or by $\mathfrak G_{S,\text{\rm (b-)c-o}}$ if reference to $S$ is needed. It follows from Proposition~\ref{p: c-o top for pseudogroups} that $\mathfrak G_{\text{\rm b-c-o}}$ is a topological groupoid if 
$T$ is locally compact. We get a commutative diagram
	$$
		\begin{CD}
			S_{\text{(b-)c-o}}*T @>{\text{inclusion}}>> \HH_{\text{\rm (b-)c-o}}*T \\
			@V{\germ}VV @VV{\germ}V \\
			\mathfrak G_{S,\text{\rm (b-)c-o}} @>{\text{identity}}>> \mathfrak G_{\HH,\text{\rm (b-)c-o}}
		\end{CD}
	$$
where the top map is an embedding and the vertical maps are identifications. Hence the identity map 
$\mathfrak G_{S,\text{\rm (b-)c-o}}\to\mathfrak G_{\HH,\text{\rm (b-)c-o}}$ is continuous. Similarly, the identity map 
$\mathfrak G_{S,\text{\rm b-c-o}}\to\mathfrak G_{S,\text{\rm c-o}}$ is continuous.

\begin{quest}\label{q: Gamma_S,b-c-o to Gamma_S,c-o}
	When are $\mathfrak G_{S,\text{\rm (b-)c-o}}=\mathfrak G_{\HH,\text{\rm (b-)c-o}}$ and $\mathfrak G_{S,\text{\rm b-c-o}}=\mathfrak G_{S,\text{\rm c-o}}$?
\end{quest}

For the second equality, a partial answer will be given in Section~\ref{ss: coincidence of topologies}.

\subsection{Local groups and local actions}

(See~\cite{Jacoby1957}.)

\begin{defn}
A {\em local group\/} is a quintuple $G\equiv(G,e,\cdot,\,',\mathfrak{D})$ satisfying the
following conditions:
  \begin{itemize}
  
    \item[{\rm(}1{\rm)}] $(G,\mathfrak{D})$ is a topological space.
    
    \item[{\rm(}2{\rm)}] $\cdot$ is a function from a subset of $G\times G$ to $G$.
    
    \item[{\rm(}3{\rm)}] $\,'$ is a function from a subset of $G$ to $G$.
    
    \item[{\rm(}4{\rm)}] There is a subset $O$ of $G$ such that 
      \begin{itemize}
      
        \item $O$ is an open neighborhood of $e$ in $G$;
        
        \item $O\times O$ is a subset of the domain of $\cdot$;
        
        \item $O$ is a subset of the domain of $\,'$;
        
        \item for all $a,b,c\in O$, if $a\cdot b,b\cdot c$ $\in O$, then $(a\cdot b)\cdot c=(a\cdot b)\cdot c$;
        
        \item for all $a\in O$, $a'\in O$, $a\cdot e=e\cdot a=a$ and $a'\cdot a=a\cdot a'=e$;
        
        \item  the map $\cdot:O\times O\to G$ is continuous; and
        
        \item the map $\,':O\to G$ is continuous.
        
      \end{itemize}
      
    \item[{\rm(}5{\rm)}] The set $\{e\}$ is closed in $G$.
    
\end{itemize}
\end{defn}

Asserting that a local group satisfies some topological property usually means that the property is satisfied on some open neighborhood of $e$. 

A {\em local homomorphism\/} of a local group $G$ to a local group $H$ is a continuous partial map $\phi:G\rightarrowtail H$, whose domain is a neighborhood of $e$ in $G$, which is compatible in the usual sense with the identity elements, the operations and inversions. If moreover $\phi$ restricts to a homeomorphism between some neighborhoods of the identities in $G$ and $H$, then it is called a {\em local isomorphism\/}, and $G$ and $H$ are said to be {\em locally isomorphic\/}.  A local group locally isomorphic to a Lie group is called a {\em local Lie group\/}.

The collection of all sets $O$ satisfying~(4) is
denoted by $\Psi G$. This is a neighborhood basis of $e$ in $G$; all of
these neighborhoods are symmetric with respect to the inverse
operation $(3)$.  Let $\Phi(G,n)$ denote the collection of subsets $A$
of $G$ such that the product of any collection of at most $n$ elements of
$A$ is defined, and the set $A^n$ of such products is contained in
some $O\in \Psi G$. 

Let $H\subset G$. It is said that $H$ is a {\em subgroup\/} of $G$ if $H\in
\Phi(G,2)$, $e\in H$, $H'=H$ and $H^2=H$; and $H$ is a {\em sub-local group\/} of $G$
if $H$ is itself a local group with respect to the induced operations and topology.

Let $\Upsilon G$ denote the set of all pairs $(H,V)$ of subsets of $G$ so that $e\in H$, $V\in \Psi G$, $a\cdot b\in H$ for all $a,b\in V\cap H$, and $c'\in H$ for all $c\in V\cap H$. Then a subset $H\subset G$ is a sub-local group if and only if there exists some $V$ such that $(H,V)\in\Upsilon G$ \cite[Theorem 26]{Jacoby1957}. 

Let $\Pi G$ denote the family of pairs $(H,V)$ of subsets of $G$ so that $e\in H$, $V\in \Psi G\cap \Phi(G,6)$, $a\cdot b\in H$ for all $a,b\in V^6\cap H$,  $c'\in H$ for all $c\in V^6\cap H$, and $V^2\setminus H$ is open. Given $(H,V)\in \Pi G$, there is a (completely regular, Hausdorff) space $G/(V,H)$
and a continuous open surjection $T:V^2\to G/(V,H)$ such that
$T(a)=T(b)$ if and only if $a'\cdot b\in H$ (cf. \cite[Theorem 29]{Jacoby1957}).
For another pair in $\Pi G$ of the form $(H,W)$, the spaces $G/(H,V)$ and
$G/(H,W)$ are locally homeomorphic at the identity class.  Thus the
concept of coset space of $H$ is well defined in this sense, as ``a germ
of a topological space''. The notation $G/H$ may be used in this sense. It will be said that $G/H$ has a certain topological property when some $G/(H,V)$ has that property around $T(e)$.

Let $\Delta G$ be the set of pairs $(H,U)$ such that $(H,U)\in \Pi G$ and, for all $a\in H\cap U^4$ and $b\in U^2$, $b'\cdot (a \cdot b)\in H$. A subset $H\subset G$ is called a normal sub-local group of $G$ if there exists $U$ such that $(H,U)\in \Delta G$. If $(H,U)\in \Delta G$ then the quotient space $G/(H,U)$ admits the structure of a local group (see \cite[Theorem~35]{Jacoby1957} for the pertinent details) and the natural projection $T:U^2\to G/(H,U)$ is a local homomorphism. As before, another such pair $(H,V)$ produces a locally isomorphic quotient local group. 

As usual, $a\cdot b$ and $a'$ will be denoted by $ab$ and $a^{-1}$.

Local groups were first studied by Jacoby~\cite{Jacoby1957}, giving local versions of important theorems for topological groups. For instance, Jacoby characterized local Lie groups as the locally compact local groups without small subgroups\footnote{A local group is said to have no small subgroups when some neighborhood of the identity element contains no nontrivial subgroup.} \cite[Theorem~96]{Jacoby1957}. Also, any finite dimensional metrizable locally compact local group is locally isomorphic to the direct product of a Lie group and a compact zero-dimensional topological group \cite[Theorem~107]{Jacoby1957}. In particular, this property shows that any locally Euclidean local group is a local Lie group, which is an affirmative answer to a local version of Hilbert's 5th problem. However the proof of Jacoby is incorrect because he did not realize that, in local groups, associativity for three elements does not imply associativity for any finite sequence of elements \cite{{Plaut1993}}, \cite{Olver1996}. Fortunately, a completely new proof of the local Hilbert's 5th problem was given by Goldbring~\cite{Goldbring2010}. Moreover van~den~Dries and Goldbring~\cite{DriesGoldbring2010,DriesGoldbring2012} proved that any locally compact local group is locally isomorphic to a topological group, and therefore all other theorems for local groups of Jacoby hold as well because they are known for locally compact topological groups \cite{MontgomeryZippin1955}.

\begin{defn}\label{d: approximated}
	It is said that a local group $G$ can be {\em approximated\/} by a class $\CC$ of local groups if, for all $W\in \Psi G\cap\Phi(G,2)$, there is some $V\in \Psi G$ and a sequence of compact normal subgroups $F_n\subset V$ such that $V\subset W$, $F_{n+1}\subset F_n$, $\bigcap_nF_n=\{e\}$, $(F_n,V)\in \Delta G$ and $G/(F_n,V)\in\CC$.
\end{defn}

\begin{thm}[{Jacoby~\cite[Theorems~97--103]{Jacoby1957}; correction by van~den~Dries-Goldbring~\cite{DriesGoldbring2010,DriesGoldbring2012}}]\label{t: jacoby-approximated}
  Any locally compact second countable local group $G$ can be approximated by local Lie groups.
\end{thm}

\begin{defn}
A {\em local action\/} of a local group $G$ on a space $X$ is a paro map $G\times X\rightarrowtail X$, $(g,x)\mapsto gx$, defined on some open neighborhood of $\{e\}\times X$, such that $ex=x$ for all $x\in X$, and $g_1(g_2 x) = (g_1g_2)x$, provided both sides are defined.
\end{defn}

\begin{rem}\label{r:pseudogroup generated by a local action}
	The local transformations given by any local action of a local group on a space generate a pseudogroup.
\end{rem}

A local action of a local group $G$ on a space $X$ is called {\em locally transitive\/} at some point $x\in X$ if there is a neighborhood $W$ of $e$ in $G$ such that the local action is defined on $W\times\{x\}$, and $Wx:=\{\,gx\mid g\in W\,\}$ is a neighborhood of $x$ in $X$. Given another local action of $G$ on a space $Y$, a paro map $\phi:X\rightarrowtail Y$ is called {\em equivariant\/} is $\phi(gx)=g\phi(x)$ for all $x\in X$ and $g\in G$, provided both sides are defined.

\begin{ex}
Let $H$ be a sub-local group of $G$. If $(H,V)\in \Pi G$ and $T:V^2\to G/(H,V)$
is the natural projection, then the map $V\times G/(H,V)\to G/(H,V)$, $(v,T(g)) \mapsto T(vg)$, defines a local action of $G$ on $G/(H,V)$.
\end{ex}

\begin{rem}
If $G$ is a local group locally acting on $X$ and the local action is
locally transitive at $x\in X$, then there is a sub-local group $H$ of $G$ such that $(H,V)\in \Pi G$ for some $V$ and the orbit paro map $G\rightarrowtail X$, $g\mapsto gx$, induces an equivariant paro map $G/(H,V)\rightarrowtail X$, which restricts to a homeomorphism between neighborhoods of $T(e)$ and $x$.
\end{rem}

\subsection{Equicontinuous pseudogroups}\label{ss: Equi. Pseu pre}

\'Alvarez and Candel introduced the following structure to define equicontinuity for pseudogroups \cite{AlvCandel2009}. 
Let\footnote{The notation will be simplified by using, for instance, $\{T_i,d_i\}$ instead of $\{(T_i,d_i)\}$.} $\{T_i,d_i\}$ be a family of metric spaces such that
$\{T_i\}$ is a covering of a set $T$, each intersection
$T_i\cap T_j$ is open in $(T_i,d_i)$ and $(T_j,d_j)$, and, for all
$\epsilon>0$, there is some $\delta(\epsilon)>0$ so that the following
property holds: for all $i$, $j$ and $z\in T_i\cap T_j$, there is
some open neighborhood $U_{i,j,z}$ of $z$ in $T_i\cap T_j$ (with
respect to the topology induced by $d_i$ and $d_j$) such that 
$$
  d_i(x,y)<\delta(\epsilon)\Longrightarrow d_j(x,y)<\epsilon 
$$ 
for all $\epsilon>0$ and all $x,y\in U_{i,j,z}$. Such a family is called a
{\em cover of $T$ by quasi-locally equal metric spaces\/}.  Two such
families are called {\em quasi-locally equal\/} when their union
is also a cover of $T$ by quasi-locally equal metric spaces.  This is an
equivalence relation whose equivalence classes are called {\em
quasi-local metrics\/} on $T$. For each quasi-local metric ${\mathfrak
Q}$ on $T$, the pair $(T,{\mathfrak Q})$ is called a {\em quasi-local
metric space\/}. Such a $\mathfrak Q$ induces a topology\footnote{In fact, it induces a uniformity. We could even use any uniformity to define equicontinuity, but such generality will not be used here.} on $T$ so
that, for each $\{T_i,d_i\}_{i\in I}\in{\mathfrak Q}$, the family of
open balls of all metric spaces $(T_i,d_i)$ form a basis of open
sets. Any topological concept or property of $(T,{\mathfrak Q})$
refers to this underlying topology. $(T,{\mathfrak Q})$ is locally compact and Polish
if and only if it is Hausdorff, paracompact and separable \cite{AlvCandel2009}.

\begin{defn}[\'Alvarez-Candel \cite{AlvCandel2009}]\label{d:equicont}
Let $\HH$ be a pseudogroup on a
quasi-local metric space $(T,{\mathfrak Q})$. Then $\HH$ is said to be ({\em strongly\/}\footnote{This adverb, used in \cite{AlvCandel2009}, will be omitted for the sake of simplicity.})
{\em equicontinuous\/} if there exists some
$\{T_i,d_i\}_{i\in I}\in{\mathfrak Q}$ and some sub-pseudo$*$group $S\subset\HH$, generating $\HH$, such that, 
for every $\epsilon>0$, there is some $\delta(\epsilon)>0$ so that 
	\[
		d_i(x,y)<\delta(\epsilon)\Longrightarrow d_j(h(x),h(y))<\epsilon 
	\]
for all $h\in S$, $i,j\in I$ and $x,y\in T_i\cap h^{-1}(T_j\cap\im h)$.

A pseudogroup $\HH$ acting on a space $T$ will be called ({\em strongly\/})
{\em equicontinuous\/} when it is equicontinuous with respect to some
quasi-local metric inducing the topology of $T$.
\end{defn}

\begin{rem}\label{r: equicont}
	If the condition on $\HH$ to be equicontinuous is satisfied with a sub-pseudo$*$group $S$, 
then it is also satisfied with the localization of $S$. It follows that equicontinuity is hereditary by taking sub-pseudogroups and restrictions to open subsets.
\end{rem}

\begin{lem} [{\'Alvarez-Candel \cite[Lemma~8.8]{AlvCandel2009}}]\label{l: equicont is inv by equivs}
 Let $\HH$ and $\HH'$ be equivalent pseudogroups on locally compact Polish spaces. Then $\HH$ is equicontinuous if and only if $\HH'$ is equicontinuous.  
\end{lem}

\begin{prop}[{Alvarez-Candel \cite[Proposition~8.9]{AlvCandel2009}}]\label{p:equicontinuous}
Let $\HH$ be a compactly generated and equicontinuous
pseudogroup on a locally compact Polish quasi-local metric space
$(T,{\mathfrak Q})$, and let $U$ be any relatively compact open subset
of $(T,{\mathfrak Q})$ that meets every $\HH$-orbit. Suppose that
$\{T_i,d_i\}_{i\in I}\in{\mathfrak Q}$ satisfies the condition of
equicontinuity. Let $E$ be any system of compact
generation of $\HH$ on $U$, and let $\bar g$ be an extension of each $g\in E$
with $\overline{\dom g}\subset\dom\bar g$. Also, let
$\{T'_i\}_{i\in I}$ be any shrinking\footnote{Recall that a {\em shrinking\/} of  an open cover $\{U_i\}$ of a space $X$ is an open cover $\{U'_i\}$ of $X$, with the same index set, such that $\overline{U'_i}\subset U_i$ for all $i$. Similarly, if $\{U_i\}$ is a cover of a subset $A\subset X$ by open subsets of $X$, a {\em shrinking\/} of  $\{U_i\}$, as a cover of $A$ by open subsets of $X$, is a cover $\{U'_i\}$ of $A$ by open subsets of $X$, with the same index set, so that $\overline{U'_i}\subset U_i$ for all $i$.} of
$\{T_i\}_{i\in I}$. Then there is a finite family $\VV$ of open subsets
of $(T,{\mathfrak Q})$ whose union contains $U$ and such that, for any
$V\in\VV$, $x\in U\cap V$, and $h\in\HH$ with $x\in\dom h$ and $h(x)\in
U$, the domain of $\tilde h=\bar g_n\cdots\bar g_1$ contains
$V$ for any composite $h=g_n\cdots g_1$ defined around
$x$ with $g_1,\dots,g_n\in E$, and moreover $V\subset
T'_{i_0}$ and $\tilde h(V)\subset T'_{i_1}$ for some
$i_0,i_1\in I$.
\end{prop}

\begin{rem}\label{r:equicontinuous}
	The statement of Proposition~\ref{p:equicontinuous} is stronger than the completeness of $\HH|_U$. Since we can choose $U$ large enough to contain two arbitrarily given points of $T$, it follows $\HH$ is complete.
\end{rem}

\begin{prop}[{\'{A}lvarez-Candel \cite[Proposition~9.9]{AlvCandel2009}}]\label{p:A B}
Let $\HH$ be a compactly generated, equicontinuous and
strongly quasi-analytic pseudogroup on a locally
compact Polish space $T$. Suppose that the conditions of
equicontinuity and strong quasi-analyticity are satisfied with a sub-pseudo$*$group $S\subset\HH$, generating $\HH$.  Let
$A,B$ be open subsets of $T$ such that $\overline{A}$ is compact and
contained in $B$. If $x$ and $y$ are close enough points in $T$, then
$$ f(x)\in A\Rightarrow f(y)\in B $$ for all $f\in S$ whose domain
contains $x$ and $y$.
\end{prop}

\begin{thm}[{\'{A}lvarez-Candel \cite[Theorem~11.11]{AlvCandel2009}}]\label{t:minimal} 
Let $\HH$ be a compactly generated and equicontinuous pseudogroup on a locally compact Polish space $T$. If $\HH$ is transitive, then $\HH$ is minimal.
\end{thm}

Theorem~\ref{t:minimal} can be restated by saying that the orbit closures form a partition of the space. The following result states that indeed the orbit closures are orbits of a pseudogroup if strong quasi-analyticity is also assumed.

\begin{thm}[{\'{A}lvarez-Candel \cite[Theorem~12.1]{AlvCandel2009}}]\label{t:closure}
  Let $\HH$ be a strongly quasi-analytic, compactly generated and
  equicontinuous pseudogroup on a locally
  compact Polish space $T$. Let $S\subset\HH$ be a sub-pseudo$*$group 
  generating $\HH$ so that $\HH$ satisfies the conditions of equicontinuity and
  strong quasi-analyticity with $S$. Let
  $\widetilde{\HH}$ be the set of maps $h$ between open subsets of $T$
  that satisfy the following property: for every $x\in\dom h$, there
  exists a neighborhood $O_x$ of $x$ in $\dom h$ so that the
  restriction $h|_{O_x}$ is in the
  closure of $C(O_x,T)\cap S$ in $C_{\text{\rm c-o}}(O_x,T)$. Then:
\begin{itemize}
  
\item[{\rm(}i{\rm)}] $\widetilde{\HH}$ is closed by composition, combination and
  restriction to open sets;
  
\item[{\rm(}ii{\rm)}] any map in $\widetilde{\HH}$ is a homeomorphism around every
  point of its domain;
  
\item[{\rm(}iii{\rm)}] $\overline{\HH}=\widetilde{\HH}\cap\Loct(T)$ is a
  pseudogroup that contains $\HH$;

\item[{\rm(}iv{\rm)}] $\overline{\HH}$ is equicontinuous;
  
\item[{\rm(}v{\rm)}] the orbits of $\overline{\HH}$ are equal to the closures of the
  orbits of $\HH$; and
  
\item[{\rm(}vi{\rm)}] $\widetilde{\HH}$ and $\overline{\HH}$ are independent of the
  choice of $S$.

\end{itemize}
\end{thm}

\begin{rem}\label{r: closure 1}
	In Theorem~\ref{t:closure}, let $\overline{S}$ be the set of local transformations that are in the union of the closures of
$C(O,T)\cap S$ in $C_{\text{\rm c-o}}(O,T)$ with $O$ running on the open sets of $T$. According to the proof of  \cite[Theorem~12.1]{AlvCandel2009}, $\overline{S}$ is a pseudo$*$group that generates $\ol\HH$. Moreover, if $\HH$ satisfies the equicontinuity condition with $S$ and some representative $\{T_i,d_i\}$ of a quasi-local metric, then $\ol\HH$ satisfies  the equicontinuity condition with $\ol S$ and $\{T_i,d_i\}$.
\end{rem}

\begin{rem}\label{r: closure 2}
	From the proof of  \cite[Theorem~12.1]{AlvCandel2009}, it also follows that, with the notation of Remark~\ref{r: closure 1}, any $x\in\ol U$ has a neighborhood $O$ in $T$ such that the closure of
		$$
			\{\,h\in C(O,T)\cap S\mid h(O)\cap\ol U\ne\emptyset\,\}
		$$
	in $C_{\text{\rm c-o}}(O,T)$ is contained in $\Loct(T)$, and therefore in $\ol S$.
\end{rem}

\begin{ex}\label{ex:G}
Let $G$ be a locally compact Polish local group with a left invariant metric, let $\Gamma\subset G$ be a dense sub-local group, and let $\HH$ be the minimal pseudogroup 
generated by the local action of $\Gamma$ by local left translations on $G$. The local left and right translations in $G$ by each 
$g\in G$ will be denoted by $L_g$ and $R_g$. The restrictions of the local left translations $L_\gamma$ ($\gamma\in\Gamma$) to open 
subsets of their domains form a sub-pseudo$*$group $S\subset\HH$ that generates $\HH$. Obviously, $\HH$ satisfies with $S$ the condition 
of being strongly locally free, and therefore strongly quasi-analytic.  Moreover $\HH$ satisfies with $S$ the condition of being 
equicontinuous (indeed isometric) by considering any left invariant metric on $G$. 
Observe that any local right translation $R_g$ ($g\in G$) generates an equivalence $\HH\to\HH$.

Now, suppose that $\HH$ is compactly generated. Then $\overline\HH$ is generated by the local action of $G$ on itself by local left translations. The sub-pseudo$*$group $\overline S\subset\overline\HH$ consists of the restrictions of the local left translations $L_g$ ($g\in G$) to open subsets of their domains. Observe that $\overline\HH$ satisfies the condition of being strongly locally free, and therefore strongly quasi-analytic, with $\overline S$.
\end{ex}

\begin{lem}\label{l:G}
	Let $G$ and $G'$ be locally compact Polish local groups with left invariant metrics, let $\Gamma\subset G$ and $\Gamma'\subset G'$ be dense sub-local groups, and let $\HH$ and $\HH'$ be the pseudogroups generated by the local actions of $\Gamma$ and $\Gamma'$ by local left translations on $G$ and $G'$. Suppose that $\HH$ and $\HH'$ are compactly generated. Then $\HH$ and $\HH'$ are equivalent if and only if $G$ is locally isomorphic to $G'$.
\end{lem}

\begin{proof}\setcounter{claim}{0}
	Consider the notation and observations of Example~\ref{ex:G} for both $G$ and $G'$; in particular, $S\subset\HH$ and $S'\subset\HH'$ denote the sub-pseudo$*$groups of restrictions of local translations $L_{\gamma}$ and $L_{\gamma'}$ ($\gamma\in\Gamma$ and $\gamma'\in\Gamma'$) to open subsets of their domains.  Let $e$ and $e'$ denote the identity elements of $G$ and $G'$. Let $\Phi:\HH\to\HH'$ be en equivalence. Since $\HH'$ is minimal, after composing $\Phi$ with the equivalence generated by some local right translation in $G$ if necessary, we can assume that $\phi(e)=e'$ for some $\phi\in\Phi$ with $e\in\dom\phi$.
	
	Let $U$ be a relatively compact open symmetric neighborhood of $e$ in $G$ with $\overline U\subset\dom\phi$. Let $\{f_1,\dots,f_n\}$ be a symmetric system of compact generation of $\HH$ on $U$. Thus each $f_i$ has an extension $\tilde f_i\in\HH$ so that $\overline{\dom f_i}\subset\dom\tilde f_i\subset\dom\phi$. 
	
	\begin{claim}\label{cl:f_i}
		We can assume that $\tilde f_i\in S$ and $\phi\tilde f_i\phi^{-1}\in S'$ for all $i$.
	\end{claim}
	
	Each point in $\dom\tilde f_i\cap\dom\phi$ has an open neighborhood $O$ such that $O\subset\dom\tilde f_i$, $\tilde f_i|_O\in S$ and $\phi\tilde f_i\phi^{-1}|_{\phi(O)}\in S'$. Take a finite covering $\{O_{ij}\}$ ($j\in\{1,\dots,k_i\}$) of the compact set $\overline{\dom f_i}$ by sets of this type. Let $\{P_{ij}\}$ be a shrinking of $\{O_{ij}\}$, as a cover of $\overline{\dom f_i}$ by open subsets of $\dom\tilde f_i$. Then the restrictions $g_{ij}=f_i|_{P_{ij}\cap U}$ ($i\in\{1,\dots,n\}$ and $j\in\{1,\dots,k_i\}$) generate $\HH|_U$, each $\tilde g_{ij}=\tilde f_i|_{O_{ij}}$ is in $S$ and extends $g_{ij}$, $\overline{\dom g_{ij}}\subset\dom\tilde g_{ij}$, and $\phi\tilde g_{ij}\phi^{-1}\in S'$, showing Claim~\ref{cl:f_i}.
	
	According to Claim~\ref{cl:f_i}, the maps $f'_i=\phi f_i\phi^{-1}$ form a symmetric system of compact generation of $\HH'$ on $U'=\phi(U)$, which can be checked with the extensions $\tilde f'_i=\phi\tilde f_i\phi^{-1}$. Let $S_0\subset S$ and $S'_0\subset S'$ be the sub-pseudo$*$groups consisting of the restrictions of compositions of maps $f_i$ and $f'_i$ to open subsets of their domains, respectively. They generate $\HH$ and $\HH'$. It follows from Claim~\ref{cl:f_i} that $\phi f\phi^{-1}\in S'$ for all $f\in S_0$. On the other hand, by Proposition~\ref{p:equicontinuous}, there is a smaller open neighborhood of the identity, $V\subset U$, such that, for all $h\in\HH$ and $x\in V\cap\dom h$ with $h(x)\in U$, there is some $f\in S_0$ such that $\dom f=V$ and $\germ(f,x)=\germ(h,x)$. 
	
	Let $W$ be another symmetric open neighborhood of the identity such that $W^2\subset V$. Let us show that $\phi:W\to\phi(W)$ is a local isomorphism. Let $\gamma\in W\cap\Gamma$. The restriction $L_\gamma:W\to\gamma W$ is well defined and belongs to $S$. Hence there is some $f\in S_0$ so that $\dom f=V$ and $\germ(f,e)=\germ(L_\gamma,e)$. Since $f$ is also a restriction of a local left translation in $G$, it follows that $f=L_\gamma$ on $W$. So $\phi L_\gamma\phi^{-1}|_{\phi(W)}\in S'$; i.e., there is some $\gamma'\in\Gamma'$ such that $\phi L_\gamma\phi^{-1}=L_{\gamma'}$ on $\phi(W)$. In fact,
		$$
			\phi(\gamma)=\phi L_\gamma(e)=\phi L_\gamma\phi^{-1}(e')=L_{\gamma'}(e')=\gamma'\;.
		$$
	Hence, for all $\gamma,\delta\in\Gamma$,
		\begin{align*}
			\phi(\gamma\delta)&=\phi L_\gamma(\delta)=L_{\phi(\gamma)}\phi(\delta)=\phi(\gamma)\phi(\delta)\;,\\[4pt]
			\phi(\gamma)^{-1}&=L_{\phi(\gamma)}^{-1}(e')=(\phi L_\gamma\phi^{-1})^{-1}(e')\\
			&=\phi L_{\gamma^{-1}}\phi^{-1}(e')=L_{\phi(\gamma^{-1})}(e')=\phi(\gamma^{-1})\;.
		\end{align*}
	Since $\phi$, and the product and inversion maps are continuous, it follows that $\phi(gh)=\phi(g)\phi(h)$ and $\phi(g^{-1})=\phi(g)^{-1}$ for all $g,h\in W$.
\end{proof}

\begin{ex}\label{ex:G/K}
This generalizes Example~\ref{ex:G}. Let $G$ be a locally compact Polish local group with a left-invariant metric, $K\subset G$ a compact subgroup, and $\Gamma\subset G$ a dense sub-local group. Take some $V$ so that $(H,V)\in\Pi(G)$. The left invariant metric on $G$ can be assumed to be also $K$-right invariant by the compactness of $K$, and therefore it defines a metric on $G/(K,V)$. Then the canonical local action of $\Gamma$ on some neighborhood of the identity class in $G/(K,V)$ induces a transitive equicontinuous pseudogroup $\HH$ on a locally compact Polish space; in fact, this is a pseudogroup of local isometries. 

Assume that $\HH$ is compactly generated. Then $\overline\HH$ is generated by the canonical local action of $G$ on some neighborhood of the identity class in $G/(K,V)$. Moreover the sub-pseudo$*$group $\overline S\subset\overline\HH$ consists of the local translations of the local action of $G$ on $G/(K,V)$.
\end{ex}

Examples~\ref{ex:G} and~\ref{ex:G/K} are particular cases of pseudogroups induced by local actions (Remark~\ref{r:pseudogroup generated by a local action}). The following result indicates their relevance. 

\begin{thm}[{\'{A}lvarez-Candel \cite[Theorem~5.2]{AlvCandel2010}}]\label{t:G/K}
  Let $\HH$ be a transitive, compactly generated and 
  equicontinuous pseudogroup on 
  a locally compact Polish space, and suppose that
  $\ol{\HH}$ is strongly quasi-analytic. Then $\HH$ is equivalent to a
  pseudogroup of the type described in Example~\ref{ex:G/K}.
\end{thm}

\begin{rem}\label{r:G}
	From the proof of \cite[Theorems 3.3 and 5.2]{AlvCandel2010}, it also follows that, in Theorem~\ref{t:G/K}, if moreover $\ol{\HH}$ is strongly locally free, then $\HH$ is equivalent to a
  pseudogroup of the type described in Example~\ref{ex:G}.
\end{rem}

\section{Molino's theory for equicontinuous pseudogroups}\label{s: Molino pseudogrs}

\subsection{Conditions on $\HH$}\label{ss: conditions}

Let $\HH$ be a pseudogroup of local transformations
of a locally compact Polish space $T$.  Suppose that  $\HH$ is compactly generated, complete and equicontinuous, and that $\overline{\HH}$ is also strongly quasi-analytic.

Let $U$ be a relatively compact open set in $T$ that meets all the orbits of  $\HH$. The 
condition of compact generation is satisfied  with $U$.
Consider a representative $\{T_{i},d_{i}\}$ of a quasi-local metric on $T$ satisfying the condition of equicontinuity of $\HH$  with some sub-pseudo$*$group $S\subset\HH$ that generates $\HH$.
We can also suppose that the condition of strong quasi-analyticity of $\HH$ is satisfied with $S$.

\begin{rem}\label{r: delta(epsilon)} 
According to Theorem~\ref{t:closure} and Remark~\ref{r: closure 1}, there is a mapping $\epsilon\mapsto\delta(\epsilon)>0$ ($\epsilon>0$) such that 
	\[
		d_{i}(x,y)<\delta(\epsilon)\Longrightarrow d_{j}(h(x),h(y))<\epsilon 
	\]
for all indices $i$ and $j$, every $h\in \overline{S}$, and $x,y\in T_{i}\cap h^{-1}(T_{j}\cap \im h)$.
\end{rem}

\begin{rem}\label{r: ol S} 
By Remark~\ref{r: closure 2} and refining $\{T_i\}$ if necessary, we can assume that $\overline{U}$ is covered by a finite collection of the sets $T_i$, $\{T_{i_{1}},\dots ,T_{i_{r}}\}$, such that the closure of
		$$
			\{\,h\in C(T_{i_k},T)\cap S\mid h(T_{i_k})\cap\ol U\ne\emptyset\,\}
		$$
	in $C_{\text{\rm c-o}}(T_{i_k},T)$ is contained in $\ol S$ for all $k\in\{1,\dots,r\}$.
\end{rem}

\begin{rem}\label{r: tilde h} 
By Proposition~\ref{p:equicontinuous} and Remark~\ref{r: ol S}, and refining $\{T_i\}$ if necessary, we can assume that, for all $h\in \overline{\HH}$ and $x\in T_{i_{k}}\cap U\cap\dom h$ with $h(x)\in U$, there is some $\tilde{h} \in \overline{S}$ with $\dom\tilde{h}=T_{i_{k}}$ and  $\germ(h,x)=\germ(\tilde{h},x)$. 
\end{rem}

\begin{rem}\label{r: strong quasi-analyticity with ol S}
	By Remarks~\ref{r: S_1 cap S_2},~\ref{r: strongly quasi-analytic} and~\ref{r: equicont}, and refining $\{T_i\}$ if necessary, we can assume that the strong quasi-analyticity of $\overline{\HH}$ is satisfied with $\overline{S}$.
\end{rem}

\subsection{Coincidence of topologies}\label{ss: coincidence of topologies}

\begin{prop}\label{p:overline S_b-c-o = overline S_c-o}
	$\overline S_{\text{\rm b-c-o}}=\overline S_{\text{\rm c-o}}$.
\end{prop}

\begin{proof}\setcounter{claim}{0}
	(This is inspired by \cite{Arens1946}.) For each $g\in\overline S$, take any index $i$ and open sets $V,W\subset T$ so that $\overline V\subset W$ and $\overline W\subset\im g$. By Proposition~\ref{p:A B}, there is some $\epsilon(i,V,W)>0$ such that, for all $x,y\in T_i$, if $d_i(x,y)<\epsilon(i,V,W)$, then
		$$
			f(x)\in\overline V\Longrightarrow f(y)\in W
		$$
	for all $f\in\overline S$ with $x,y\in\dom f$. Let $\KK(g,i,V,W)$ be the family of compact subsets $K\subset T_i\cap\dom g$ such that
		$$
			\rK\neq\emptyset\;,\quad\diam_{d_i}(K)<\epsilon(i,V,W)\;,\quad g(K)\subset V\;,
		$$
	where $\rK$ and $\diam_{d_i}(K)$ denote the interior and $d_i$-diameter of $K$. Moreover let $\KK(g)$ denote the union of the families $\KK(g,i,V,W)$ as above. Then a subbasis $\NN(g)$ of open neighborhoods of each $g$ in $\overline S_{\text{\rm c-o}}$ is given by the sets $\NN(K,O)\cap\overline S$, where $K\in\KK(g)$ and $O$ is an open neighborhood of $g(K)$ in $T$.
	
	We have to prove the continuity of the inversion map $\overline S_{\text{\rm c-o}}\to\overline S_{\text{\rm c-o}}$, $h\mapsto h^{-1}$. Let $h\in\overline S$ and let $\NN(K,O)\in\NN(h^{-1})$ with $K\in\KK(h^{-1},i,V,W)$, and fix any point $x\in\rK$. Then
		$$
			\VV=\NN(\{h^{-1}(x)\},\rK)\cap\NN(\overline W\setminus O,T\setminus K)
		$$
	is an open neighborhood of $h$ in $\HH_{\text{\rm c-o}}$. We have $d_i(fh^{-1}(x),y)<\epsilon(i,V,W)$ for all $f\in\VV\cap\overline S$ and $y\in K$ since $fh^{-1}(x)\in\rK$ and $\diam_{d_i}(K)<\epsilon(i,V,W)$. So $f^{-1}(y)\in W$ by the definition of $\epsilon(i,V,W)$ since $f^{-1}\in\overline S$ and $h^{-1}(x)\in h^{-1}(K)\subset V$. Thus, if $f^{-1}(y)\not\in O$, we get $f^{-1}(y)\in\overline W\setminus O$, obtaining $y\in T\setminus K$, which is a contradiction. Hence $f^{-1}\in\NN(K,O)$ for all $f\in\VV\cap\overline S$.
\end{proof}

Let $\overline{\mathfrak G}$ denote the groupoid of germs of $\overline\HH$. The following direct consequence of Proposition~\ref{p:overline S_b-c-o = overline S_c-o} gives a partial answer to Question~\ref{q: Gamma_S,b-c-o to Gamma_S,c-o}.

\begin{cor}\label{c: overline Gamma_b-c-o = overline Gamma_c-o}
	$\overline{\mathfrak G}_{\overline S,\text{\rm b-c-o}}=\overline{\mathfrak G}_{\overline S,\text{\rm c-o}}$; i.e., $\overline{\mathfrak G}_{\overline S,\text{\rm c-o}}$ is a topological groupoid.
\end{cor}

\subsection{The space $\widehat{T}$}\label{ss: widehat T}

Recall that $s,t:\overline{\mathfrak G}_{\overline S,\text{\rm c-o}}\to T$ denote the source and target projections. Let $\widehat{T}=\mathfrak G_{\overline S,\text{\rm c-o}}$, where the following subsets are open:
\begin{gather*}
 \widehat{T}_{U}=s^{-1}(U)\cap t^{-1}(U)\;,\quad
 \widehat{T}_{k,l}=s^{-1}(T_{i_k,i_l})\cap t^{-1}(T_{i_k,i_l})\;,\\
 \widehat{T}_{U,k,l}=\widehat{T}_{U}\cap \widehat{T}_{k,l}\;.
 \end{gather*}
Observe that $\widehat{T}_{U}$ is an open subspace of $\widehat{T}$, and
the family of sets $\widehat{T}_{U,k,l}$ form an open covering of $\widehat{T}_{U}$.

Let $\germ(h,x)\in\widehat{T}_{U,k,l}$.  We can assume that
$h\in\overline{S}$ and $\dom h=T_{i_{k}}$ according to Remark~\ref{r: tilde h}. Since  $x\in T_{i_{k}}\cap U$ 
and  $h(x)\in T_{i_{l}}\cap U$, there are relatively compact open neighborhoods, $V$ of $x$ and $W$
of $h(x)$, such that $\overline{V}\subset T_{i_{k}}\cap U$, 
$\overline{W}\subset T_{i_{l}}\cap U$ and $h(\overline{V})\subset W$. 

By Remark~\ref{r: tilde h}, for each $f\in\overline{S}$ with $x\in\dom f$, there is some $\tilde f\in \overline{S}$ with $\dom\tilde f=T_{i_{k}}$ and $\germ(\tilde f,x)=\germ(f,x)$.

\begin{lem}\label{l:f=tilde f on V}
	$f=\tilde f$ on $V$.
\end{lem}

\begin{proof}\setcounter{claim}{0}
	The composition $f|_V\,\tilde f^{-1}$ is defined on $\tilde f(V)$,
belongs to $\overline{S}$, and is the identity on some neighborhood of $\tilde f(x)=f(x)$. 
So $f|_V\,\tilde f^{-1}$ is the identity on $\tilde f(V)$ because $\overline{\HH}$ satisfies 
the strong quasi-analyticity condition with $\overline{S}$.  Hence $f=\tilde f$ on $V$.
\end{proof}

Let
	\begin{align}
		 \overline{S}_{0} 
		 &=\{\,f\in \overline{S} \mid \overline{V} \subset \dom f,\ f(\overline{V}) \subset W\,\}\;,
 		\label{S0} \\
		 \overline{S}_{1}
		 &=\{\,f\in \overline{S} \mid \overline{V} \subset \dom f,\ f(\overline{V}) \subset \overline{W}\,\}\;,
 		\label{S1}
	\end{align}
equipped with the restriction of the compact-open topology. Notice that $\overline{S}_{0}$ is an open neighborhood of $h$ in $\overline{S}_{\text{\rm c-o}}$. Consider the compact-open topology on $C(\overline{V},\overline{W})$.

\begin{lem}\label{l: RR defines an identification} 
 	The restriction map $\RR:\overline{S}_{1}\rightarrow C(\overline{V},\overline{W})$, $\RR(f)=f|_{\overline{V}}$, defines an identification $\RR:\overline{S}_{1}\to\RR(\overline{S}_{1})$.
\end{lem}

\begin{proof}\setcounter{claim}{0}
	The continuity of $\RR$ is elementary.
	
	Let $G\subset\RR(\overline S_1)$ such that $\RR^{-1}(G)$ is open in
 $\overline{S}_{1}$. For each $g_{0} \in G$, there is some $g'_{0} \in \RR^{-1}(G)$
such that $\RR(g'_{0})=g_{0}$. Since $\RR^{-1}(G)$ is open in $\overline{S}_{1}$,
there are finite collections, $\{K_{1},\dots,K_{p}\}$ of compact subsets and $\{O_{1},\dots,O_{p}\}$ of open
subsets, such that
	\begin{multline*}
		g'_{0}\in \{\,f\in\overline{S}_{1}\mid K_{1}\cup\cdots\cup K_{p}\subset\dom f,\\
		f(K_{1})\subset O_{1},\dots,f(K_{p})\subset O_{p}\,\}\subset \RR^{-1}(G)\;.
	\end{multline*}
Then
	\begin{multline*}
		{g}_{0}\in \{\,g\in\overline{S}_{1}\mid (K_{1}\cup\dots\cup K_{p})\cap\overline{V}\subset\dom g,\\
		g(K_{1}\cap \overline{V})\subset O_{1}\cap\overline{W},\dots,g(K_{p}\cap \overline{V})\subset O_{p}\cap\overline{W}\,\}\subset G\;.
	\end{multline*}
Since $K_{1}\cap\overline{V}, \dots, K_{p}\cap \overline{V}$  are compact in $\overline{V}$, and $O_{1}\cap\overline{W}, \dots ,O_{p}\cap \overline{W}$ are open in $\overline{W}$, it follows that $g_{0}$ is in the interior of $G$ in $\RR(\overline{S}_{1})$. Hence $G$ is open in $\RR(\overline S_1)$.
 \end{proof}

\begin{lem} \label{l: RR(overline S_1) is closed} 
 $\RR(\overline{S}_{1})$ is closed in $C(\overline{V},\overline{W})$.
\end{lem}

\begin{proof}\setcounter{claim}{0}
	Observe that $C(\overline{V},\overline{W})$ is second countable because $T$ is Polish.
 Take a sequence $g_{n}$ in $\RR(\overline{S}_{1})$  converging to $g$ in $C(\overline{V}, \overline{W})$. Then it easily follows that $g_n|_V$ converges to $g|_V$ in $C(V,T)$ with the compact-open topology. Thus $g|_V\in\overline{S}$ according to Remark~\ref{r: ol S}, and let $f=\widetilde{g|_V}$. By Lemma~\ref{l:f=tilde f on V}, we have $g=f|_{\overline{V}}$. Therefore $f\in \overline{S}_{1}$ and $g=\RR(f)$. 
\end{proof}
 
\begin{cor}\label{c: RR(overline S_1) is compact} 
 	$\RR(\overline{S}_{1})$  is compact in $C(\overline{V},\overline{W})$.
\end{cor}

\begin{proof}\setcounter{claim}{0}
 	This follows by Arzela-Ascoli Theorem and Lemma~\ref{l: RR(overline S_1) is closed} because $\overline{V}$ and $\overline{W}$ are compact,  and $\RR(\overline{S}_{1})$ is equicontinuous since $\overline{\HH}$ satisfies the equicontinuity condition with $\overline{S}$ and  $\{T_{i}, d_{i}\}$.
\end{proof}
 
Let $V_{0}$ be an open subset of $T$ such that  $x\in V_{0}$ and $\overline{V_0}\subset V$. Since $\overline{V_{0}}\subset\dom f$ for all $f\in\overline{S}_{1}$, we can consider the restriction $\overline{S}_{1} \times\overline{V_{0}}\rightarrow \widehat{T}$ of the germ map.
 
\begin{lem}\label{l: gamma(overline S_1 times overline V_0) is compact} 
 $\germ( \overline{S}_{1} \times\overline{V_{0}})$ is compact in $\widehat{T}$.
\end{lem}
 
 \begin{proof}\setcounter{claim}{0}
  For each $g\in C(\overline{V},\overline{W})$ and $y\in \overline{V}$, let  $\bar{\germ}(g,y)$ denote the germ of $g$ at $y$, defining a germ map
  $$
  \bar{\germ}:C(\overline{V},\overline{W})\times\overline{V}\rightarrow\bar{\germ}(C(\overline{V},\overline{W})\times\overline{V})\;.
  $$
Since $\overline{V_{0}}\subset V$, we get that
 $\germ( \overline{S}_{1} \times\overline{V_{0}})= \bar{\germ}( \RR(\overline{S}_{1}) \times\overline{V_{0}})$ and the 
diagram
\begin{equation}\label{Diagram 1}
\begin{CD}
  \overline{S}_{1} \times\overline{V_{0}} @>{\RR \times\id}>> \RR(\overline{S}_{1}) \times\overline{V_{0}} \\
  @V{\germ}VV @VV{\bar{\germ}}V \\
  \germ( \overline{S}_{1} \times\overline{V_{0}}) @=  \bar{\germ}( \RR(\overline{S}_{1}) \times\overline{V_{0}})
\end{CD}
\end{equation}
is commutative. Then 
$$
\bar{\germ}:\RR( \overline{S}_{1}) \times\overline{V_{0}}\to\bar{\germ}( \RR(\overline{S}_{1}) \times\overline{V_{0}})
$$
is continuous because 
$$
\RR \times\id:\overline{S}_{1} \times\overline{V_{0}}\rightarrow \RR(\overline{S}_{1}) \times\overline{V_{0}}
$$ 
is an identification by 
Lemma~\ref{l: RR defines an identification} and 
 $$\germ: \overline{S}_{1} \times\overline{V_{0}}\rightarrow \germ( \overline{S}_{1} \times\overline{V_{0}})$$ is continuous. Hence 
$\germ( \overline{S}_{1} \times\overline{V_{0}})$ is compact by Corollary~\ref{c: RR(overline S_1) is compact}.
 \end{proof}

\begin{lem}\label{l: gamma(overline S_0 times V_0) is open} 
$\germ( \overline{S}_{0} \times V_{0})$ is open in $\widehat{T}$.
\end{lem}

\begin{proof}\setcounter{claim}{0}
 This holds because $\overline{S}_{0} \times V_{0}$ is open in $\overline S_{\text{\rm c-o}}*T$ and saturated by the fibers of $\germ:\overline S_{\text{\rm c-o}}*T\to\widehat T$.
\end{proof}

\begin{rem}\label{r: gamma(overline S_0 times V_0) is open} 
	Observe that the proof of Lemma~\ref{l: gamma(overline S_0 times V_0) is open} does not require $\overline{V_0}\subset V$; it holds for any open $V_0\subset V$.
\end{rem}

\begin{cor}\label{c: widehat T_U is locally compact} 
 $\widehat{T}_{U}$ is locally compact.
\end{cor}

\begin{proof}\setcounter{claim}{0}
 We have that  $\germ( \overline{S}_1 \times\overline{V_{0}})$
 is compact by Lemma~\ref{l: gamma(overline S_1 times overline V_0) is compact} and contains $\germ( \overline{S}_{0} \times V_{0})$, which is an 
 open neighborhood of $\germ(h,x)$ by Lemma~\ref{l: gamma(overline S_0 times V_0) is open}. Then the result follows because $\germ(h,x)\in \widehat{T}_{U}$ is arbitrary.
\end{proof}

\begin{lem}\label{c: bar gamma is injective} 
 $\bar{\germ}: \RR(\overline{S}_{1}) \times\overline{V_{0}}\rightarrow \widehat{T}$ is injective.
\end{lem}

\begin{proof}\setcounter{claim}{0}
 Let  $$(\RR(f_{1}),y_{1}),(\RR(f_{2}),y_{2})\in\RR(\overline{S}_{1}) \times\overline{V_{0}}$$ for 
$f_{1}, f_{2}\in \overline{S}_{1}$ with $\bar{\germ}(\RR(f_{1}),y_{1})=\bar{\germ}(\RR(f_{2}),y_{2})$.
Thus  $y_{1}=y_{2}=:y$ and  $\germ(f_{1},y_{1})=\germ(f_{2},y_{2})$; i.e., $f_{1}=f_{2}$ on some neighborhood $O$ of $y$ in $\dom f_{1}\cap\dom f_{2}$. Then $f_{1}(O)\subset\dom(f_{2}f_{1}^{-1})$ and  $f_{2}f_{1}^{-1}=\id_T$ on $f_{1}(O)$. Since $f_{2}f_{1}^{-1} \in \overline{S}$, we get  $f_{2}f_{1}^{-1}=\id_T$  on $\dom(f_{2}f_{1}^{-1})=f_{1}(\dom f_{1}\cap \dom f_{2})$ by the strong quasi-analyticity of $\overline{S}$.  Since $\overline{V}\subset\dom f_{1}\cap \dom f_{2}$, it follows that 
$f_{2}f_{1}^{-1}=\id_T$ on $f_{1}(\overline{V})$, and therefore $f_{1}=f_{2}$ on $\overline{V}$; i.e., $\RR(f_{1})=\RR(f_{2})$.
\end{proof}

Let $\hat{\pi}:=(s,t):\widehat{T}\rightarrow T\times T$, which is continuous.

\begin{cor}\label{c: hat pi:widehat T_U to U times U is proper} 
The restriction $\hat{\pi}:\widehat{T}_U\to U\times U$ is proper.
\end{cor}

\begin{proof}\setcounter{claim}{0} 
 Since $U\times U$  can be covered by sets of the form $V_{0}\times W$, for $V_0$ and $W$ as above, it is enough to prove that $\hat{\pi}^{-1}(K_1\times K_2)$ is compact for all compact sets $K_1\subset V_0$ and $K_2\subset W$.  Then, with the above notation,
\[
	\hat{\pi}^{-1}(K_1\times K_2)\subset  \germ( \overline{S}_{1} \times K_1)
	\subset \germ( \overline{S}_{1} \times\overline{V_{0}})\;, 
\]
and the result follows from Lemma~\ref{l: gamma(overline S_1 times overline V_0) is compact}.
\end{proof}

\begin{cor}\label{c: the closure of widehat T_U is compact} 
The closure of $\widehat{T}_{U}$ in  $\widehat{T}$ is compact. 
\end{cor}

\begin{proof}\setcounter{claim}{0}
Take a relatively compact open subset $U'\subset T$ containing $\overline U$. By applying Corollary~\ref{c: hat pi:widehat T_U to U times U is proper} to $U'$, it follows that $\hat\pi:\widehat T_{U'}\to U'\times U'$ is proper. Therefore $\hat\pi^{-1}(\overline U\times\overline U)$ is compact and contains the closure of $\widehat{T}_{U}$ in $\widehat T$.
\end{proof}

\begin{lem}\label{l: widehat T_U is Hausdorff} 
$\widehat{T}_{U}$ is Hausdorff.
\end{lem}
\begin{proof}\setcounter{claim}{0}
 Let  $\germ(h_{1},x_{1})\neq \germ(h_{2},x_{2})$ in $\widehat{T}_{U}$. 
 
 Suppose first that
 $x_{1}\neq x_{2}$. Since $T$ is Hausdorff, there are disjoint open subsets $V_{1}$ and $V_{2}$ 
such that  $x_{1}\in V_{1}$ and $x_{2}\in V_{2}$.  Then  
$\widehat{V}_{1}=\widehat{T}_{U}\cap s^{-1}(V_{1})$ and $\widehat{V}_{2}=\widehat{T}_{U}\cap s^{-1}(V_{2})$ are disjoint and open in $\widehat{T}_{U}$, and $\germ(h_{1}, x_{1})\in \widehat{V}_{1}$ and $\germ(h_{2}, x_{2})\in \widehat{V}_{2}$.

Now, assume that $x_{1}= x_{2}=:x$ but $h_1(x)\neq h_2(x)$. Take disjoint open subsets $W_{1},W_{2}\subset U$ such that $h_{1}(x)\in W_{1}$ and $h_{2}(x)\in W_{2}$. Then $\widehat{W}_{1}=\widehat{T}_{U}\cap t^{-1}(W_{1})$
and $\widehat{W}_{2}=\widehat{T}_{U}\cap t^{-1}(W_{2})$ are disjoint and open in $\widehat{T}_{U}$, and $\germ(h_{1},x)\in \widehat{W}_{1}$ and $\germ(h_{2},x)\in\widehat{W}_{2}$.

Finally, suppose that $x_{1}= x_{2}=:x$ and $h_1(x) = h_2(x)=:y$. Then $x\in T_{i_{k}}\cap U$ and 
$y\in T_{i_{l}}\cap U $ for some indexes $k$ and $l$. Take  open neighborhoods, $V$ of $x$ and $W$ of $y$, such that $\overline{V}\subset T_{i_{k}}\cap U$, $\overline{W}\subset T_{i_{l}}\cap U$ and $h_{1}(\overline{V})\cup h_{2}(\overline{V})\subset W$. Define $\overline{S}_{0}$ and $\overline{S}_{1}$ by using $V$ and 
$W$ like in \eqref{S0} and \eqref{S1}, and take an open subset  $V_{0}\subset T$ such that $x\in V_{0}$ and $\overline{V_{0}}\subset V$, as above. 
We can assume that $h_{1}, h_{2}\in \overline{S}_{1}$. Then 
\[
\bar{\germ}(\RR(h_{1}), x)=\germ(h_{1},x_{1})\neq \germ(h_{2},x_{2}) = \bar{\germ}(\RR(h_{2}), x)\;,
\]
and therefore $\RR(h_{1})\neq \RR(h_{2})$  in $\RR(\overline{S}_{1})$ by Lemma~\ref{c: bar gamma is injective}. Since $\RR(\overline{S}_{1})$ is Hausdorff (because it is a subspace of $C_{\text{\rm c-o}}(\overline{V},\overline{W})$), it follows that there are disjoint open subsets 
$\NN_{1},\NN_{2}\subset \RR(\overline{S}_{1})$ such that  $\RR(h_{1})\in\NN_{1}$ and $\RR(h_{2})\in\NN_{2}$. 
So  $\RR^{-1}(\NN_{1})$ and $\RR^{-1}(\NN_{2})$ are disjoint open subsets of $\overline{S}_1$ with $h_{1}\in\RR^{-1}(\NN_{1})$ and $h_{2}\in\RR^{-1}(\NN_{2})$. Hence $\MM_{1}=\RR^{-1}(\NN_{1})\cap\overline{S}_{0}$ and $\MM_{2}=\RR^{-1}(\NN_{2})\cap\overline{S}_{0}$ are disjoint and open in $\overline{S}_{0}$, and therefore they  are open in $\overline{S}$. Moreover $\MM_{1}\times V_{0}$ and  $\MM_{2}\times V_{0}$ are saturated by the fibers of $\germ:\overline{S}_{0}\times V_{0}\rightarrow \germ(\overline{S}_{0}\times V_{0})$; in fact, if $(f,z)\in\overline{S}_{0}\times V_{0}$ satisfies $\germ(f,z) =\germ(f',z)$ for some $f'\in\MM_{a}$ ($a\in\{1,2\}$), then 
\[
\bar{\germ}(\RR(f),z)=\germ(f,z)=\germ(f',z)=\bar{\germ}(\RR(f'),z)\;,
\]
giving $\RR(f)=\RR(f')\in\NN_a$ by Lemma~\ref{c: bar gamma is injective}, and therefore $f\in \RR^{-1}(\NN_{a})\cap\overline{S}_{0}=\MM_{a}$. It follows that $\germ(\MM_{1}\times V_{0})$ and $\germ(\MM_{2}\times V_{0})$ are open in 
 $\germ(\overline{S}_{0}\times V_{0})$  since $\germ:\overline{S}_{0}\times V_{0}\rightarrow \germ(\overline{S}_{0}\times V_{0})$ is an 
identification because $\overline{S}_{0}\times V_{0}$ is open in $\overline{S}_{\text{\rm c-o}} * T$
and saturated by the fibers of $\germ:\overline{S}_{\text{\rm c-o}} * T\to \widehat{T}$. Furthermore 
    \begin{align*}
      \germ(\MM_{1}\times V_{0})\cap \germ(\MM_{2}\times V_{0})&=\bar{\germ}(\NN_{1}\times V_{0})\cap \bar{\germ}(\NN_{2}\times V_{0})\\
      &=\bar{\germ}((\NN_{1}\cap \NN_{1})\times V_{0})=\emptyset
    \end{align*}
by the commutativity of the diagram~\eqref{Diagram 1}, and $ \germ(h_{1},x)\in\germ(\MM_{1}\times V_{0})$ and 
$ \germ(h_{2},x)\in\germ(\MM_{2}\times V_{0})$.
\end{proof}

\begin{cor}\label{c: bar gamma is a homeomorphism} 
$\bar{\germ}:\RR(\overline{S}_{1})\times \overline{V_{0}}\rightarrow \bar{\germ}(\RR(\overline{S}_{1})\times \overline{V_{0}})$ is a homeomorphism.
\end{cor}

\begin{lem}\label{l: widehat T_U is second countable} 
 $\widehat{T}_{U}$  is second countable.
\end{lem}

\begin{proof}\setcounter{claim}{0}
	$\widehat{T}_{U}$  can be covered by a countable collection of  open subsets  of the type $\germ(\overline{S}_{0}\times V_{0})$
as above.  But  $\germ(\overline{S}_{0}\times V_{0})$  is second countable because it is a subspace  of 
$ \germ(\overline{S}_{1}\times \overline{V_{0}})= \ol{\germ}(\RR(\overline{S}_{1})\times \overline{V_{0}})$, which is homeomorphic to $\RR(\ol{S}_{1})\times \overline{V_{0}}$ by Corollary~\ref{c: bar gamma is a homeomorphism}, and this space is second countable because it is a  subspace of the second countable space $C(\overline{V_{0}},\overline{W_{0}})\times\overline{V_{0}}$.
\end{proof}

\begin{cor}\label{c: widehat T_U is Polish} 
 $\widehat{T}_{U}$ is Polish.
\end{cor}

\begin{proof}\setcounter{claim}{0}
 This follows from Corollary~\ref{c: widehat T_U is locally compact}, Lemmas~\ref{l: widehat T_U is Hausdorff} and~\ref{l: widehat T_U is second countable}, and \cite[Theorem~5.3]{Kechris1991}.
\end{proof}

\begin{prop}\label{p: widehat T is Polish and locally compact} 
 $\widehat{T}$ is Polish and locally compact.
\end{prop}

\begin{proof}\setcounter{claim}{0}
First, let us prove that $\widehat{T}$ is Hausdorff. Take different points $\germ(g,x)$ and $\germ(g',x')$ in
$\widehat{T}$. Let $O$, $O'$, $P$ and $P'$ be relatively compact open neighborhoods of  $x$, $x'$, $g(x)$ and $g(x')$, 
respectively. Then $U_{1}=U\cup O \cup O' \cup P \cup P'$ is a relatively compact open subset of $T$ that meets all  
$\HH$-orbits.  By Lemma~\ref{l: widehat T_U is Hausdorff}, $\widehat{T}_{U_{1}}$ is a Hausdorff open subset of $\widehat{T}$ that contains 
$\germ(g,x)$ and $\germ(g',x')$. Hence $\germ(g,x)$ and $\germ(g',x')$  can be separated in $\widehat{T}_{U_{1}}$
by disjoint open neighborhoods in  $\widehat{T}_{U_{1}}$, and therefore also in  $\widehat{T}$.

Second, let us show that  $\widehat{T}$ is locally compact. For  $\germ(g,x)\in \widehat{T}$, let $O$  and $P$ be relatively compact open 
neighborhoods of $x$ and $g(x)$, respectively. Then $U_{1}=U\cup O  \cup P $ is a relatively compact open set of $T$ that meets 
all $\HH$-orbits. By Corollary~\ref{c: widehat T_U is locally compact}, it follows that $\widehat{T}_{U_{1}}$ is a locally compact open neighborhood of $\germ(g,x)$ in $\widehat{T}$. Hence $\germ(g,x)$ has a compact neighborhood in $\widehat{T}_{U_{1}}$, and therefore also in $\widehat{T}$.

Finally, let us show that $\widehat{T}$ is second countable. Since $T$ is second countable (it is Polish) and 
locally compact, it can be covered by countably many relatively compact open subsets $O_{n}\subset T$. Then each $U_{n,m}=O_{n}\cup O_{m} \cup U$ is a relativelly compact open set of $T$ that meets all $\HH$-orbits. Hence, by Lemma~\ref{l: widehat T_U is second countable}, the sets $\widehat{T}_{U_{n,m}}$ are second countable and open in $\widehat{T}$. Moreover these sets form a countable cover of $\widehat{T}$ because, for any $\germ(g,x)\in \widehat{T}$, we have $x\in O_{n}$ and $g(x)\in O_{m}$ for some $n$ and $m$, obtaining $\germ(g,x)\in\widehat{T}_{U_{n,m}}$. So $\widehat{T}$ is second countable.

Now the result follows by \cite[Theorem~5.3]{Kechris1991}.
\end{proof}

\begin{prop}\label{p: hat pi:widehat T to T is proper} 
 $\hat\pi:\widehat{T}\to T\times T$ is proper.
\end{prop}

\begin{proof}\setcounter{claim}{0}
	Take any compact $K\subset T\times T$, and any relatively compact open $U'\subset T$ meeting all $\HH$-orbits so that $K\subset U'\times U'$. By applying Corollary~\ref{c: hat pi:widehat T_U to U times U is proper} to $U'$, we get that $\hat\pi^{-1}(K)$ is compact in $\widehat T_{U'}$, and therefore in $\widehat T$.
\end{proof}

\subsection{The space $\widehat{T}_{0}$}\label{ss: widehat T_0}

From now on, assume that $\HH$ is minimal, and  therefore 
$\overline{\HH}$ has only one orbit, the whole of $T$. Fix a point $x_{0}\in U$, and let\footnote{The definition $\widehat{T}_{0}=s^{-1}(x_{0})$ would be valid too, of course, but it seems that the proofs in Sections~\ref{ss: widehat T_0} and~\ref{ss: widehat HH_0} have a simpler notation with the choice $\widehat{T}_{0}=t^{-1}(x_{0})$.}
$$
	\widehat{T}_{0}=t^{-1}(x_{0})=\{\,\germ(g,x)\in\widehat T\mid g(x)=x_{0}\,\}\;,\quad
	\widehat{T}_{0,U}=\widehat{T}_{0}\cap \widehat{T}_{U}\;.
$$
Observe that $\widehat{T}_{0}$ is closed in $\widehat{T}$ and $\widehat{T}_{0,U}$ is open in $\widehat{T}_{0}$. Moreover $\hat{\pi}(\widehat{T}_{0})=T\times\{x_0\}\equiv T$ and $\hat{\pi}(\widehat{T}_{0,U})=U\times\{x_0\}\equiv U$  because $T$ is the unique $\overline{\HH}$-orbit; indeed, $\hat{\pi}(\germ(h,x))=x$ for each $x\in T$ and any $h\in\overline{S}$ with $x\in\dom h$ and $h(x)=x_0$.  Let $\hat{\pi}_{0}:=s:\widehat{T}_{0}\to T$, which is continuous and surjective.

The following two corollaries are direct consequences of Proposition~\ref{p: widehat T is Polish and locally compact} (see \cite[Theorem~3.11]{Kechris1991}) and Corollary~\ref{c: the closure of widehat T_U is compact}.

\begin{cor}\label{c: widehat T_0 is Polish and locally compact}
	$\widehat{T}_{0}$ is Polish and locally compact.
\end{cor}

\begin{cor}\label{c: the closure of widehat T_0,U is compact}
	The closure of $\widehat{T}_{0,U}$ in $\widehat{T}_{0}$ is compact.
\end{cor}

The following corollary is a direct consequence of Proposition~\ref{p: hat pi:widehat T to T is proper} because $\hat{\pi}_{0}:\widehat{T}_{0}\to T$ can be identified with the restriction $\hat\pi:\widehat{T}_{0}\rightarrow T\times\{x_0\}\equiv T$.

\begin{cor}\label{c: hat pi_0: widehat T_0 to T is proper} 
$\hat{\pi}_{0}:\widehat{T}_0\rightarrow T$ is proper.
\end{cor}

\begin{prop}\label{p: the fibers of hat pi_0: widehat T_0 to T}
 The fibers of $\hat{\pi}_{0}:\widehat{T}_{0}\to T$ are homeomorphic to each other.
\end{prop}

\begin{proof}\setcounter{claim}{0}
 For each $x\in T$, there is some $f\in \overline{S}$ with $f(x)=x_{0}$. Then the mapping $\germ(g,x)\mapsto\germ(gf^{-1},x_{0})$ defines a homeomorphism $\hat{\pi}^{-1}_{0}(x)\rightarrow \hat{\pi}^{-1}_{0}(x_{0})$ whose inverse is given by $\germ(g_{0},x_{0})\mapsto\germ(g_{0}f,x)$.
\end{proof}

\begin{quest}
	When is $\hat{\pi}_{0}$ a fiber bundle?
\end{quest}

\subsection{The pseudogroup $\widehat{\HH}_{0}$}\label{ss: widehat HH_0}

For $h\in S$, define 
$$
\hat{h}:\hat{\pi}^{-1}_{0}(\dom h)\rightarrow\hat{\pi}^{-1}_{0}(\im h)\;,\quad\hat{h}(\germ(g,x))=\germ(gh^{-1},h(x))\;,
$$ 
for $g\in S$ and $x\in\dom g\cap\dom h$ with $g(x)=x_{0}$. The following two results are elementary.

\begin{lem}\label{l: hat pi_0(dom hat h)=dom h} 
 For any $h\in S$, we have $\hat{\pi}_{0}(\dom\hat{h})=\dom h$ and $\hat{\pi}_{0}(\im\hat{h})=\im h$,
and the following diagram is commutative:
$$
  \begin{CD}
    \dom\hat{h} @>{\hat{h}}>> \im\hat{h} \\
    @V{\hat{\pi}_{0}}VV  @VV{\hat{\pi}_{0}}V \\
    \dom h @>h>> \im h
  \end{CD}
$$
\end{lem}

\begin{lem}\label{l: widehat id_O}
If $O\subset T$ is open with $\id_{O}\in S$, then $\widehat{\id_{O}}=\id_{\hat{\pi}_{0}^{-1}(O)}$. 
\end{lem}

\begin{lem}\label{l: widehat h'h}
 For $h,h'\in S$, we have $\widehat{h'h}=\widehat{h'}\hat{h}$.
\end{lem}

\begin{proof}\setcounter{claim}{0}
  By Lemma~\ref{l: hat pi_0(dom hat h)=dom h}, we have
\begin{multline*}
         \dom(\widehat{h'}\hat{h})=\hat{h}^{-1}(\dom\widehat{h'}\cap\im\hat{h}) 
        = \hat{h}^{-1}(\hat{\pi}_{0}^{-1}(\dom h'\cap\im h))\\
        = \hat{\pi}_{0}^{-1}(h^{-1}(\dom h'\cap\im h)) = \hat{\pi}_{0}^{-1}(\dom(h'h)) = \dom\widehat{h'h}\;.  
    \end{multline*}
Now let $\germ(g,x)\in\dom(\widehat{h'}\hat{h})= \dom\widehat{h'h}$; thus $g\in \overline{S}$, $x\in\dom g\cap \dom h$, $h(x)\in\dom h'$ and $g(x)=x_{0}$. Then 
\begin{multline*}
         \widehat{h'h}(\germ(g,x))=\germ(g(h'h)^{-1},h'h(x)) 
        = \germ(gh^{-1}(h')^{-1},h'h(x)) \\
        = \widehat{h'}(\germ(gh^{-1},h(x))) 
        = \widehat{h'}\hat{h}(\germ(g,x))\;.\qed
   \end{multline*}
 \renewcommand{\qed}{}
\end{proof}

The following is a direct consequence of Lemmas~\ref{l: widehat id_O} and~\ref{l: widehat h'h}.

\begin{cor}\label{c: widehat h^-1} 
 For $h\in S$, the map $\hat{h}$ is bijective with $\hat{h}^{-1}=\widehat{h^{-1}}$.
\end{cor}

\begin{lem}\label{l: hat h is a homeomorphism} 
$\hat{h}$ is  a homeomorphism for all $h\in S$.
\end{lem}

\begin{proof}\setcounter{claim}{0}
By Corollary~\ref{c: widehat h^-1}, it is enough to prove that $\hat{h}$ is continuous, which holds 
because it can be expressed as the composition of the following continuous maps:
\begin{align*}
         \hat{\pi}_{0}^{-1}(\dom h) &\begin{CD}  \text{} @>{(\id, \const, h\hat{\pi}_0)}>> \hat{\pi}_{0}^{-1}(\dom h)\times \{h^{-1}\} \times \im h \end{CD} \\
        &\begin{CD} \text{} @>{\id\times \germ}>> \hat{\pi}_{0}^{-1}(\dom h)\times \germ(\{h^{-1}\} \times\im h)\end{CD} \\
        &\begin{CD} \text{} @>{\text{product}}>>  \hat{\pi}_{0}^{-1}(\im h)\;, \end{CD}
\end{align*}
as can be checked on elements:
\begin{align*}
         \germ(g,x) &\mapsto (\germ(g,x), h^{-1}, h(x)) \\
         &\mapsto (\germ(g,x),\germ(h^{-1},h(x))) \\
         &\mapsto\germ(gh^{-1},h(x))=\hat{h} (\germ(g,x))\;.\qed
\end{align*}
\renewcommand{\qed}{}
\end{proof}

Set $\widehat{S}_{0}=\{\,\hat{h}\mid h\in S\,\}$, and let $\widehat{\HH}_0$ be the pseudogroup on $\widehat{T}_0$ generated by $\widehat{S_{0}}$. Lemmas~\ref{l: widehat h'h} and~\ref{l: hat h is a homeomorphism},  and Corollary~\ref{c: widehat h^-1} give the following.

\begin{cor}\label{c: widehat S_0 is closed by the operations} 
 $\widehat{S_{0}}$ is a pseudo$*$group on $\widehat T_0$.
\end{cor}

\begin{lem}\label{l: widehat T_0,U meets all orbits} 
$\widehat{T}_{0,U}$ meets all orbits of  $\widehat{\HH}_{0}$.
\end{lem}

\begin{proof}\setcounter{claim}{0}
 Let $\germ(g,x)\in \widehat{T}_{0}$  with $g\in \overline{S}$; thus  $x\in\dom g$ and $g(x)=x_{0}$. 
Since $U$  meets all orbits of $\HH$, there is some $h\in S$ such that  $x\in\dom h$ and $h(x)\in U$. 
Then  $\germ(g,x) \in\dom\hat{h}$ and $\hat{h}(\germ(g,x))=\germ(gh^{-1},h(x))$ satisfies
$$
\hat\pi_0(\hat{h}(\germ(g,x)))=\hat\pi_0(\germ(gh^{-1},h(x)))=h(x)\in U\;.
$$
Hence  $\hat{h}(\germ(g,x))\in \widehat{T}_{0,U}$ as desired.
\end{proof}

\begin{lem}\label{l: h to hat h is a homeomorphism} 
 The map $S_{\text{\rm c-o}}\rightarrow \widehat{S}_{0,\text{\rm c-o}}$, $h\mapsto\hat{h}$, is a homeomorphism.
\end{lem}

\begin{proof}\setcounter{claim}{0}
If $\widehat{h_{1}}=\widehat{h_{2}}$ for some $h_{1},h_{2}\in S$, then $h_{1}=h_{2}$ by Lemma~\ref{l: hat pi_0(dom hat h)=dom h}. So the 
stated map is  injective, and therefore it is bijective by the definition of $\widehat{S}_{0}$.

Take a subbasic open set of $S_{\text{\rm c-o}}$, which is of the form 
$S\cap \NN(K,O)$ for some compact $K$ and open $O$ in $T$. The set  $\hat{\pi}^{-1}_{0}(K)$ is compact by Corollary~\ref{c: hat pi_0: widehat T_0 to T is proper}, and $\hat{\pi}^{-1}_{0}(O)$ is open. Then the map of the statement is open because 
	\[
		\{\,\hat{h}\mid h\in \NN(K,O)\cap S\,\}=\widehat{\NN}(\hat{\pi}^{-1}_{0}(K),\hat{\pi}^{-1}_{0}(O))\cap \widehat{S}_{0}
	\] 
by Lemma~\ref{l: hat pi_0(dom hat h)=dom h}, which is open in $\widehat{S}_{0,\text{\rm c-o}}$.

To prove its continuity, let us first show that its restriction to $S_U=S\cap \HH|_{U}$ is continuous. Fix $h_{0}\in S_U$, and take relatively compact open subsets  
	$$
		V,V_{0},W,V',V'_{0},W'\subset U\;,
	$$
and indices $k$ and $k'$ such that
	\begin{gather}
         		\overline{V_{0}} \subset V\;,\quad  \overline{V} \subset T_{i_{k}}\cap\dom h_{0}\;, 
		\label{01}\\
         		\overline{V'_{0}} \subset  V'\;,\quad \overline{V'} \subset T_{i_{k'}}\cap\im h_{0}, 
		\label{02} \\
         		\overline{W} \subset  W'\;,\quad \overline{W'} \subset T_{i_{k_{0}}}\;, \label{03} \\
                   h_0^{-1}(\overline{V'}) \subset V\;, \label{O4}\\
                   h_{0}(\overline{V_{0}}) \subset V'\;. \label{05}
    	\end{gather}
Let $\overline{S}_0$ and $\overline{S}_{1}$ (respectively,  $\overline{S}'_0$ and $\overline{S}'_1$) be defined  like in \eqref{S0} and \eqref{S1}, by using $V$ and $W$  (respectively, $V'$ and $W'$). Then $\widehat{K}=\germ(\overline{S}_{1}\times\overline{V_{0}})$ is compact 
in $\widehat{T}$ by Lemma~\ref{l: gamma(overline S_1 times overline V_0) is compact}, and  $\widehat{O}=\germ(\overline{S}'_{0}\times V')$ is open 
in $\widehat{T}$ by Lemma~\ref{l: gamma(overline S_0 times V_0) is open} and Remark~\ref{r: gamma(overline S_0 times V_0) is open}. Thus $\widehat{K}_{0}=\widehat{K}\cap\widehat{T}_{0}$ is compact and $\widehat{O}_{0}=\widehat{O}\cap\widehat{T}_{0}$ is open in $\widehat{T}_{0}$. So  $\widehat{\NN}(\widehat{K}_{0},\widehat{O}_{0})\cap \widehat{S}_{0}$ is a subbasic open set of $\widehat{S}_{0,\text{\rm c-o}}$.

\begin{claim}\label{cl: hat h_0} 
 $\widehat{h_0}\in \widehat{\NN}(\widehat{K}_{0},\widehat{O}_{0})$. 
\end{claim}
 
Let $\germ(g,x)\in \widehat{K}_{0}$; thus $g\in\overline{S}_{1}$, $x\in\overline{V}_{0}\cap\dom g$ and $g(x)=x_{0}$. The condition $g\in\overline{S}_{1}$ means that $g\in\overline{S}$, $\overline{V}\subset\dom g$  and $g(\overline{V})\subset \overline{W}$. By~\eqref{02}--\eqref{O4}, it follows  that $\overline{V'}\subset\dom gh^{-1}_{0}$ and 
	$$
		gh^{-1}_{0}(\overline{V'})\subset g(\overline{V})\subset \overline{W}\subset W'\;.
	$$
Hence $gh^{-1}_{0}\in \overline{S}'_0$, obtaining that 
\[\widehat{h_{0}}(\germ(g,x))=\germ(gh^{-1}_{0},h_{0}(x))\in \widehat{O}\;,\]
which completes the proof of Claim~\ref{cl: hat h_0}.

\begin{claim}\label{local subbasis}
The sets $\widehat{\NN}(\widehat{K}_{0},\widehat{O}_{0})\cap\widehat S_0$, constructed as above, form a local subbasis of $\widehat{S}_{0,\text{\rm c-o}}$ at $\widehat{h_{0}}$.
\end{claim}

This assertion follows by Claim~\ref{cl: hat h_0} and because the sets of the type $\widehat{O}_{0}$ form a basis of the topology of $\im\widehat{h_{0}}$, and any compact subset of $\dom\widehat{h_{0}}$ is contained in a finite union of sets of the type of $\widehat{K}_{0}$.

The sets 
$$
\NN=\NN(\overline{V_{0}},V')\cap\NN(\overline{V'},V)^{-1}\cap S_U
$$
are open neighborhoods of $h_{0}$ by~\eqref{O4},~\eqref{05}, and Propositions~\ref{p: c-o top for pseudogroups} and ~\ref{p:overline S_b-c-o = overline S_c-o}.

\begin{claim}\label{cl: hat h} 
$\hat{h}\in \widehat{\NN}(\widehat{K}_{0},\widehat{O}_{0})$ for all $h\in\NN$.
\end{claim}

Given $h\in\NN$, we have $\overline{V'}\subset\im h$  and $h^{-1}(\overline{V'})\subset V$.
Let $\germ(g,x)\in\widehat{K}_0$; thus $x\in\overline{V_0}\cap\dom g$, $g(x)=x_0$, and we can assume that $g\in \overline{S}_1$, which means that $g\in\overline{S}$, $\overline{V}\subset \dom g$ and $g(\overline{V})\subset\overline{W}$. Then $\overline{V'}\subset\dom(gh^{-1})$, $gh^{-1}(\overline{V'})\subset \overline{W}\subset W'$ and
$h(x)\in h(\overline{V_{0}})\subset V'$.
Therefore
\[\hat{h}(\germ(g,x))=\germ(gh^{-1},h(x))\in \germ(\overline{S}'_0\times V')\cap \widehat{T}_{0}=\widehat{O}_{0},\]
proving Claim~\ref{cl: hat h}.

Claims~\ref{local subbasis} and~\ref{cl: hat h} show that the map $S_{U,\text{\rm c-o}}\to\widehat{S}_{0,\text{\rm c-o}}$, $h\mapsto \hat{h}$, is continuous at $h_0$. Now, let us prove that the whole map $S_{\text{\rm c-o}}\to\widehat{S}_{0,\text{\rm c-o}}$, $h\mapsto\hat h$, is continuous. Since the sets $\NN(\widehat{K},\widehat{O})\cap\widehat{S}_0$, for small enough compact subsets $\widehat{K}\subset\widehat{T}_0$ and small enough open subsets $\widehat{O}\subset\widehat{T}_0$, form a subbasis of $\widehat{S}_{0,\text{\rm c-o}}$, it is enough to prove that the inverse image of these subbasic sets are open in $S_{\text{\rm c-o}}$. We can assume that $\widehat{K},\widehat{O}\subset\hat\pi_0^{-1}(U')$ for some relatively compact open subset $U'\subset T$ that meets all $\HH$-orbits. Consider the inclusion map $\iota:U'\hookrightarrow T$, and the paro map $\phi:T\rightarrowtail U'$ with $\dom\phi=U'$, where it is the identity map. According to Proposition~\ref{p: g_*, h^*}, we get a continuous map $\phi_*\iota^*:\Paro_{\text{\rm c-o}}(T,T)\to\Paro_{\text{\rm c-o}}(U',U')$, which restricts to a continuous map $\phi_*\iota^*:S_{\text{\rm c-o}}\to S_{U',\text{\rm c-o}}$. Observe that $\phi_*\iota^*(h)$ is the restriction $h:U'\cap h^{-1}(U')\to h(U')\cap U'$ for each $h\in S$. Hence, since $\widehat{K},\widehat{O}\subset\hat\pi_0^{-1}(U')$, it follows from Lemma~\ref{l: hat pi_0(dom hat h)=dom h} that $\NN(\widehat{K},\widehat{O})\cap\widehat{S}_0$ has the same inverse image by the map $S_{\text{\rm c-o}}\to\widehat{S}_{0,\text{\rm c-o}}$, $h\mapsto\hat h$, and by the composition
	$$
		\begin{CD}
			S_{\text{\rm c-o}} @>{\phi_*\iota^*}>> S_{U',\text{\rm c-o}} @>>> \widehat{S}_{0,\text{\rm c-o}}\;,
		\end{CD}
	$$
where the second map is given by $h\mapsto\hat h$. This composition is continuous by the above case applied to $U'$, and therefore the inverse image of $\NN(\widehat{K},\widehat{O})\cap\widehat{S}_0$ by  $S_{\text{\rm c-o}}\to\widehat{S}_{0,\text{\rm c-o}}$, $h\mapsto\hat h$, is open in $S_{\text{\rm c-o}}$.
\end{proof}

Since the compact generation of $\HH$ is satisfied with the relatively compact open set $U$, there is a symmetric finite set $\{f_{1},\dots,f_{m}\}$ generating $\HH|_U$, which can be chosen in $S$, such that each $f_{a}$ has an extension 
$\tilde{f}_{a}$ with $\overline{\dom f_{a}}\subset \dom\tilde{f}_{a}$. We can also assume that $\tilde{f}_{a}\in S$.
Let  $\widehat{\HH}_{0,U}=\widehat{\HH}|_{\widehat{T}_{0,U}}$.  Obviously, each $\widehat{\tilde{f}_a}$ is an extension of $\widehat{f_a}$. Moreover
  	$$
        		\overline{\dom\widehat{f_a}}= \overline{\hat\pi_0^{-1}(\dom f_a)}
        		\subset\hat\pi_0^{-1}(\overline{\dom f_a})
        		\subset\hat\pi_0^{-1}(\dom\tilde{f}_a) 
        		= \dom\widehat{\tilde{f}_a}\;.  
    	$$

\begin{lem}\label{l:widehat f_a}
The maps $\widehat{f_{a}}$ {\rm(}$a\in\{1,\dots, m\}${\rm)} generate $\widehat{\HH}_{0,U}$.
\end{lem}

\begin{proof}\setcounter{claim}{0}
 $\widehat{\HH}_{0,U}$ is generated by the maps of the form $\hat{h}$ with $h\in S_U$, and any such $\hat{h}$ can be written as a composition of maps $\widehat{f_a}$ around any $\germ(g,x)\in\dom\hat{h}=\hat\pi_0^{-1}(\dom h)$ by Lemma~\ref{l: widehat h'h}.
\end{proof}

\begin{cor}\label{c: widehat HH_0 is compactly generated} 
$\widehat{\HH}_{0}$ is compactly generated. 
\end{cor}

\begin{proof}\setcounter{claim}{0}
 We saw that $\widehat{T}_{0,U}$ is relatively compact in $\widehat{T}_{0}$ (Corollary~\ref{c: the closure of widehat T_0,U is compact}) and meets all $\widehat{\HH}_0$-orbits (Lemma~\ref{l: widehat T_0,U meets all orbits}), the maps $\widehat{f_a}$ generate $\widehat{\HH}_{0,U}$ (Lemma~\ref{l:widehat f_a}), and each $\widehat{\tilde f_a}$ is an extension of each $\widehat{f_a}$ with $\overline{\dom\widehat{f_a}}\subset\dom\widehat{\tilde f_a}$.
\end{proof}

Recall that the sets $T_{i_{k}}$ form a finite covering of $\overline{U}$ by open sets of $T$. Fix some index $k_{0}$ such that $x_{0}\in T_{i_{k_0}}$. Let $\{W_{k}\}$ be a shrinking of  $\{T_{i_{k}}\}$ as a cover of $\overline{U}$ by open subsets of $T$; i.e., $\{W_{k}\}$ is a cover of $\overline{U}$ by open subsets of $T$ and  $\overline{W_k}\subset T_{i_{k}}$ for all $k$. By applying Proposition~\ref{p:A B} several times, we get finite covers, $\{V_a\}$ and $\{V'_u\}$, of $\overline{U}$ by open subsets of $T$, and shrinkings,  $\{W_{0,k}\}$ of $\{W_k\}$ and $\{V_{0,a}\}$ of $\{V_a\}$, as covers of $\overline{U}$ by open subsets of $T$, such that the following properties hold:
\begin{itemize}
 
 \item For all $h\in \HH$ and $x\in\dom h\cap U\cap V_{a}\cap W_{0,k}$ with $h(x)\in U\cap W_{0,l}$, there is some $\tilde{h}\in S$ such that
$$
\overline{V_a}\subset\dom\tilde{h}\cap W_{k}\;,\quad\germ(\tilde{h}, x)=\germ (h, x)\;,\quad
\tilde{h}(\overline{V_a})\subset W_{l}\;.
$$

 \item For all $h\in \HH$ and $x\in\dom h\cap U\cap V'_{u}\cap V_{0,a}$ with $h(x)\in U\cap V_{0,b}$, there is some $\tilde{h}\in S$ such that
$$
\overline{V'_{u}}\subset\dom\tilde{h}\cap V_{a}\;,\quad\germ(\tilde{h}, x)=\germ (h, x)\;,\quad
\tilde{h}(\overline{V'_u})\subset V_{b}\;.
$$

\end{itemize}

By the definition of $\overline{\HH}$ and $\overline{S}$, it follows that these properties also hold for all $h\in\overline{\HH}$  with 
$\tilde{h}\in\overline{S}$.
Let $\{V'_{0,u}\}$ be a shrinking of $\{V'_{u}\}$ as a cover of $\overline{U}$ by open subsets of $T$. We have 
 $x_0\in W_{0,k_{0}}\cap V_{0,a_{0}}\cap V'_{0,u_{0}}$ for some indices $k_0$, $a_0$ and $u_0$. For each $a$, let $\overline{S}_{0,a}$, $\overline{S}_{1,a}\subset \overline{S}$ be defined like $\overline{S}_{0}$ and $\overline{S}_1$ in~\eqref{S0} and~\eqref{S1} by using $V_{a}$ and $W_{k_{0}}$ instead of $V$ and $W$. Take an index $u$  such that  $\overline{V'_{u}}\subset V_{a}$. The sets  $V_{0,a}\cap V'_{0,u}$, defined in this way, form a cover of $\overline{U}$, obtaining that the sets $\widehat{T}_{a,u}=\germ(\overline{S}_{0,a}\times(V_{0,a}\cap V_{0,u}'))$ form a cover of $\overline{\widehat{T}_{U}}$ by open subsets of $\widehat{T}$ (Lemma~\ref{l: gamma(overline S_0 times V_0) is open}), and therefore the sets $\widehat{T}_{0,a,u}=\widehat{T}_{a,u}\cap \widehat{T}_0$ form a cover of $\overline{\widehat{T}_{0,U}}$ by open subsets of $\widehat{T}_0$. Let $\widehat{T}_{0,U,a,u}=\widehat{T}_{0,U}\cap\widehat{T}_{a,u}$. Like in Section~\ref{ss: widehat T}, let  $\bar{\germ}$ denote the germ map defined on $ C(\overline{V_{a}},\overline{W_{k_{0}}})\times \overline{V_{a}}$, and let $\RR_a:\overline{S}_{1,a}\rightarrow C(\overline{V_a},\overline{W_{k_{0}}})$ be the restriction map $f\mapsto f|_{\overline{V_a}}$.  Then 
	\begin{equation}\label{bar gamma}
 		 \bar{\germ}:\RR_{a}(\overline{S}_{1,a})\times \overline{V_{0,a}\cap V_{0,u}'} 
		\rightarrow\bar{\germ}(\RR_{a}(\overline{S}_{1,a})\times \overline{V_{0,a}\cap V_{0,u}'}) 
	\end{equation} 
is a homeomorphism by Corollary~\ref{c: bar gamma is a homeomorphism}. Since $\overline{V_a}$ is compact, the compact-open topology on $\RR_{a}(\overline{S}_{1,a})$ equals the topology induced by the supremum  metric $d_{a}$ on $C(\overline{V_a},\overline{W_{k_{0}}})$, defined
with the metric  $d_{i_{k_{0}}}$ on $T_{i_{k_{0}}}$. Take some index $k$ such that $\overline{V_{a}}\subset W_{k}$. Then the topology of 
$\RR_{a}(\overline{S}_{1,a})\times \overline{V_{0,a}\cap V_{0,u}'} $ is induced by the metric  $d_{a,u,k}$ given by 
\[d_{a,u,k}((g,y),(g',y'))=d_{i_{k}}(y,y')+d_{a}(g,g')\]
(recall that $\overline{W_{k}}\subset T_{i_{k}}$).
Let $\hat{d}_{a,u,k}$ be the metric on $\bar{\germ}(\RR_{a}(\overline{S}_{1,a})\times \overline{V_{0,a}\cap V_{0,u}'})$ that corresponds to $d_{a,u,k}$ by the homeomorphism~\eqref{bar gamma}; it induces the topology of $\bar{\germ}(\RR_{a}(\overline{S}_{1,a})\times \overline{V_{0,a}\cap V'_{0,u}})$. By the commutativity of~\eqref{Diagram 1},
	\[
		\bar{\germ}(\RR_{a}(\overline{S}_{1,a})\times \overline{V_{0,a}\cap V_{0,u}'})
		=\germ(\RR_{a}(\overline{S}_{1,a})\times \overline{V_{0,a}\cap V_{0,u}'})\;,
	\]
which is contained in $\widehat{T}$.  Then the restriction $\hat{d}_{0,a,u,k}$ of $\hat{d}_{a,u,k}$ to 
	\[ 
		\bar{\germ}(\RR_{a}(\overline{S}_{1,a})\times \overline{V_{0,a}\cap V_{0,u}'}) \cap \widehat{T}_{0}
	\] 
induces  the topology of this space. Moreover, according to  the proof of Corollary~\ref{c: widehat T_U is locally compact}, we get 
	$$
		\widehat{T}_{a,u}\subset\bar{\germ}(\RR_{a}(\overline{S}_{1,a})\times \overline{V_{0,a}\cap V_{0,u}'})\;,
	$$
and therefore 
	\[
		\widehat{T}_{0,a,u}\subset\bar{\germ}(\RR_{a}(\overline{S}_{1,a})\times \overline{V_{0,a}\cap V_{0,u}'})\cap\widehat{T}_0\;.
	\]

For any index  $v$, define $\overline{S}'_{0,v}$ and $\overline{S}'_{1,v}$ like $\overline{S}_0$ and $\overline{S}_1$ in~\eqref{S0} and~\eqref{S1} by using $V'_{v}$ and $W_{k_0}$ instead of $V$ and $W$. Let $\RR'_{v}:\overline{S}'_{1,v}\rightarrow C(\overline{V'_v},\overline{W_{k_0}})$ denote the restriction map. Again, the compact-open topology  on $\RR'_v(\overline{S}'_{1,v})$ equals the topology induced by the supremum metric $d'_v$  on $ C(\overline{V'_v},\overline{W_{k_0}})$, defined with the metric $d_{i_{k_0}}$ on $T_{i_{k_0}}$ (recall that $\overline{W_{k_0}}\subset T_{i_{k_0}}$). Take indices $b$ and $l$ such that   $\overline{V'_v}\subset V_b$ and $\overline{V_b}\subset W_l$.  Then we can consider the restriction map 
	\[
		\RR^{v}_{b}:C(\overline{V_b},\overline{W_{k_0}})\rightarrow C(\overline{V'_v},\overline{W_{k_0}})\;.
	\]
Its restriction $\RR^{v}_{b}:\RR_{b}(\overline{S}_{1,b})\rightarrow \RR'_{v}(\overline{S}'_{1,v})$
is injective by Remark~\ref{r: strong quasi-analyticity with ol S}, and surjective by Remark~\ref{r: tilde h}. So $\RR^{v}_{b}:\RR_{b}(\overline{S}_{1,b})\rightarrow \RR'_{v}(\overline{S}'_{1,v})$ is a continuous bijection between compact Hausdorff spaces, obtaining that it is a homeomorphism. Then, by compactness, it is a uniform homeomorphism with respect to the supremum metrics $d_{b}$ and $d'_{v}$. Since $b$ and $v$ run in finite families of indices, there is a mapping $\epsilon\mapsto\delta_{1}(\epsilon)>0$ ($\epsilon>0$) such that 
	\begin{equation}\label{d'_v} 
 		d'_{v}(\RR^{v}_{b}\RR_{b}(f),\RR^{v}_{b}\RR_{b}(f'))<\delta_{1}(\epsilon)\Longrightarrow d_{b}(\RR_{b}(f),\RR_{b}(f'))<\epsilon
	\end{equation}
for all indices $v$ and $b$, and maps $f,f'\in \overline{S}_{1,b}$.

\begin{lem}\label{l: widehat HH_0,U is equicontinuous} 
 $\widehat{\HH}_{0,U}$ satisfies the equicontinuity condition with $\widehat{S}_{0,U}=\widehat{S}_{0}\cap\widehat{\HH}_{0,U}$ and the quasi-local metric represented by the family $\{\widehat{T}_{0,U,a,u},\hat{d}_{0,a,u,k}\}$.
\end{lem}

\begin{proof}\setcounter{claim}{0}
 Let $h\in S$, and
 	$$
		\germ(g,y),\germ(g',y')\in\widehat{T}_{0,U,a,u}\cap\hat{h}^{-1}(\widehat{T}_{0,U,b,v})\;,
	$$
where $g,g'\in\overline{S}_{0,a}$ and $y,y'\in V_{0,a}\cap V'_{0,u}$ with $g(y)=g(y')=x_{0}$. Take some indices $k$ and
 $l$ such that $\overline{V_a}\subset W_{k}$ and $\overline{V_b}\subset W_{l}$ (recall that $\overline{W_k}\subset T_{i_{k}}$ and $\overline{W_{l}}\subset T_{i_{l}}$). By Remark~\ref{r: tilde h}, we can assume that $\dom h=T_{i_{k}}$. Then 
	$$
		\hat{h}(\germ(g,y))=\germ(gh^{-1},h(y))\;,\quad
		\hat{h}(\germ(g',y'))=\germ(g'h^{-1},h(y'))
	$$
belong to $\widehat{T}_{0,U,b,v}$, which means that $h(y),h(y')\in V_{0,b}\cap V'_{0,v}$ and there are $f,f'\in \overline{S}_{0,b}$ so that
	\begin{equation}\label{gamma(f,h(y))}
		\germ(f,h(y))=\germ(gh^{-1},h(y))\;,\quad
		\germ(f',h(y'))=\germ(g'h^{-1},h(y'))\;;
	\end{equation}
in particular, $\overline{V_{b}}\subset\dom f\cap\dom f'$. In fact, we can assume that $\dom f=\dom f'=T_{i_{l}}$ by Remark~\ref{r: tilde h}. Observe that the image of $h$ may not be included in $T_{i_{l}}$, and the images of $f$, $f'$,  $g$ and $g'$ may no be included in $T_{i_{k_{0}}}$.

\begin{claim}\label{overline V'_v} 
 $\overline{V'_v}\subset\im h$ and $h^{-1}(\overline{V'_v})\subset V_{a}$.
 \end{claim}
 
By the assumptions on $\{V'_w\}$, since 
	$$
		h(y)\in U\cap V'_{v}\cap V_{0,b}\cap\dom h^{-1}\;,\quad
		h^{-1}h(y)=y \in U\cap V'_{u}\cap V_{0,a}\;,
	$$
there is some $\widetilde{h^{-1}}\in S$ such that 
	$$
		\overline{V'_v}\subset\dom\widetilde{h^{-1}}\cap V_{b}\;,\quad
		\widetilde{h^{-1}}(\overline{V'_{v}})\subset V_{a}\;,\quad
		\germ(\widetilde{h^{-1}},h(y))=\germ(h^{-1},h(y))\;;
	$$
indeed, we can suppose that $\dom\widetilde{h^{-1}}=T_{i_{k_{0}}}$ by Remark~\ref{r: tilde h}. Then 
	$$
		\widetilde{h^{-1}}(\overline{V'_v})\subset V_{a}\subset T_{i_{k}}=\dom h\;,
	$$
obtaining $\overline{V'_v}\subset\dom(h\widetilde{h^{-1}})$. Moreover 
	$$
		\germ(h\widetilde{h^{-1}},h(y))=\germ(\id_T,h(y))\;.
	$$
Therefore $h\widetilde{h^{-1}}=\id_{\dom(h\widetilde{h^{-1}})}$ because $h\widetilde{h^{-1}}\in S$ since  $h,\widetilde{h^{-1}}\in S$. So $h\widetilde{h^{-1}}=\id_T$ on some neighborhood of $\overline{V'_v}$, and therefore $\overline{V'_v}\subset\im h$ and $h^{-1}=\widetilde{h^{-1}}$ on $\overline{V'_v}$. Thus 
$h^{-1}(\overline{V'_v})=\widetilde{h^{-1}}(\overline{V'_v})\subset V_{a}$, which shows Claim~\ref{overline V'_v}.

By Claim~\ref{overline V'_v} and since $\overline{V_a}\subset\dom g\cap\dom g'$ because $g,g'\in \overline{S}_{0,a}$, we get
	\begin{equation}\label{overline V'_v subset dom gh^-1} 
 		\overline{V'_v}\subset\dom(gh^{-1})\cap\dom(g'h^{-1})\;.
	\end{equation} 
Since $f,f'\in \overline{S}_{0,b}$, we have $\overline{V_b}\subset\dom f\cap\dom f'$ and $f(\overline{V_b})\cup f'(\overline{V_b})\subset W_{k_{0}}$. On the other hand, it follows from~\eqref{gamma(f,h(y))} that $fh(y)=f'h(y')=x_{0}$ and 
$$
\germ(gh^{-1}f^{-1},x_{0})=\germ(g'h^{-1}{f'}^{-1},x_{0})=\germ(\id_T,x_{0})\;.
$$
Moreover 
	$$
		f(\overline{V'_v})\subset \dom(gh^{-1}f^{-1})\;,\quad
		f'(\overline{V'_v})\subset\dom(g'h^{-1}{f'}^{-1})
	$$
by~\eqref{overline V'_v subset dom gh^-1}. So, by Remark~\ref{r: strong quasi-analyticity with ol S}, $gh^{-1}f^{-1}=\id_T$ on some neighborhood of $f(\overline{V'_v})$, and $g'h^{-1}{f'}^{-1}=\id_T$ on some neighborhood of $f'(\overline{V'_v})$. Thus $gh^{-1}=f$   and $g'h^{-1}=f'$ on some neighborhood of $\overline{V'_v}$; in particular,
	\[
		\RR^{v}_{b}\RR_{b}(f)=gh^{-1}|_{\overline{V'_v}}\;,\quad\RR^{v}_{b}\RR_{b}(f')=g'h^{-1}|_{\overline{V'_v}}\;.
	\]

Consider the mappings  $\epsilon\mapsto\delta(\epsilon)>0$ and $\epsilon\mapsto\delta_1(\epsilon)>0$ satisfying~Remark~\ref{r: delta(epsilon)} and~\eqref{d'_v}. Then,  for each $\epsilon>0$, define
	$$
 		\hat\delta(\epsilon)=\min\{\delta(\epsilon/2),\delta_{1}(\epsilon/2)\}\;.
	$$ 

Given any $\epsilon>0$, suppose that
	$$
		\hat d_{0,a,u,k}(\germ(g,y),\germ(g',y'))<\hat{\delta}(\epsilon)\;.
	$$
This means that
	\[
		d_{a,u,k}((\RR_{a}(g),y), (\RR_{a}(g'),y'))<\hat{\delta}(\epsilon)\;,
	\] 
or, equivalently,
	\[
		d_{i_{k}}(y,y')+ \sup_{x\in \overline{V_{a}}}d_{i_{k_0}}(g(x),g'(x))<\hat{\delta}(\epsilon)\;.
	\]
Therefore
	\begin{gather}
		 d_{i_{k}}(y,y')<\delta(\epsilon/2)\;,\label{d_i_k} \\
		\sup_{x\in \overline{V_{a}}}d_{i_{k_0}}(g(x),g'(x))< \delta_{1}(\epsilon/2)\;.\label{sup} 
	\end{gather} 
From~\eqref{d_i_k} and Remark~\ref{r: delta(epsilon)}, it follows that 
\begin{equation}\label{d_i_l}
  d_{i_{l}}(h(y),h(y'))<\epsilon/2
\end{equation} 
since $h\in S\subset \overline{S}$  and  $y,y'\in T_{i_{k}}\cap h^{-1}(T_{i_l}\cap\im h)$.
On the other hand, by Claim~\ref{overline V'_v} and \eqref{sup}, we get 
\begin{multline*}
      d'_{v}(\RR^{v}_{b}\RR_{b}(f),\RR^{v}_{b}\RR_{b}(f'))\\
      = \sup_{z\in \overline{V'_{v}}}d_{i_{k_{0}}}(gh^{-1}(z),g'h^{-1}(z)) 
      = \sup_{x\in h^{-1}(\overline{V'_{v}})}d_{i_{k_{0}}}(g(x),g'(x))\\
      \leq \sup_{x\in \overline{V_{a}}}d_{i_{k_{0}}}(g(x),g'(x))
      = d_{a}(\RR_{a}(g),\RR_{a}(g'))
      < \delta_{1}(\epsilon/2)\;.  
\end{multline*}
So, by~\eqref{d'_v},
	\begin{equation}\label{d_b} 
 		d_{b}(\RR_{b}(f),\RR_{b}(f'))<\epsilon/2\;.
	\end{equation} 
From \eqref{d_i_l} and \eqref{d_b}, we get 
\begin{multline*}
    \hat{d}_{0,b,v,l}(\hat{h}(\germ(g,y)),\hat{h}(\germ(g',y'))) \\
    	\begin{aligned}
    		&=  \hat{d}_{0,b,v,l}(\germ(f,h(y)),\germ(f',h(y')))\\
    		& = d_{b,v,l}((\RR_{b}(f),h(y)),(\RR_{b}(f'),h(y')))\\
    		& =d_{i_{l}}(h(y),h(y'))+ d_{b}(\RR_{b}(f),\RR_{b}(f'))
    		< \epsilon\;.\qed
    	\end{aligned}
\end{multline*}
\renewcommand{\qed}{}
\end{proof}

\begin{cor}\label{c: widehat HH_0 is equicontinuous}
 $\widehat{\HH}_{0}$ is equicontinuous.
\end{cor}

\begin{proof}\setcounter{claim}{0}
 $\widehat{\HH}_{0}$ is  equivalent to $\widehat{\HH}_{0,U}$ by Lemma~\ref{l: widehat T_0,U meets all orbits}. Thus the result follows from Lemma~\ref{l: widehat HH_0,U is equicontinuous} because equicontinuity is preserved by  equivalences.
\end{proof}

\begin{lem}\label{l: widehat HH_0 is minimal}
 $\widehat{\HH}_{0}$ is minimal.
\end{lem}

\begin{proof}\setcounter{claim}{0}
 By Lemma~\ref{l: widehat T_0,U meets all orbits},  it is enough to prove that $\widehat{\HH}_{0,U}$ is minimal. Let $\germ(g,y),\germ(g',y')\in\widehat{T}_{0,U}$ with $g,g'\in \overline{S}$, $y\in\dom g\cap U$, $y'\in\dom g'\cap U$ and $g(y)=g'(y')=x_{0}$.  Take indices $k$ and $k'$ such that $y\in T_{i_{k}}$ and $y'\in T_{i_{k'}}$. We can assume that $\dom g= T_{i_{k}}$ and $ \dom g'=T_{i_{k'}}$ by Remark~\ref{r: tilde h}.

Let $f=g^{-1}g'\in \overline{S}$; we have $y'\in\dom f$ and $f(y')=y$. By Remark~\ref{r: tilde h}, there exists $\tilde{f}\in \overline{S}$ with  $\dom\tilde{f}=T_{i_{k'}}$ and $\germ(\tilde{f},y')=\germ(f,y')$. By the definition of $\overline{S}$, there is a sequence $f_{n}$ in $S$ with $\dom f_{n}=T_{i_{k'}}$ and $f_{n}\to f$ in $C_{\text{\rm c-o}}(T_{i_{k}},T)$ as $n\to \infty$; in particular, $f_{n}(y')\to f(y')=y$. So we can assume that $f_{n}(y')\in T_{i_{k}}$ for all $n$. 

Take some relatively compact open neighborhood $V$ of $y'$  such that $\overline{V}\subset\dom(g\tilde{f})\cap\dom(gf)$ and $\tilde{f}=f$ in some neighborhood of $\overline{V}$. Since  $f_{n}\to\tilde f$ in $\overline{S}_{\text{\rm c-o}}$ as $n\to\infty$, we get $gf_{n}\to g\tilde f$ and $f^{-1}_{n}\to\tilde f^{-1}$ by Propositions~\ref{p: c-o top for pseudogroups} and~\ref{p:overline S_b-c-o = overline S_c-o}. So $\overline{V}\subset\dom(gf_{n})$ and $y\in \dom f_{n}^{-1}=\im f_{n}$ for $n$ large enough, and $f^{-1}_{n}(y)\to\tilde f^{-1}(y)=y'$. Moreover $gf_{n}|_V\to g\tilde{f}|_V=gf|_V=g'|_V$ in $C_{\text{\rm c-o}}(V,T)$. So $\germ(gf_{n},f^{-1}_{n}(y)) \to \germ(g',y')$ in $\widehat{T}_{0,U}$ by Proposition~\ref{p:C(U,Y)} and the definition of the topology of $\widehat{T}$. Thus, with  $h_{n}=f^{-1}_{n}\in S$, we get 
	$$
		\widehat{h_{n}}(\germ(g,y))=\germ(gh_n^{-1},h_{n}(y))=\germ(gf_{n},f_n^{-1}(y))
		\to\germ(g',y')\;,
	$$
and therefore $\germ(g',y')$ is in the closure of the $\widehat{\HH}_{0,U}$-orbit of $\germ(g,y)$.
\end{proof}

\begin{rem}
By Lemma~\ref{l: hat pi_0(dom hat h)=dom h}, the map $\hat{\pi}_{0}:\widehat{T}_{0}\rightarrow T$ generates a morphism of pseudogroups $\widehat{\HH}_{0}\rightarrow \HH$ in the sense of \cite{AlvMasa2008}---this morphism is not \'etal\'e. 
\end{rem}

The following result is elementary.

\begin{prop}\label{p:G/K}
	In Example~\ref{ex:G/K}, if $\HH$ is compactly generated and $\ol{\HH}$ is strongly quasi-analytic, then $\widehat{\HH}_{0}$ is equivalent to the pseudogroup generated by the local action of $\Gamma$ on $G$ by local left translations, so that $\hat{\pi}_{0}:\widehat{T}_{0}\rightarrow T$ corresponds to the projection $T:V^2\to G/(K,V)$.
\end{prop}

\begin{cor}\label{c: hat pi_0:widehat T_0 rightarrow T is open}
	The map $\hat{\pi}_{0}:\widehat{T}_{0}\rightarrow T$ is open.
\end{cor}

\begin{proof}\setcounter{claim}{0}
	This follows from Theorem~\ref{t:G/K} and Proposition~\ref{p:G/K} since, in Example~\ref{ex:G/K},  the projection $T:V^2\to G/(K,V)$ is open.
\end{proof}

\subsection{The closure of $\widehat{\HH}_{0}$}\label{ss: ol widehat HH_0}

Let $\widehat{\overline{\HH}}_{0}$ be the pseudogroup on $\widehat T_0$ defined like $\widehat{\HH}$ by taking the maps $h$ in $\overline{S}$ instead of $S$; thus it is generated by $\widehat{\overline{S}}_{0}=\{\,\hat{h}\mid h\in\overline S\,\}$. Observe that $\widehat{\overline{\HH}}_{0}$, $\overline{S}$ and $\widehat{\overline{S}}_{0}$ satisfy the obvious versions of Lemmas~\ref{l: hat pi_0(dom hat h)=dom h}--\ref{l: widehat h'h},~\ref{l: hat h is a homeomorphism},~\ref{l: widehat T_0,U meets all orbits} and~\ref{l: h to hat h is a homeomorphism}, and Corollaries~\ref{c: widehat h^-1} and~\ref{c: widehat S_0 is closed by the operations} (Section~\ref{ss: widehat HH_0}). In particular, $\widehat{\overline{S}}_{0}$ is a pseudo$*$group, and $\widehat{T}_{0, U}$ meets all the orbits of $\widehat{\overline{\HH}}_{0}$. The restriction of $\widehat{\overline{\HH}}_{0}$ to $\widehat{T}_{0, U}$ will be denoted by $\widehat{\overline{\HH}}_{0,U}$.

\begin{lem}\label{l: overline widehat HH_0=widehat overline HH_0} 
	$\overline{\widehat{\HH}_0}=\widehat{\overline{\HH}}_{0}$
\end{lem}

\begin{proof}\setcounter{claim}{0}
	By the version of Lemma~\ref{l: h to hat h is a homeomorphism} for $\overline{S}$ and $\widehat{\overline{S}}_{0}$, the set $\widehat{S}_0$ is dense in $\widehat{\overline{S}}_{0,\text{\rm c-o}}$. Then the result follows easily by Proposition~\ref{p:C(U,Y)} and the definition of $\overline{\widehat{\HH}_0}$ (see Theorem~\ref{t:closure} and Remark~\ref{r: closure 1}).
\end{proof}

\begin{lem}\label{l: overline widehat HH_0 is strongly locally free}
 $\overline{\widehat{\HH}}_{0}$ is strongly locally free.
\end{lem}

\begin{proof}\setcounter{claim}{0}
 Let $\hat h\in\widehat{\overline{S}}_0$ for $h\in \overline{S}$, and $\germ(g,x)\in\dom\hat h$ for $g\in \overline{S}$ and $x\in\dom g\cap\dom h$ with $g(x)=x_{0}$. Suppose that $\hat{h}(\germ(g,x))=\germ(g,x)$. This means $\germ(gh^{-1},h(x))=\germ(g,x)$.
So $h(x)=x$ and $gh^{-1}=g$ on some neighborhood of $x$, and therefore $h=\id_T$ on some neighborhood of $x$. Then $h=\id_{\dom h}$ by the strong quasi-analyticity condition of $\overline\HH$ since $h\in \overline{S}$. Hence $\hat{h}=\id_{\dom\hat{h}}$ by Lemma~\ref{l: widehat id_O}.
\end{proof}

\begin{prop}\label{p:G}
 There is a locally compact Polish local group $G$ and some dense finitely  generated sub-local group $\Gamma \subset  G$ such that $\widehat{\HH}_{0}$ is equivalent to the pseudogroup defined by the local action of $\Gamma$ on $G$ by local left translations.
\end{prop}

\begin{proof}\setcounter{claim}{0}
 This follows from Remark~\ref{r:G} (see also Theorem~\ref{t:G/K}) since $\widehat{\HH}_{0}$ is compactly generated (Corollary~\ref{c: widehat HH_0 is compactly generated}) and equicontinuous (Corollary~\ref{c: widehat HH_0 is equicontinuous}), and $\overline{\widehat{\HH}_0}$ is strongly locally free (Lemma~\ref{l: overline widehat HH_0 is strongly locally free}).
\end{proof}

\subsection{Independence of the choices involved}\label{ss: independence}

First, let us prove that $\widehat T_0$ and $\widehat\HH_0$ are independent of the choice of the point $x_0$ up to an  equivalence generated by a homeomorphism. Let $x_1$ be another point of $T$, and let $\widehat{T}_1$, $\hat{\pi}_1$, $\widehat S_1$ and $\widehat{\HH}_1$ be constructed like $\widehat{T}_0$, $\hat{\pi}_0$, $\widehat S_0$ and $\widehat{\HH}_0$ by using $x_1$ instead of $x_0$. Now, for each $h\in S$, let us use the notation $\hat h_0:=\hat h\in\widehat S_0$, and let $\hat h_1:\hat\pi_1^{-1}(\dom h)\to\hat\pi_1^{-1}(\im h)$ be the map in $\widehat S_1$ defined like $\hat h$.

\begin{prop}\label{p: widehat T_0 to widehat T_1}
There is a homeomorphism $\theta:\widehat{T}_{0}\to \widehat{T}_{1}$ that generates an equivalence 
$\Theta:\widehat{\HH}_{0}\rightarrow \widehat{\HH}_{1}$ and so that $\hat{\pi}_{0}=\hat{\pi}_{1}\theta$.
\end{prop}

\begin{proof}\setcounter{claim}{0}
 Since $\HH$ is minimal, there is some $f_0\in \overline{S}$ such that $x_{0}\in \dom f_{0}$ and $f_{0}(x_{0})=x_{1}$. Let $\theta:\widehat{T}_{0}\rightarrow \widehat{T}_{1}$ be defined by $\theta(\germ(f,x))= \germ(f_{0}f,x)$. This map is continuous because $\theta(\germ(f,x))= \germ(f_{0},x)\germ(f,x)$. So $\theta$ is a homeomorphism because  $f^{-1}_{0}$ defines $\theta^{-1}$  in the same way. We also have 
$\hat{\pi}_{0}=\hat{\pi}_{1}\theta$ since $\theta$ preserves the source of each germ.
For each $h\in S$, we have $ \dom \hat{h}_{1}= \theta(\dom \hat{h}_{0})$  because $\hat{\pi}_{0}=\hat{\pi}_{1}\theta$, and $\hat{h}_{1}\theta=\theta$ since 
	\begin{multline*}
    		\hat{h}_{1}\theta(\germ(f,x))
    		=\hat{h}_{1}(\germ(f_{0}f,x)) 
    		= \germ(f_{0}fh^{-1},h(x)) \\
    		= \theta(\germ(fh^{-1},h(x)))
    		= \theta(\hat{h}_{0}(\germ(f,x))) 
	\end{multline*}
for all $\germ(f,x)\in \dom \hat{h}_{0}$. It follows easily that $\theta$ generates an \'{e}tale morphism $\Theta:\widehat{\HH}_{0}\rightarrow \widehat{\HH}_{1}$, which is an equivalence since $\theta^{-1}$ generates $\Theta^{-1}$.
\end{proof}

Now, let us show that the topology of $\widehat T$ is independent of the choice of $S$. Therefore the topology of $\widehat T_0$ will be independent of the choice of $S$ as well. Let $S',S''\subset\HH$ be two sub-pseudo$*$groups generating $\HH$ and satisfying the conditions of Section~\ref{ss: conditions}. With the notation of Section~\ref{ss: coincidence of topologies}, we have to prove the following.

\begin{prop}\label{p: overline Gamma_overline S',c-o = overline Gamma_overline S'',c-o}
	$\overline{\mathfrak G}_{\overline{S'},\text{\rm c-o}}=\overline{\mathfrak G}_{\overline{S''},\text{\rm c-o}}$.
\end{prop}

\begin{proof}\setcounter{claim}{0}
	First, up to solving the case where $S'\subset S''$, we can assume that $S'$ and $S''$ are local by Remarks~\ref{r: strongly quasi-analytic} and~\ref{r: equicont}. Second, if $S'$ and $S''$ are local, then the sub-pseudo$*$group $S'\cap S''$ of $\HH$ also generates $\HH$. Moreover $S'\cap S''$ obviously satisfies all other properties required in Section~\ref{ss: conditions}; note that a refinement of $\{T_i\}$ may be necessary to get the properties stated in Remarks~\ref{r: delta(epsilon)}--\ref{r: strong quasi-analyticity with ol S} with $S'\cap S''$. Hence the result follows from the special case where $S'\subset S''$. With this assumption, the identity map $\overline{\mathfrak G}_{\overline{S'},\text{\rm c-o}}\to\overline{\mathfrak G}_{\overline{S''},\text{\rm c-o}}$ is continuous because the diagram
		$$
			\begin{CD}
				\overline{S'}_{\text{\rm c-o}} @>{\text{inclusion}}>> \overline{S''}_{\text{\rm c-o}} \\
				@V{\germ}VV	@VV{\germ}V  \\
				\overline{\mathfrak G}_{\overline{S'},\text{\rm c-o}} @>{\text{identity}}>> \overline{\mathfrak G}_{\overline{S''},\text{\rm c-o}}
			\end{CD}
		$$
	is commutative, where the vertical maps are identifications and the top map is continuous. 
	
	For any compact subset $Q\subset T$, let $s^{-1}(Q)_{\overline{S'},\text{\rm c-o}}$ and  $s^{-1}(Q)_{\overline{S''},\text{\rm c-o}}$ denote the spaces obtained by endowing $s^{-1}(Q)$ with the restriction of the topologies of $\overline{\mathfrak G}_{\overline{S'},\text{\rm c-o}}$ and $\overline{\mathfrak G}_{\overline{S''},\text{\rm c-o}}$, respectively. They are compact and Hausdorff by Propositions~\ref{p: widehat T is Polish and locally compact} and~\ref{p: hat pi:widehat T to T is proper} . So $s^{-1}(Q)_{\overline{S'},\text{\rm c-o}}=s^{-1}(Q)_{\overline{S''},\text{\rm c-o}}$ because the identity map $s^{-1}(Q)_{\overline{S'},\text{\rm c-o}}\to s^{-1}(Q)_{\overline{S''},\text{\rm c-o}}$ is continuous. Hence, for any $\germ(f,x)\in\overline{\mathfrak G}$ and a compact neighborhood $Q$ of $x$ in $T$, the set $s^{-1}(Q)$ is a neighborhood of $\germ(f,x)$ in $\overline{\mathfrak G}_{\overline{S'},\text{\rm c-o}}$ and $\overline{\mathfrak G}_{\overline{S''},\text{\rm c-o}}$ with $s^{-1}(Q)_{\overline{S'},\text{\rm c-o}}=s^{-1}(Q)
_{\overline{S''},\text{\rm c-o}}$. This shows that the identity map $\overline{\mathfrak G}_{\overline{S'},\text{\rm c-o}}\to\overline{\mathfrak G}_{\overline{S''},\text{\rm c-o}}$ is a local homeomorphism, and therefore a homeomorphism.
\end{proof}

Let $T'$ be an open subset of $T$ containing $x_0$, which meets all orbits because $\HH$ is minimal. Then use $T'$, $\HH'=\HH|_{T'}$ and $S'=S\cap\HH'$ to define $\widehat{T}'_0$, $\hat\pi'_0$, $\widehat{S}'_0$ and $\widehat{\HH}'_0$ like $\widehat{T}_0$, $\hat{\pi}_0$, $\widehat S_0$ and $\widehat{\HH}_0$. The proof of the following result is elementary.

\begin{prop}\label{p:HH'}
	There is a canonical identity of topological spaces, $\widehat{T}'_0\equiv{\hat\pi_0}^{-1}(T')$, so that $\hat\pi'_0\equiv\hat\pi_0|_{\widehat{T}'_0}$ and $\widehat{\HH}'_0=\widehat{\HH}_0|_{\widehat{T}'_0}$.
\end{prop}

\begin{cor}\label{c:HH'}
	Let $\HH$ and $\HH'$ be minimal equicontinuous compactly generated pseudogroups on locally compact Polish spaces such that $\ol\HH$ and $\ol{\HH'}$ are strongly quasi-analytic. If $\HH$ is equivalent to $\HH'$, then $\widehat\HH_0$ is equivalent to $\widehat{\HH}'_0$.
\end{cor}

\begin{proof}\setcounter{claim}{0}
	This is a direct consequence of Propositions~\ref{p: widehat T_0 to widehat T_1}--\ref{p:HH'}.
\end{proof}

The following definition makes sense by Lemma~\ref{l:G}, Propositions~\ref{p: widehat T_0 to widehat T_1} and~\ref{p: overline Gamma_overline S',c-o = overline Gamma_overline S'',c-o}, and Corollary~\ref{c:HH'}.

\begin{defn}\label{d: structural local group}
 In Proposition~\ref{p:G}, it is said that (the local isomorphism class of) $G$ is the {\em structural local group\/} of (the equivalence class of) $\HH$.
\end{defn}

\section{Molino's theory for equicontinuous foliated spaces}\label{s: Molino foliated}

\subsection{Preliminaries on equicontinuous foliated spaces}\label{ss: prelim equicont foliated spaces}

(See \cite{MooreSchochet1988}, \cite[Chapter~11]{CandelConlon2000-I} and \cite{Ghys2000}.)

Let $X$ and $Z$ be locally compact Polish spaces. A {\em foliated chart\/} in $X$ of {\em leaf dimension\/} $n$, {\em transversely modeled\/} on $Z$, is a pair $(U, \phi)$, where $U\subseteq X$ is open and $\phi:U\rightarrow B\times T$ is a homeomorphism for some open $T\subset Z$ and some open ball $B$ in $\R^{n}$. It is said that $U$ is a {\em distinguished open set\/}. The sets  $P_{y}=\phi^{-1}(B\times \{y\})$ ($y\in T$) are called {\em plaques\/} of this foliated chart. For each $x\in B$, the set $S_{x}=\phi^{-1}(\{x\}\times T)$ is called a {\em transversal\/} of the foliated chart. This local product structure defines a local projection $p:U\to T$, called {\em distinguished submersion\/}, given as composition of $\phi$ with the second factor projection $\pr_{2}: B\times T\to T$.

Let $\UU=\{U_{i}, \phi_{i}\}$ be a family of foliated charts in $X$ of leaf dimension $n$ modeled transversally on $Z$ and covering $X$. Assume further that the foliated charts are {\em coherently foliated\/} in the sense that, if $P$ and $Q$ are plaques in different charts of $\UU$, then $P\cap Q$ is open both in $P$ and $Q$. Then $\UU$ is called a  {\em foliated atlas\/} on $X$ of {\em leaf dimension\/} $n$ and {\em transversely modeled\/} on $Z$. A maximal foliated atlas $\FF$ of leaf dimension $n$ and transversely modeled on $Z$ is called a {\em foliated structure\/} on $X$ of {\em leaf dimension\/} $n$ and {\em transversely modeled\/} on $Z$. Any foliated atlas $\UU$ of this type is contained in a unique foliated structure $\FF$; then it is said that $\UU$ {\em defines\/} (or is an atlas of) $\FF$. If $Z=\R^m$, then $X$ is a manifold of dimension $n+m$, and $\FF$ is traditionally called a {\em foliation\/} of {\em dimension\/} $n$ and {\em codimension\/} $m$. The reference to $Z$ will be omitted.

For a foliated structure $\FF$ on $X$ of dimension $n$, the plaques form a basis of a topology on $X$ called the {\em leaf topology\/}. With the leaf topology, $X$ becomes an $n$-manifold whose connected components are called {\em leaves\/} of $\FF$. $\FF$ is determined by its leaves.

 A foliated atlas $\UU=\{U_{i}, \phi_{i}\}$ of $\FF$ is called {\em regular\/} if
 \begin{itemize}
 
\item each $\overline{U_{i}}$ is a compact subset of a foliated chart $(W_{i}, \psi_{i})$ and $\phi_{i}=\psi_{i}|_{U_i}$;

\item the cover $\{U_{i}\}$ is locally finite; and,

\item if $(U_{i}, \phi_{i})$ and  $(U_{j}, \phi_{j})$ are elements of $\UU$, then each plaque 
$P$ of $(U_{i}, \phi_{i})$ meets at most one plaque of $(U_{j}, \phi_{j})$.

\end{itemize}
In this case, there are homeomorphisms $h_{ij}:T_{ij}\to T_{ji}$ such that $h_{ij}p_i=p_j$ on $U_i\cap U_j$, where $p_i:U_i\to T_i$ is the distinguished submersion defined by $(U_i,\phi_i)$ and $T_{ij}=p_i(U_i\cap U_j)$. Observe that the cocycle condition $h_{ik}=h_{jk}h_{ij}$ is satisfied on $T_{ijk}=p_i(U_i\cap U_j\cap U_k)$. For this reason, $\{U_i,p_i,h_{ij}\}$ is called a {\em defining cocycle\/} of $\FF$ with values in $Z$---we only consider defining cocycles induced by regular foliated atlases. The equivalence class of the pseudogroup $\HH$ generated by the maps $h_{ij}$ on $T=\bigsqcup_{i\in I} T_{i}$ is called the 
{\em holonomy pseudogroup\/} of the foliated space $(X,\FF)$; $\HH$ is the representative of the 
holonomy pseudogroup of $(X,\FF)$ induced by the defining cocycle $\{U_i,p_i,h_{ij}\}$. This $T$ can be identified with a {\em total\/} {\rm(}or {\em complete\/}{\rm)} {\em transversal\/} to the leaves in the sense that it meets all leaves and is locally given by the transversals defined by foliated charts. All compositions of maps $h_{ij}$ form a pseudo$*$group $S$ that generates $\HH$, called the {\em holonomy pseudo$*$group\/} of $\FF$ induced by $\{U_i,p_i,h_{ij}\}$. There is a canonical identity between the space of leaves and the space of $\HH$-orbits, $X/\FF\equiv T/\HH$.

A foliated atlas (respectively, defining cocycle) contained in another one is called {\em sub-foliated atlas\/} (respectively, {\em sub-foliated cocycle\/}).

The {\em holonomy group\/} of each leaf $L$ is defined as the germ group of the corresponding orbit. It can be considered as a quotient of $\pi_1(L)$ by taking ``chains'' of sets $U_i$ along loops in $L$; this representation of $\pi_1(L)$ is called the {\em holonomy representation\/}. The kernel of the holonomy representation is equal to $q_*\pi_1(\widetilde L)$ for a regular covering space $q:\widetilde{L}\to L$, which is called the {\em holonomy cover\/} of $L$. If $\FF$ admits a countable defining cocycle, then the leaves in some dense $G_\delta$ subset of $M$ have trivial holonomy groups  \cite{HectorHirsch1981-A,HectorHirsch1983-B,CandelConlon2000-I}, and therefore they can be identified with their holonomy covers.

It is said that a foliated space is ({\em topologically\/}) {\em transitive\/} or {\em minimal\/} if any representative of its holonomy pseudogroup is such. Transitiviness (respectively, minimality) of a foliated space means that some leaf is dense (respectively, all leaves are dense).

Haefliger \cite{Haefliger2002} has observed that, if $X$ is compact, then $\HH$ is compactly generated, which can be seen as follows. There is some defining cocycle $\{U'_i,p'_i,h'_{ij}\}$, with $p'_i:U'_i\to T'_i$, such that $\overline{U_i}\subset U'_i$, $T_i\subset T'_i$ and $p'_i$ extends $p_i$. Therefore each $h'_{ij}$ is an extension of $h_{ij}$ so that $\overline{\dom h_{ij}}\subset\dom h'_{ij}$. Moreover $\HH$ is the restriction to $T$ of the pseudogroup $\HH'$ on $T'=\bigsqcup_iT'_i$ generated by the maps $h'_{ij}$, and $T$ is a relatively compact open subset of $T'$ that meets all $\HH'$-orbits.

\begin{defn}
It is said that a foliated space is {\em equicontinuous\/} if any representative of its holonomy
pseudogroup is equicontinuous.
\end{defn}

\begin{rem}
	The above definition makes sense by Lemma~\ref{l: equicont is inv by equivs}. 
\end{rem}

\begin{defn}
Let $G$ be a locally compact Polish local group. A minimal foliated space is called a {\em $G$-foliated space\/} if its holonomy pseudogroup can be represented by a pseudogroup given by Example~\ref{ex:G} on a local group locally isomorphic to $G$.
\end{defn}

\subsection{Molino's theory for equicontinuous foliated spaces}

Let $(X, \FF)$ be a compact minimal foliated space that  is equicontinuous and such that the closure of its holonomy pseudogroup is strongly quasi-analytic. Let $\{U_{i}, p_{i}, h_{ij}\}$ be a defining cocycle of $\FF$ induced by a regular foliated atlas, where $p_i:U_i\to T_i$. Let $\HH$ denote the corresponding representative of  the holonomy pseudogroup on $T=\bigsqcup_iT_i$, which satisfies the conditions of Section~\ref{ss: conditions}. Let $S$ be the localization of the holonomy pseudo$*$group induced by $\{U_{i}, p_{i}, h_{ij}\}$. Fix an index $i_{0}$ and a point $x_{0}\in U_{i_{0}}$. Let $\hat\pi_{0}:\widehat{T}_{0}\rightarrow T$ and $\widehat{\HH}_{0}$ be defined like in Sections~\ref{ss: widehat T_0} and~\ref{ss: widehat HH_0}, by using $T$, $\HH$, the point $p_{i_{0}}(x_{0})\in T_{i_{0}}\subset T$, and a local sub-pseudo$*$group $S\subset\HH$.

With the notation $\widehat{T}_{i,0}=\hat\pi^{-1}_{0}(T_{i})\subset \widehat{T}_{0}$, let
	\[
  		\check{X}_{0}=\bigsqcup_{i} U_{i} \times\widehat{T}_{i,0}
		=\bigcup_{i} U_{i} \times\widehat{T}_{i,0}\times \{i\}\;, 
	\]
equipped with the corresponding topological sum of the product topologies, and consider its closed subspace
	\[
  		\widetilde{X}_{0}=\{\,(x,\gamma,i)\in\check{X}_{0} \mid p_{i}(x)=\hat\pi_{0}(\gamma)\,\}\subset\check{X}_{0}\;.
	\]
For $(x,\gamma,i),(y,\delta,j)\in\widetilde{X}_{0}$, write $(x,\gamma,i)\sim(y,\delta,j)$ if $x=y$ and $\gamma= \widehat{h_{ji}}(\delta)$. Since $h_{ij}p_{i}(x)=p_{j}(x)$, $h^{-1}_{ji}=h_{ij}$ and $h_{ik}=h_{jk}h_{ij}$, it follows that this defines an equivalence relation ``$\sim$'' on $\widetilde{X}_{0}$. Let $\widehat{X}_{0}$ be the corresponding quotient space, \(q:\widetilde{X}_{0}\rightarrow \widehat{X}_{0}\) the quotient map, and $[x,\gamma,i]$ the equivalence class of each triple $(x,\gamma,i)$. For each $i$, let 
	$$
		\check{U}_{i,0}=U_{i} \times\widehat{T}_{i,0}\times \{i\}\;,\quad
		\widetilde{U}_{i,0}=\check{U}_{i,0}\cap \widetilde{X}_{0}\;,\quad
		\widehat{U}_{i,0}=q(\widetilde{U}_{i,0})\;.
	$$

\begin{lem}\label{l:widehat U_i,0 is open}
 	$\widehat{U}_{i,0}$ is open in $\widehat{X}_{0}$.
\end{lem}

\begin{proof}\setcounter{claim}{0}
 We have to check that $q^{-1}(\widehat{U}_{i,0})\cap \widetilde{U}_{j,0}$ is open in $\widetilde{U}_{j,0}$ for all $j$, which is true because
	$$
  		q^{-1}(\widehat{U}_{i,0})\cap \widetilde{U}_{j,0}= ((U_{i}\cap U_{j}) \times\widehat{T}_{j,0}\times \{j\})\cap \widetilde{X}_{0}\;.\qed
	$$
\renewcommand{\qed}{}
\end{proof}

\begin{lem}\label{l: q is a homeomorphism}
 $q:\widetilde{U}_{i,0}\rightarrow \widehat{U}_{i,0}$ is a homeomorphism.
\end{lem}

\begin{proof}\setcounter{claim}{0}
 This map is surjective by the definition of $\widehat{U}_{i,0}$. On the other hand, two equivalent triples in $\widetilde{U}_{i,0}$ are of the 
form $(x,\gamma,i)$ and $(x,\delta,i)$ with $\gamma=\widehat{h_{ii}}(\delta)=\delta$.  So $q:\widetilde{U}_{i,0}\rightarrow \widehat{U}_{i,0}$ is also injective.  Since $q:\widetilde{U}_{i,0}\rightarrow \widehat{U}_{i,0}$ is continuous, it only remains to prove that this map is open.
A basis of the topology of $\widetilde{U}_{i,0}$ consists of the sets of the form $(V\times W\times\{i\})\cap \widetilde{X}_{0}$, where
$V$ and $W$ are open in $U_{i}$ and $\widehat{T}_{i,0}$, respectively. These basic sets satisfy
\begin{multline*}
  	\widetilde{U}_{j,0}\cap q^{-1}q\left((V\times W\times\{i\})\cap \widetilde{X}_{0}\right)\\
  	=\widetilde{U}_{j,0}\cap\left(V\times\widehat{h_{ij}}(W\cap \dom \widehat{h_{ij}})\times\{j\}\right)
\end{multline*}
for all $j$, which is open in $\widetilde{U}_{j,0}$. So $q^{-1}q((V\times W\times\{i\})\cap \widetilde{X}_{0})$ is open in $\widetilde{X}_{0}$ and therefore $q((V\times W\times\{i\})\cap \widetilde{X}_{0})$ is open in $\widehat{X}_{0}$.
\end{proof}

\begin{prop}\label{p: widehat X_0 is compact and Polish}
 $\widehat{X}_{0}$ is compact and Polish.
\end{prop}

\begin{proof}\setcounter{claim}{0}
 Let $\{U'_{i},p'_{i},h'_{ij}\}$ be a shrinking of $\{U_{i},p_{i},h_{ij}\}$; i.e., it is a defining cocycle of $\FF$  such that
 $\overline{U'_{i}}\subset U_{i}$ and $p_{i}':U'_{i}\rightarrow T'_{i}$ is the restriction of $p_{i}$ for all $i$. Therefore each $h'_{ij}$ is also a restriction
of $h_{ij}$ and $T'_{i}$ is a relatively compact open subset of $T_{i}$. Then $\hat\pi^{-1}_{0}(\overline{T'_{i}})$ is a compact subset of 
$\widehat{T}_{i,0}$ by Corollary~\ref{c: hat pi_0: widehat T_0 to T is proper}. Moreover $\widehat{X}_{0}$  is the union of the sets $q(\overline{U'_{i}}\times\hat\pi^{-1}_{0}(\overline{T'_{i}})\times\{i\})$. So $\widehat{X}_{0}$ is compact because it is a finite union of compact sets. 

On the other hand, since $\widetilde{X}_{0}$ is closed in $\check{X}_{0}$, and $\check{U}_{i,0}$ is Polish and locally compact by Corollary~\ref{c: widehat T_0 is Polish and locally compact}, it follows that $\widetilde{U}_{i,0}$ is Polish and locally compact, and therefore $\widehat{U}_{i,0}$ is Polish and locally compact by Lemma~\ref{l: q is a homeomorphism}. Then, by the compactness of $\widehat{X}_{0}$, Lemma~\ref{l:widehat U_i,0 is open} and \cite[Theorem~5.3]{Kechris1991}, it only remains to prove that $\widehat{X}_{0}$ is Hausdorff.

Let $[x,\gamma,i]\neq[y,\delta,j]$ in $\widehat{X}_{0}$. So $x\in U_{i}$ and $y\in U_{j}$. If $x=y$, then 
$[y,\delta,j]=[x,\widehat{h_{ji}}(\delta),i]\in\widehat{U}_{i,0}$. Thus, in this case, $[x,\gamma,i]$  and  $[y,\delta,j]$ can be separated
by open subsets of $\widehat{U}_{i,0}$ because $\widehat{U}_{i,0}$ is Hausdorff.

Now suppose that $x\neq y$. Then take disjoint open neighborhoods, $V$ of $x$ in $U_{i}$ and  $W$ of $y$ in  $U_{j}$. Let 
\begin{alignat*}{2}
  \check{V}&=V\times \widehat{T}_{i,0}\times\{i\}\subset \check{U}_{i,0}\;,&\quad
  \check{W}&=V\times \widehat{T}_{j,0}\times\{j\}\subset \check{U}_{j,0}\;,\\
  \widetilde{V}&=\check{V}\cap\widetilde{X}_{0}\subset \widetilde{U}_{i,0}\;,&\quad
  \widetilde{W}&=\check{W}\cap\widetilde{X}_{0}\subset \widetilde{U}_{j,0}\;,\\
  \widehat{V}&=q(\widetilde{V})\subset\widehat{U}_{i,0}\;,&\quad
  \widehat{W}&=q(\widetilde{W})\subset\widehat{U}_{j,0}\;.
\end{alignat*}
The sets $\widehat{V}$ and  $\widehat{W}$ are open neighborhoods of $[x,\gamma,i]$ and $[y,\delta,j]$ in  $\widehat{X}_{0}$. Suppose that $\widehat{V}\cap\widehat{W}\neq \emptyset$. Then there is a point 
$(x', \gamma',i)\in\widetilde{V}$ equivalent to  some point 
$(y', \delta',j)\in\widetilde{W}$. This implies that $x'=y'\in V\cap W$, which is a contradiction because $V\cap W=\emptyset$. Therefore $\widehat{V}\cap\widehat{W}=\emptyset$.
\end{proof} 

According to the above equivalence relation of triples, a map $\hat{\pi}_{0}:\widehat{X}_{0}\rightarrow X$ is defined by $\hat{\pi}_{0}([x,\gamma,i])=x$.

\begin{prop}\label{p: hat pi_0 is continuous and surjective} 
The map $\hat{\pi}_{0}:\widehat{X}_{0}\to X$ is continuous and surjective, and its fibers are homeomorphic to each other.
\end{prop}

\begin{proof}\setcounter{claim}{0}
Since each map $\hat{\pi}_{0}:\widehat{T}_{i,0}\rightarrow T_{i}$ is surjective, we have $\hat{\pi}_{0}(\widehat{U}_{i,0})=U_{i}$, obtaining that $\hat{\pi}_{0}:\widehat{X}_{0}\to X$ is surjective. Moreover the composition 
$$
  \begin{CD} \widetilde{U}_{i,0} @>q>> \widehat{U}_{i,0} @>{\hat{\pi}_{0}}>> U_{i} \;, \end{CD} 
$$
is the restriction of the first factor projection $\check{U}_{i,0}\rightarrow U_{i}$, $(x,\gamma,i)\mapsto x$. 
So $\hat{\pi}_{0}:\widehat{X}_{0}\rightarrow X$ is continuous by Lemmas~\ref{l:widehat U_i,0 is open} and~\ref{l: q is a homeomorphism}.

For  $x\in U_{i}$, we have $\hat{\pi}^{-1}_{0}(x)\subset\widehat{U}_{i,0}$ and
$$
	\widetilde{U}_{i,0}\cap q^{-1}(\hat{\pi}^{-1}_{0}(x))
	=\{x\}\times\hat{\pi}^{-1}_{0}(p_{i}(x))\times\{i\}
	\equiv \hat{\pi}^{-1}_{0}(p_{i}(x)) \subset \widehat{T}_{i,0}\;.
$$
So the last assertion of the statement follows from Lemma~\ref{l: q is a homeomorphism} and Proposition~\ref{p: the fibers of hat pi_0: widehat T_0 to T}.
\end{proof}

Let $\tilde{p}_{i,0}:\widetilde{U}_{i,0}\rightarrow\widehat{T}_{i,0}$ denote the restriction of the second factor projection $\check{p}_{i,0}:\check{U}_{i,0}=U_{i}\times\widehat{T}_{i,0}\times\{i\}\rightarrow\widehat{T}_{i,0}$.
By Lemma~\ref{l: q is a homeomorphism}, $\tilde{p}_{i,0}$ induces a continuous map 
$\hat{p}_{i,0}:\widehat{U}_{i,0}\rightarrow\widehat{T}_{i,0}$.

\begin{prop}\label{p: defining cocycle of widehat FF_0}
$\{\widehat{U}_{0.i},\hat{p}_{i,0},\widehat{h_{ij}}\}$ is a defining cocycle of a foliated structure $\widehat{\FF}_{0}$ on $\widehat{X}_{0}$.
\end{prop}

\begin{proof}\setcounter{claim}{0}
Let $\{U_{i},\phi_{i}\}$ be a regular foliated atlas of $\FF$ inducing the defining cocycle $\{U_{i},p_{i},h_{ij}\}$, where 
$\phi_{i}:U_{i}\rightarrow B_{i}\times T_{i}$ is a homeomorphism and $B_{i}$ is a ball in $\R^{n}$ ($n=\dim\FF$). Then we get a homeomorphism
\[
  \check{\phi}_{i,0}=\phi_{i}\times \id\times \id:\check{U}_{i,0}=U_{i}\times\widehat{T}_{i,0}\times\{i\}\rightarrow B_{i}\times T_{i}\times\widehat{T}_{i,0}\times\{i\}\;.
\]
Observe that $\check{\phi}_{i,0}(\widetilde{U}_{i,0})$ consists of the elements $(y,z,\gamma,i)$ with $\hat{\pi}_{0}(\gamma)=z$.
So $\check{\phi}_{i,0}$ restricts to a homeomorphism 
\[
  \tilde{\phi}_{i,0}:\widetilde{U}_{i,0}\rightarrow \check{\phi}_{i,0}(\widetilde{U}_{i,0})\equiv B_{i}\times\widehat{T}_{i,0}\times\{i\}\equiv B_{i}\times\widehat{T}_{i,0}\;.
\]
By Lemma~\ref{l: q is a homeomorphism}, $\tilde{\phi}_{i,0}$ induces a homeomorphism $\hat{\phi}_{i,0}:\widehat{U}_{i,0}\rightarrow B_{i}\times\widehat{T}_{i,0}$.
Moreover, $\check{p}_{i,0}$ corresponds to the third factor projection via $\check{\phi}_{i,0}$, 
obtaining that $\tilde{p}_{i,0}$ corresponds to the second factor  projection via $\tilde{\phi}_{i,0}$, and therefore $\hat{p}_{i,0}$ also 
corresponds  to the second factor projection via $\hat{\phi}_{i,0}$. Observe that $\hat{p}_{i,0}=\widehat{h_{ji}}\hat{p}_{j,0}$ on $\widehat{U}_{i,0}\cap\widehat{U}_{j,0}$ by the definition of ``$\sim$''. The regularity of the foliated atlas $\{\widehat{U}_{0.i},\hat{\phi}_{i,0}\}$ follows easily from the regularity of $\{U_{i},\phi_{i}\}$.
\end{proof}

According to Proposition~\ref{p: defining cocycle of widehat FF_0}, the holonomy pseudogroup of $\widehat{\FF}_{0}$ is represented by the pseudogroup on $\bigsqcup_{i}\widehat{T}_{i,0}$ 
generated by the maps $\widehat{h_{ij}}$, which is the pseudogroup $\widehat{\HH}_{0}$ on $\widehat{T}_{0}$.

\begin{cor}\label{c: G of FF}
 There is some locally compact Polish local group $G$ such that $(\widehat X_0,\widehat\FF_0)$ is a minimal $G$-foliated space; in particular, it is equicontinuous.
\end{cor}

\begin{proof}\setcounter{claim}{0}
 This follows from Propositions~\ref{p: defining cocycle of widehat FF_0} and~\ref{p:G}, and Corollary~\ref{l: widehat HH_0 is minimal}.
\end{proof}

\begin{prop}\label{p: hat pi_0 is foliated}
 $\hat{\pi}_{0}:(\widehat{X}_{0},\widehat{\FF}_{0})\rightarrow(X,\FF)$ is a foliated map.
\end{prop}

\begin{proof}\setcounter{claim}{0}
According to Proposition~\ref{p: defining cocycle of widehat FF_0}, this follows by checking the commutativity of each diagram
 $$
  \begin{CD}
    \widehat{U}_{i,0} @>{\hat{p}_{i,0}}>> \widehat{T}_{i,0} \\
    @V{\hat{\pi}_{0}}VV  @VV{\hat{\pi}_{0}}V \\
    U_{i} @>p_{i}>> T_{i}
  \end{CD}
$$
By Lemma~\ref{l: q is a homeomorphism}, and the definition of $\hat{p}_{i,0}$ and $\hat{\pi}_{i,0}$, this commutativity follows from
the commutativity of 
$$
  \begin{CD}
    \widetilde{U}_{i,0} @>>> \widehat{T}_{i,0} \\
    @VVV  @VV{\hat{\pi}_{0}}V \\
    U_{i} @>p_{i}>> T_{i}
  \end{CD}
$$
where the left vertical and the top horizontal arrows denote the restrictions of the first and second factor projections of $\check{U}_{i,0}= U_{i}\times\widehat{T}_{i,0}\times\{i\}$. But the commutativity of this diagram holds by the definition of $\widetilde{X}_{0}$ and $ \widetilde{U}_{i,0}$.
\end{proof}

\begin{prop}\label{p: holonomy covers}
 The restrictions of  $\hat{\pi}_{0}:\widehat{X}_{0}\rightarrow X$ to the leaves are the holonomy covers of the leaves of $\FF$.
\end{prop} 

\begin{proof}\setcounter{claim}{0}
 With the notation of the proof of Proposition~\ref{p: defining cocycle of widehat FF_0}, the diagram 
 	\begin{equation}\label{holonomy covers}
 		\begin{CD}
    			\widehat{U}_{i,0} @>{\hat{\phi}_{i,0}}>> B_{i}\times \widehat{T}_{i,0} \\
    			@V{\hat{\pi}_{0}}VV  @VV{\id_{B_i} \times\hat{\pi}_{0}}V \\
    			U_{i} @>\phi_{i}>>B_{i}\times T_{i}
  		\end{CD}
	\end{equation}
is commutative, and $\widehat{U}_{i,0}=\hat{\pi}_0^{-1}(U_i)$. Hence, for corresponding plaques in $U_{i}$ and
 $ \widehat{U}_{i,0}$, $P_{z}=\phi^{-1}_{0}(B_{i}\times\{\hat{z}\})$ and
$\widehat{P}_{\hat{z}}=\hat{\phi}^{-1}_{0}(B_{i}\times\{z\})$ with $z\in T_{i}$ and $\hat{z}\in \hat{\pi}^{-1}_{0}(z)\subset \widehat{T}_{i,0}$,
the restriction $\hat{\pi}_{0}:\widehat{P}_{\hat{z}}\rightarrow P_{z}$ is a homeomorphism. It follows easily that $\hat{\pi}_{0}:\widehat{X}_{0}\rightarrow X$ restrics to covering maps of the leaves of $\widehat{\FF}_{0}$ to the leaves of $\FF$. In fact, these are the holonomy covers, which can be seen as follows.

According to the proof of Proposition~\ref{p: hat pi_0 is continuous and surjective} and the definition of the equivalence relation ``$\sim$'' on $\widetilde X_0$, for each 
$x$ in $U_{i}\cap U_{j}$, we have homeomorphisms
 $$
   \begin{CD} 
     \hat{\pi}^{-1}_{0}(p_{i}(x)) @<{\hat{p}_{i,0}}<< \hat{\pi}^{-1}_{0}(x) @>{\hat{p}_{j,0}}>> \hat{\pi}^{-1}_{0}(p_{j}(x))
   \end{CD} 
 $$
satisfying $\hat{p}_{j,0}\hat{p}^{-1}_{i,0}=\widehat{h_{ij}}$. This easily implies the following. Given $x\in U_{i}$ and $\hat{x}\in \hat{\pi}^{-1}_{0}(X)$,
denoting by  $L$ and  $\widehat{L}$ the leaves through $x$ and $\hat{x}$, respectively, and given a loop $c$ in $L$ basisd at $x$ inducing a local holonomy transformation $h\in S$ around $p_{i}(x)$ in $T_{i}$, the lift $\hat{c}$ of $c$ to $\widehat{L}$ with $\hat{c}(0)=\hat{x}$ 
satisfies $\hat{p}_{i,0}\hat{c}(1)=\hat{h}\hat{p}_{i,0}(\hat{x})$. Writing  $\hat{p}_{i,0}(\hat{x})=\germ(f,p_{i}(x))$, we obtain
\[
        \hat{p}_{i,0}\hat{c}(1)= \hat{h}(\germ(f,p_{i}(x)))=\germ(fh,p_{i}(x))\;.
\]
Thus $\hat{c}$ is  a loop if and only if $\germ(fh,p_{i}(x))=\germ(f,p_{i}(x))$, which means $\germ(h,p_{i}(x))=\germ(\id_T,p_{i}(x))$.
So $\widehat{L}$ is the holonomy cover of $L$.
\end{proof}

\begin{prop}\label{p: hat pi_0:widehat X_0 to X is open}
	The map $\hat\pi_0:\widehat X_0\to X$ is open
\end{prop}

\begin{proof}\setcounter{claim}{0}
	This follows from Corollary~\ref{c: hat pi_0:widehat T_0 rightarrow T is open} and the commutativity of~\eqref{holonomy covers}.
\end{proof}

Theorem~\ref{mt: topological Molino} is the combination of the results of this section.

\subsection{Independence of the choices involved}

Let $x_1$ be another point of $X$, and let $\widehat X_1$, $\widehat\FF_1$ and $\hat\pi_1:\widehat X_1\to X$ be constructed like $\widehat X_0$, $\widehat\FF_0$ and $\hat\pi_0:\widehat X_0\to X$ by using $x_1$ instead of $x_0$. 

\begin{prop}\label{p: independence of x_1}
	There is a foliated homeomorphism $\hat\theta:(\widehat X_0,\widehat\FF_0)\to(\widehat X_1,\widehat\FF_1)$ such that $\hat\pi_1 F=\hat\pi_0$.
\end{prop}

\begin{proof}\setcounter{claim}{0}
	Take an index $i_1$ such that $x_1\in U_{i_1}$. Let $\widehat S_1$, $\widehat T_1$, $\widehat\HH_1$ and $\hat\pi_1:\widehat T_1\to T$ be constructed like $\widehat S_0$, $\widehat T_0$, $\widehat\HH_0$ and $\hat\pi_0:\widehat T_0\to T$ by using $p_{i_1}(x_1)$ instead of $p_{i_0}(x_0)$, and let $\widehat T_{i,1}={\hat\pi_1}^{-1}(T_i)$. Then the construction of $\widehat X_1$, $\widehat\FF_1$ and $\hat\pi_1:\widehat X_1\to X$ involves the objects $\check{X}_1$, $\widetilde{X}_1$, $\check{U}_{i,1}$, $\widetilde{U}_{i,1}$, $\widehat{U}_{i,1}$, $\check{p}_{i,1}$, $\tilde{p}_{i,1}$, $\hat{p}_{i,1}$, $\check{\phi}_{i,1}$, $\tilde{\phi}_{i,1}$ and $\hat{\phi}_{i,1}$, defined like $\check{X}_0$, $\widetilde{X}_0$, $\check{U}_{i,0}$, $\widetilde{U}_{i,0}$, $\widehat{U}_{i,0}$, $\check{p}_{i,0}$, $\tilde{p}_{i,0}$, $\hat{p}_{i,0}$, $\check{\phi}_{i,0}$, $\tilde{\phi}_{i,0}$ and $\hat{\phi}_{i,0}$, by using $\widehat T_{i,1}$ and $\hat\pi_1:\widehat T_{i,1}\to T_i$ instead of $\widehat T_{i,0}$ and $\hat\pi_0:\widehat 
T_{i,0}\to T_i$. Let $\theta:\widehat{T}_{0}\to \widehat{T}_{1}$ be the homeomorphism given by Proposition~\ref{p: widehat T_0 to widehat T_1}, which obviously restricts to homeomorphisms $\theta_i:\widehat{T}_{i,0}\to \widehat{T}_{i,1}$. Since $\hat{\pi}_{0}=\hat{\pi}_{1}\theta$, it follows that each homeomorphism
		$$
			\check\theta_i=\id_{U_i}\times\theta_i\times\id:\check{U}_{i,0}
			=U_{i}\times\widehat{T}_{i,0}\times \{i\}\to\check{U}_{i,1}
			=U_{i} \times\widehat{T}_{i,1}\times \{i\}
		$$
	restricts to a homeomorphism $\tilde\theta_i=\widetilde{U}_{i,0}\to\widetilde{U}_{i,1}$. The combination of the homeomorphisms $\tilde\theta_i$ is a homeomorphism $\tilde\theta:\widetilde X_0\to\widetilde X_1$. 
	
	For each $h\in S$, use the notation $\hat h_0\in\widehat S_0$ and $\hat h_1\in\widehat S_1$ for the map $\hat h$ defined with $p_{i_0}(x_0)$ and $p_{i_1}(x_1)$, respectively. From the proof of Proposition~\ref{p: widehat T_0 to widehat T_1}, we get $\hat h_1\theta=\theta\hat h_0$ for all $h\in S$; in particular, this holds with $h=h_{ij}$. So $\tilde\theta:\widetilde X_0\to\widetilde X_1$ is compatible with the equivalence relations used to define $\widehat X_0$ and $\widehat X_1$, and therefore it induces a homeomorphism $\hat\theta:\widehat X_0\to\widehat X_1$. Note that $\hat\theta$ restricts to homeomorphisms $\hat\theta_i:\widehat U_{i,0}\to\widehat U_{i,1}$. Obviously, $\check{p}_{i,1}\check\theta_i=\theta_i\check{p}_{i,1}$, yielding $\tilde{p}_{i,1}\tilde\theta_i=\theta_i\tilde{p}_{i,1}$, and 
therefore $\hat{p}_{i,1}\hat\theta_i=\theta_i\hat{p}_{i,1}$. It follows that $\hat\theta$ is a foliated map.	
\end{proof}

Let $\{U'_a,p'_a,h'_{ab}\}$ be another defining cocycle of $\FF$ induced by a regular foliated atlas. Then let $\widehat X'_0$, $\widehat\FF'_0$ and $\hat\pi'_0:\widehat X'_0\to X$ be constructed like $\widehat X_0$, $\widehat\FF_0$ and $\hat\pi_0:\widehat X_0\to X$ by using $\{U'_a,p'_a,h'_{ab}\}$ instead of $\{U_i,p_i,h_{ij}\}$.

\begin{prop}\label{p: independence of the defining cocycle}
	There is a foliated homeomorphism $F:(\widehat X_0,\widehat\FF_0)\to(\widehat X'_0,\widehat\FF'_0)$ such that $\hat\pi'_0 F=\hat\pi_0$.
\end{prop}

\begin{proof}\setcounter{claim}{0}
	By using a common refinement of the open coverings $\{U_i\}$ and $\{U'_a\}$, we can assume that $\{U'_a\}$ refines $\{U_i\}$. In this case, the union of the defining cocycles $\{U_i,p_i,h_{ij}\}$ and  $\{U'_a,p'_a,h'_{ab}\}$ is contained in another defining cocycle induced by a regular foliated atlas. Thus the proof boils down to showing that a sub-defining cocycle\footnote{A sub-defining cocycle is a defining cocycle contained in another one.} $\{U_{i_k},p_{i_k},h_{i_ki_l}\}$ of  $\{U_i,p_i,h_{ij}\}$ induces a foliated space homeomorphic to $(\widehat X_0,\widehat\FF_0)$. But the pseudogroup $\HH'$ induced by $\{U_{i_k},p_{i_k},h_{i_ki_l}\}$ is the restriction of $\HH$ to an open subset $T'\subset T$, and the pseudo$*$group induced by $\{U_{i_k},p_{i_k},h_{i_ki_l}\}$ is $S'=S\cap\HH'$. Then, by using the canonical identity given by Proposition~\ref{p:HH'}, it easily follows that the foliated space $(\widehat X'_0,\widehat\FF'_0)$ defined with $\{U_{i_k},p_{i_k},h_{i_ki_l}\}$ can be canonically identified with an open foliated subspace of $(\widehat X_0,\widehat\FF_0)$, which indeed is the whole of $(\widehat X_0,\widehat\FF_0)$ because $\{U_{i_k}\}$ covers $X$.
\end{proof}

The following definition makes sense by Propositions~\ref{p: independence of x_1}--\ref{p: independence of the defining cocycle}, and the results used to justify Definition~\ref{d: structural local group}.

\begin{defn}\label{d: structural local group of (X,FF)}
 In Corollary~\ref{c: G of FF}, (the local isomorphism class of) $G$ is called the {\em structural local group\/} of $(X,\FF)$.
\end{defn}

\section{Growth of equicontinuous pseudogroups and foliated spaces}\label{s: growth}

\subsection{Coarse quasi-isometries and growth of metric spaces}\label{ss: growth metric}

A {\em net\/} in a metric space $M$, with metric $d$, is a subset $A\subset M$ that satisfies $d(x,A)\le C$ for some $C>0$ and all $x\in M$; the term {\em $C$-net\/} is also used.  A {\em coarse quasi-isometry\/} between $M$ and another metric space $M'$ is a bi-Lipschitz bijection between nets of $M$ and $M'$; in this case, $M$ and $M'$ are said to be {\em coarsely quasi-isometric\/} (in the sense of Gromov) \cite{Gromov1993}. If such a bi-Lipschitz bijection, as well as its inverse, has dilation $\le\lambda$, and it is defined between $C$-nets, then it will be said that the coarse quasi-isometry has {\em distortion\/} $(C,\lambda)$. A family of coarse quasi-isometries with a common distortion will be called a family of {\em equi-coarse quasi-isometries\/}, and the corresponding metric spaces are called {\em equi-coarsely quasi-isometric\/}. 

The version of growth for metric spaces given here is taken from \cite{AlvCandel:gcgol,AlvWolak:growth}. 

Recall that, given non-decreasing functions\footnote{Usually, growth types are defined by using non-decreasing functions $\Z^+\to[0,\infty)$, but non-decresing functions $[0,\infty)\to[0,\infty)$ give rise to an equivalent concept.} $u,v:[0,\infty)\to[0,\infty)$, it is said that $u$ is {\em dominated\/} by $v$, written $u\preccurlyeq v$, when there are $a,b\ge1$ and $c\ge0$ such that $u(r)\le a\,v(br)$ for all $r\ge c$. If $u\preccurlyeq v\preccurlyeq u$, then it is said that $u$ and $v$ represent the same {\em growth type\/} or have {\em equivalent growth\/}; this is an equivalence relation, and ``$\preccurlyeq$'' defines a partial order relation between growth types called {\em domination\/}. For a family of pairs of non-decreasing functions $[0,\infty)\to[0,\infty)$, {\em equi-domination\/} means that those pairs satisfy the above condition of domination with the same constants $a,b,c$. A family of functions $[0,\infty)\to[0,\infty)$ will be said to have {\em equi-equivalent growth\/} if they equi-dominate one another.

For a complete connected Riemannian manifold $L$, the growth type of each mapping $r\mapsto\vol B(x,r)$ is independent of $x$ and is called the {\em growth type\/} of $L$. Another definition of {\em growth type\/} can be similarly given for metric spaces whose bounded sets are finite, where the number of points is used instead of the volume.

Let $M$ be a metric space with metric $d$. A {\em quasi-lattice\/} $\Gamma$ of $M$ is a $C$-net of $M$ for some $C\ge0$ such that, for every $r\ge0$, there is some $K_r\ge0$ such that $\card(\Gamma\cap B(x,r))\le K_r$ for every $x\in M$. It is said that $M$ is of {\em coarse bounded geometry\/} if it has a quasi-lattice. In this case, the {\em growth type\/} of $M$ can be defined as the growth type of any quasi-lattice $\Gamma$ of $M$; i.e., it is the growth type of the {\em growth function\/} $r\mapsto v_\Gamma(x,r)=\card(B(x,r)\cap\Gamma)$ for any $x\in\Gamma$. This definition is independent of $\Gamma$. 

For a family of metric spaces, if they satisfy the above condition of coarse bounded geometry with the same constants $C$ and $K_r$, then they are said to have {\em equi-coarse bounded geometry\/}. If moreover the lattices involved in this condition have growth functions with equi-equivalent growth, then these metric spaces are said to have {\em equi-equivalent growth\/}.

 The condition of coarse bounded geometry is satisfied by complete connected Riemannian manifolds of bounded geometry, and also by discrete metric spaces with a uniform upper bound on the number of points in all balls of each given radius \cite{BlockWeinberger1997}. In those cases, the two given definitions of growth type are equal.

\begin{lem}[{\'Alvarez-Candel \cite{AlvCandel2009}; see also \cite[Lemma~2.1]{AlvWolak:growth}}]
\label{l:growth type}
  Two coarsely quasi-isometric metric spaces of coarse bounded
  geometry have the same growth type. Moreover, if a family of metric spaces are equi-coarsely quasi-isometric to each other, then they have equi-equivalent growth.
\end{lem}

\subsection{Quasi-isometry and growth types of orbits}\label{ss: q-i}

Let $\HH$ be a pseudogroup on a space $T$, and $E$ a symmetric set of generators of $\HH$. Let $\mathfrak G$ be the groupoid of germs of maps in $\HH$.

For each $h\in\HH$ and $x\in\dom h$, let $|h|_{E,x}$ be the length of the shortest expression of $\germ(h,x)$ as a product of germs of maps in $E$ (being $0$ if $\germ(h,x)=\germ(\id_T,x)$). For each $x\in T$, define metrics $d_E$ on $\HH(x)$ and $\mathfrak G_x$ by
\begin{gather*} 
  d_E(y,z)=\min\{\,|h|_{E,y}\mid h\in\HH,\ y\in\dom h,\ h(y)=z\,\}\;,\\ 
  d_E(\germ(f,x),\germ(g,x))=|fg^{-1}|_{E,g(x)}\;.
\end{gather*}  
Notice that
  \[
    d_E(f(x),g(x))\le d_E(\germ(f,x),\germ(g,x))\;.
  \]
Moreover, on the germ covers, $d_E$ is right invariant in the sense that, if $y\in\HH(x)$, the bijection $\mathfrak G_y\to\mathfrak G_x$, given by right multiplication with any element in $\mathfrak G_x^y$, is isometric; so the isometry types of the germ covers of the orbits make sense without any reference to basis points. In fact, the definition of $d_E$ on $\mathfrak G_x$ is analogous to the definition of the right invariant metric $d_S$ on a group $\Gamma$ induced by a symmetric set of generators $S$: $d_S(\gamma,\delta)=|\gamma\delta^{-1}|$ for $\gamma,\delta\in\Gamma$, where $|\gamma|$ is the length of the shortest expression of $\gamma$ as a product of elements of $S$ (being $0$ if $\gamma=e$). 

Assume that $\HH$ is compactly generated and $T$ locally compact. Let $U\subset T$ be a relatively compact open subset that meets all $\HH$-orbits, let $\GG=\HH|_U$, and let $E$ be a symmetric system of compact generation of $\HH$ on $U$. With these conditions, the quasi-isometry type of the $\GG$-orbits with $d_E$ may depend on $E$ \cite[Section~6]{AlvCandel2009}. So the following additional condition on $E$ is considered.

\begin{defn}[{\'Alvarez-Candel \cite[Definition~4.2]{AlvCandel2009}}]\label{d:recurrent finite symmetric family of generators} With the above notation, it is said that $E$ is {\em recurrent\/} if, for any relatively compact open subset $V\subset U$ that meets all $\GG$-orbits, there exists some $R>0$ such that $\GG(x)\cap V$ is an $R$-net in $\GG(x)$ with $d_E$ for all $x\in U$.
\end{defn}

The role played by $V$ in Definition~\ref{d:recurrent finite symmetric family of generators} can be played by any relatively compact open subset meeting all orbits \cite[Lemma~4.3]{AlvCandel2009}. Furthermore there exists a recurrent system of compact generation on $U$ \cite[Corollary~4.5]{AlvCandel2009}.

\begin{thm}[{\'Alvarez-Candel \cite[Theorem~4.6]{AlvCandel2009}}]
\label{t:quasi-isometric orbits}
Let $\HH$ and $\HH'$ be compactly generated pseudogroups on locally compact spaces $T$ and $T'$, let $U$ and $U'$ be relatively compact open subsets of $T$ and $T'$ that meet all orbits of $\HH$ and $\HH'$, let $\GG$ and $\GG'$ denote the restrictions of $\HH$ and $\HH'$ to $U$ and $U'$, and let $E$ and $E'$ be recurrent symmetric systems of compact generation of $\HH$ and $\HH'$ on $U$ and $U'$, respectively. Suppose that there exists an equivalence $\HH\to\HH'$, and consider the induced equivalence $\GG\to\GG'$ and homeomorphism $U/\GG\to U'/\GG'$. Then the $\GG$-orbits with $d_E$ are equicoarsely quasi-isometric to the corresponding $\GG'$-orbits with $d_{E'}$.
\end{thm}

An obvious modification of the arguments of the proof of \cite[Theorem~4.6]{AlvCandel2009} gives the following.

\begin{thm}\label{t:quasi-isometric germ covers of orbits}
  With the notation and conditions of Theorem~\ref{t:quasi-isometric orbits}, the germ covers of the $\GG$-orbits with $d_E$ are equicoarsely quasi-isometric to the germ covers of the corresponding $\GG'$-orbits with $d_{E'}$.
\end{thm}

\begin{cor}\label{c:growth orbits and their germ covers}
  With the notation and conditions of Theorem~\ref{t:quasi-isometric orbits}, the corresponding orbits of $\GG$ and $\GG'$, as well as their germ covers, have equi-equivalent growth.
\end{cor}
  
\begin{proof}\setcounter{claim}{0}
  This follows from Lemma~\ref{l:growth type} and Theorems~\ref{t:quasi-isometric orbits} and~\ref{t:quasi-isometric germ covers of orbits}.
\end{proof}

\begin{ex}\label{ex: G growth}
Let $G$ be a locally compact Polish local group with a left-invariant metric, let $\Gamma\subset G$ be a dense finitely generated sub-local group, and let $\HH$ denote the pseudogroup generated by the local action of $\Gamma$ on $G$ by local left translations. Suppose that $\HH$ is compactly generated, and let $\GG=\HH|_U$ for some relatively compact open neighborhood $U$ of the identity element $e$ in $G$, which meets all $\HH$-orbits because $\Gamma$ is dense. For every $\gamma\in\Gamma$ with $\gamma U\cap U\ne\emptyset$, let $h_\gamma$ denote the restriction $U\cap\gamma^{-1}U\to\gamma U\cap U$ of the local left translation by $\gamma$. There is a finite symmetric set $S=\{s_1,\dots,s_k\}\subset\Gamma$ such that $E=\{h_{s_1},\dots,h_{s_k}\}$ is a recurrent system of compact generation of $\HH$ on $U$; in fact, by reducing $\Gamma$ if necessary, we can assume that $S$ generates $\Gamma$. The recurrence of $E$ means that there is some $N\in\N$ such that 
     \begin{equation}\label{E^N}
       U=\bigcup_{h\in E^N}h^{-1}(V\cap\im h)\;,
     \end{equation}
where $E^N$ is the family of compositions of at most $N$ elements of $E$. 

For each $x\in U$, let
  \[
    \Gamma_{U,x}=\{\,\gamma\in\Gamma\mid \gamma x\in U\,\}\;.
  \]
Let $\mathfrak G$ denote the topological groupoid of germs of $\GG$. The map $\Gamma_{U,x}\to\mathfrak G_x$, $\gamma\mapsto\germ(h_\gamma,x)$, is bijective. For $\gamma\in\Gamma_{U,x}$, let $|\gamma|_{S,U,x}:=|h_\gamma|_{E,x}$. Thus $|e|_{S,U,x}=0$, and, if $\gamma\ne e$, then $|\gamma|_{S,U,x}$ equals the minimum $n\in\N$ such that there are $i_1,\dots,i_n\in\{1,\dots,k\}$ with  $\gamma=s_{i_n}\cdots s_{i_1}$ and $s_{i_m}\cdots s_{i_1}\cdot x\in U$ for all $1\le m\le n$. Moreover $d_E$ on $\mathfrak G_x$ corresponds to the metric $d_{S,U,x}$ on $\Gamma_{U,x}$ given by
  \[
    d_{S,U,x}(\gamma,\delta)=|\delta\gamma^{-1}|_{S,U,\gamma(x)}\;.
  \]
Observe that, for all $\gamma\in\Gamma_{U,x}$ and $\delta\in\Gamma_{U,\gamma\cdot x}$,
  \begin{alignat}{2}
    \delta\gamma&\in\Gamma_{U,x}\;,&\quad
    |\delta\gamma|_{S,U,x}&\le|\gamma|_{S,U,x}+|\delta|_{S,U,\gamma\cdot x}\;,\label{delta gamma}\\
    \gamma^{-1}&\in\Gamma_{U,\gamma\cdot x}\;,&\quad
    |\gamma|_{S,U,x}&=|\gamma^{-1}|_{S,U,\gamma\cdot x}\;.\label{gamma^-1}
  \end{alignat}
\end{ex}

In this example, we will be interested on the growth type of the orbits of $\GG$ with $d_E$, or, equivalently, the growth type of the metric spaces $(\Gamma_{U,x},d_{S,U,x})$. The following result was used by  Breuillard-Gelander to study this growth type when $G$ is a Lie group.

\begin{prop}[{Breuillard-Gelander \cite[Proposition~10.5]{BreuillardGelander2007}}]\label{p: Breuillard-Gelander}
	Let $G$ be a non-nilpotent connected real Lie group and $\Gamma$ a finitely generated dense subgroup. For any finite set $S = \{s_1,\dots, s_k\}$ of generators of $\Gamma$, and any neighborhood $B$ of $e$ in $G$, there are elements $t_i\in\Gamma\cap s_iB$ {\rm(}$i \in\{1,\dots, k\}${\rm)} which freely generate a free
semi-group. If $G$ is not solvable, then we can choose the elements $t_i$ so that they generate a free group.
\end{prop}

\subsection{Growth of equicontinuous pseudogroups}\label{ss: growth pseudogroups}

Let $G$ be a locally compact Polish local group with a left-invariant metric, let $\Gamma\subset G$ be a dense finitely generated sub-local group, and let $\HH$ denote the pseudogroup generated by the local action of $\Gamma$ on $G$ by local left translations. Suppose that $\HH$ is compactly generated. Let $\GG=\HH|_U$ for some relatively compact open neighborhood $U$ of the identity element $e$ in $G$, which meets all $\HH$-orbits because $\Gamma$ is dense. Let $E$ be a recurrent symmetric system of compact generation of $\HH$ on $U$.  Let $\mathfrak G$ be the groupoid of germs of maps in $\GG$.

\begin{thm}\label{t:growth homog pseudogroup}
  With the above notation and conditions, one of the following properties hold:
	\begin{itemize}
		
		\item $G$ can be approximated by nilpotent local Lie groups; or
			
		\item the germ covers of all $\GG$-orbits have exponential growth with $d_E$.
		
	\end{itemize}
\end{thm}

\begin{proof}\setcounter{claim}{0}
	By Theorem~\ref{t: jacoby-approximated}, there is some $U_0\in \Psi G$, contained in any given element of $\Psi G\cap\Phi(G,2)$, and there exists a sequence of compact normal subgroups $F_n\subset U_0$ such that $F_{n+1}\subset F_n$, $\bigcap_nF_n=\{e\}$, $(F_n,U_0)\in \Delta G$, and $G/(F_n,U_0)$ is a local Lie group. Let $T_n:U_0^2\to G/(F_n,U_0)$ denote the canonical projection. Take an open neighborhood $U_1$ of $e$ such that $\overline{U_1}\subset U_0$. Then $F_n\overline{U_1}\subset U_0$ for $n$ large enough by the properties of the sequence $F_n$. Let $U_2=F_nU_1$ for such an $n$. Thus $U_2$ is saturated by the fibers of $T_n$, and $\overline{U_2}\subset U_0$. Then $U':=T_n(U_2)$ is a relatively compact open neighborhood of the identity in the local Lie group $G':=G/(F_n,U_0)$. Let $\Gamma'=T_n(\Gamma\cap U_0^2)$, which is a dense sub-local group of $G'$, and let $\HH'$ denote the pseudogroup on $G'$ generated by the local action of $\Gamma'$ by local left translations. 
	
	For every $\gamma\in\Gamma\cap U_0$ with $\gamma U_2\cap U_2\ne\emptyset$, let $h_\gamma$ denote the restriction $U_2\cap\gamma^{-1}U_2\to\gamma U_2\cap U_2$ of the local left translation by $\gamma$. There is a finite symmetric set $S=\{s_1,\dots,s_k\}\subset\Gamma$ such that $E_2=\{h_{s_1},\dots,h_{s_k}\}$ is a recurrent system of compact generation of $\HH$ on $U_2$. By reducing $\Gamma$ if necessary, we can suppose that $S$ generates $\Gamma$. For every $\delta\in\Gamma'$ with $\delta U'\cap U'\ne\emptyset$, let $h'_\delta$ denote the restriction $U'\cap\delta^{-1}U'\to\delta U'\cap U'$ of the local left translation by $\delta$. We can assume that $s_1,\dots,s_k$ are in $U_2$, and therefore we can consider their images $s'_1,\dots,s'_k$ by $T_n$. Moreover each $h_{s_i}$ induces via $T_n$ the map $h'_{s_i}$, and $E'=\{h'_{s_1},\dots,h'_{s_k}\}$ is a system of compact generation of $\HH'$ on $U'$. By increasing $E_2$ if necessary, we can assume that $E'$ is also recurrent. Fix any open set $V'$ in $G'$ with 
$\overline{V'}\subset U'$. Then $V=T_n^{-1}(V')$ satisfies $\overline V\subset U_2$. 
  
  \begin{claim}\label{cl:U_2}
    For each finite subset $F\subset\Gamma\cap U_2$, we have $U_2\subset\bigcup_{\gamma\in\Gamma\setminus F}\gamma V$.
  \end{claim}
  
  Since $U_2$ and $V$ are saturated by the fibers of $T_n$, Claim~\ref{cl:U_2} follows by showing that $U'\subset\bigcup_{\gamma\in\Gamma'\setminus F'}\gamma V'$, where $F'=T_n(F)$. Suppose that this inclusion is false. Then there is some finite symmetric subset $F\subset\Gamma\cap U_2$ and some $x\in U'$ such that $((\Gamma'\setminus F')x)\cap V'=\emptyset$. By the recurrence of $E'$, there is some $N\in\N$ satisfying~\eqref{E^N} with $U'$ and $E'$. Since $\Gamma'_{U',x}$ is infinite because $\Gamma'$ is dense in $G'$, it follows that there is some $\gamma\in\Gamma'_{U',x}\setminus F'$ such that 
    \begin{equation}\label{|gamma|_S,U,x}
      |\gamma|_{S',U',x}>N+\max\{\,|\epsilon|_{S',U',x}\mid\epsilon\in F'\cap\Gamma'_{U',x}\,\}\;.
    \end{equation}
  By~\eqref{E^N}, there is some $h\in {E'}^N$ such that $\gamma x\in h^{-1}(V'\cap\im h')$. We have $h=h'_\delta$ for some $\delta\in\Gamma'$. Note that $\delta\in\Gamma'_{U',\gamma' x}$ and $|\delta|_{S',U',\gamma x}\le N$. Hence
    \begin{multline*}
      |\gamma|_{S',U',x}\le|\delta\gamma|_{S',U',x}+|\delta^{-1}|_{S',U',\delta\gamma x}\\
      =|\delta\gamma|_{S',U',x}+|\delta|_{S',U',\gamma' x}\le|\delta\gamma|_{S',U',x}+N
    \end{multline*}
  by~\eqref{delta gamma} and~\eqref{gamma^-1}, obtaining that $\delta\gamma\not\in F'$ by~\eqref{|gamma|_S,U,x}. However, $\delta\gamma x\in V'$, obtaining a contradiction, which completes the proof of Claim~\ref{cl:U_2}.
  
  \begin{claim}\label{cl:overline U_2}
    For each finite subset $F\subset\Gamma\cap U_2$, we have $\overline{U_2}\subset\bigcup_{\gamma\in \Gamma\setminus F}\gamma V$.
  \end{claim}
  
  Take a relatively compact open subset $O_1\subset G$ such that $\overline{U_1}\subset O_1$ and $F_n\overline{O_1}\subset U_0$. Let $O_2=F_nO_1$ and $\KK=\HH|_{O_2}$. Then Claim~\ref{cl:overline U_2} follows by applying Claim~\ref{cl:U_2} to $O_2$.
  
  According to Claim~\ref{cl:overline U_2}, by increasing $S$ if necessary, we can suppose that
    \begin{equation}\label{s_i}
      \overline{U_2}\subset\bigcup_{i<j}(s_i\cdot V\cap s_j\cdot V)=\bigcup_{i<j}(s_i^{-1}\cdot V\cap s_j^{-1}\cdot V)\;.
    \end{equation}
    
  Suppose that $G$ cannot be approximated by nilpotent local Lie groups. Then we can assume that the local Lie group $G'$ is not nilpotent. Moreover we can suppose that $G'$ is a sub-local Lie group of a simply connected Lie group $L$. Let $\Delta$ be the dense subgroup of $L$ whose intersection with $G'$ is $\Gamma'$. Then, by Proposition~\ref{p: Breuillard-Gelander}, there are elements $t'_1,\dots,t'_k$ in $\Delta$, as close as desired to $s'_1,\dots,s'_k$, which are free generators of a free semi-group. If the elements $t'_i$ are close enough to $s'_i$, then they are in $U'$. So there are elements $t_i\in U_2$ such that $T_n(t_i)=t'_i$. By the compactness of $\overline{U_2}$, and because $U_2$ and $V$ are saturated by the fibers of $T_n$, if $t'_1,\dots,t'_k$ are close enough to $s'_1,\dots,s'_k$, then~\eqref{s_i} gives 
    \begin{equation}\label{t_i}
      \overline{U_2}\subset\bigcup_{i<j}(t_i^{-1}V\cap t_j^{-1}V)\;.
    \end{equation}
  
  Now, we adapt the argument of the proof of \cite[Lemma~10.6]{BreuillardGelander2007}. Let $\widehat\Gamma\subset\Gamma$ be the sub-local group generated by $t_1,\dots,t_k$; thus $\widehat S=\{t_1^{\pm1},\dots,t_k^{\pm1}\}$ is a symmetric set of generators of $\widehat\Gamma$, and $S\cup\widehat S$ is a symmetric set of generators of $\Gamma$. With $\widehat E=\{h_{t_1}^{\pm1},\dots,h_{t_k}^{\pm1}\}$, observe that $E_2\cup\widehat E$ is a recurrent system of compact generation of $\HH$ on $U_2$. Given $x\in U_2$, let $\mathsf{S}(n)$ be the sphere with center $e$ and radius $n\in\N$ in $\widehat\Gamma_{U_2,x}$ with $d_{\widehat S,U,x}$. By~\eqref{t_i}, for each $\gamma\in\mathsf{S}(n)$, there are indices $i<j$ such that $\gamma x\in t_i^{-1}V\cap t_j^{-1}V$. So the points $t_i\gamma x$ and $t_j\gamma x$ are in $V$, obtaining that $t_i\gamma,t_j\gamma\in\mathsf{S}(n+1)$. Moreover all elements obtained in this way from elements of $\mathsf{S}(n)$ are pairwise distinct because $t'_1,\dots,t'_k$ freely generate a free semigroup. Hence $\card(\mathsf{S}(n+1))\ge2\card(\mathsf{S}(n))$, giving $\card(\mathsf{S}(n))\ge2^n$. So $(\widehat\Gamma_{U_2,x},d_{\widehat S,U_2,x})$ has exponential growth. Since $\widehat\Gamma_{U_2,x}\subset\Gamma_{U_2,x}$ and $d_{S\cup\widehat S,U_2,x}\le d_{\widehat S,U_2,x}$ on $\widehat\Gamma_{U_2,x}$, it follows that $(\Gamma_{U_2,x},d_{S\cup\widehat S,U_2,x})$ also has exponential growth. So $(\mathfrak G_x,d_{E_2\cup\widehat E})$ has exponential growth, obtaining that $(\mathfrak G_x,d_E)$ has exponential growth by Corollary~\ref{c:growth orbits and their germ covers}.
\end{proof}

\subsection{Growth of equicontinuous foliated spaces}\label{ss: growth foliated spaces}

Let $X\equiv(X,\FF)$ be a compact Polish foliated space. Let $\{U_i,p_i,h_{ij}\}$ be a defining cocycle of $\FF$, where $p_i:U_i\to T_i$ and $h_{ij}:T_{ij}\to T_{ji}$, and let $\HH$ be the induced representative of the holonomy pseudogroup. As we saw in Section~\ref{ss: prelim equicont foliated spaces}, $\HH$ can be considered as the restriction of some compactly generated pseudogroup $\HH'$ to some relatively compact open subset, and $E=\{h_{ij}\}$ is a system of compact generation on $T$. Moreover \'Alvarez and Candel \cite{AlvCandel2009} observed that $E$ is recurrent. According to Theorems~\ref{t:quasi-isometric orbits} and~\ref{t:quasi-isometric germ covers of orbits}, it follows that the quasi-isometry type of the $\HH$-orbits and their germ covers with $d_E$ are independent of the choice of $\{U_i,p_i,h_{ij}\}$ under the above conditions; thus they can be considered as quasi-isometry types of the corresponding leaves and their holonomy covers. 

This has the following interpretation  when $X$ is a smooth manifold. In this case, given any Riemannian metric $g$ on $X$, for each leaf $L$, the differentiable (and coarse) quasi-isometry type of $g|_L$ is independent of the choice of $g$; they depend only on $\FF$ and $L$; in fact, it is coarsely quasi-isometric to the corresponding $\HH$-orbit, and therefore they have the same growth type \cite{Carriere1988} (this is an easy consequence of the existence of a uniform bound of the diameter of the plaques). Similarly, the germ covers of the $\HH$-orbits are also quasi-isometric to the holonomy covers of the corresponding leaves.

Theorem~\ref{mt: growth} follows from these observations and Theorem~\ref{t:growth homog pseudogroup}.

\section{Examples and open problems}\label{s: ex}

Theorems~\ref{mt: topological Molino} and~\ref{mt: growth} may be relevant in the following examples; most of them are taken from \cite[Chapter~11]{CandelConlon2000-I}.
	
\begin{ex}\label{ex: locally free action}
	Any locally free action of a connected Lie group on a locally compact Polish space, $\phi:H\times X\to X$, defines a foliated structure $\FF$ on $X$ whose leaves are the orbits \cite[Theorem~11.3.14]{CandelConlon2000-I}, \cite{Palais1961}. Moreover $\FF$ is equicontinuous if $\phi$ is equicontinuous.
\end{ex}
		
\begin{ex}\label{ex: solenoids}
	A {\em matchbox manifold\/} is a foliated continuum\footnote{Recall that a {\em continuum\/} is a non-empty compact connected metrizable space.} $X\equiv(X,\FF)$ transversely modeled on a totally disconnected space. The case of a single leaf is discarded, and it is assumed that $X$ is $C^1$ in the sense that the changes of foliated coordinates are $C^1$ along the leaves, with transversely continuous leafwise derivatives. An example of matchbox manifold is given by any inverse limit of smooth proper covering maps of compact $n$-manifolds, called an $n$-dimensional \textit{solenoid}; if moreover any composite of a finite number of bounding maps is a normal covering, then it is called a \textit{McCord solenoid}. A matchbox manifold $X$ is equicontinuous if and only if it is a solenoid \cite[Theorem~7.9]{ClarkHurder2013}; and $X$ is homogeneous if and only if it is a McCord solenoid \cite[Theorem~1.1]{ClarkHurder2013}; this is the case where it is a $G$-foliated space. See \cite{AlcaldeLozanoMacho2011} for a generalization using inverse limits of compact branched manifolds.	
\end{ex}
	
\begin{ex}\label{ex: almost periodic}
	Let $C_{b}(\R)$ be the space of continuous bounded functions $\R\to\R$, with the topology of uniform convergence.  For a function $f\in C_{b}(\R)$ and $t\in\R$, let $f_t\in C_{b}(\R)$ be defined by $f_t(r)=f(r+t)$. It is said that $f$ is \textit{almost periodic} if $\{\,f_{t}\mid t\in\R\,\}$ is equicontinuous \cite{Besicovich1954}, \cite{Gottschalk1946}, which means that $\mathfrak M(f):=\overline{\{\,f_{t}\mid t\in\R\,\}}$ is compact in $C_b(\R)$. An equicontinuous flow $\Phi:\R\times\mathfrak M(f)\to\mathfrak M(f)$ is defined by $\Phi_{t}(g)=g_t$. If $f$ is non-constant, then $\Phi$ is nonsingular, defining an equicontinuous foliated structure $\FF$ on $\mathfrak M(f)$. If $f$ is non-periodic, then $\FF$ does not reduce to a single leaf. With more generality, we can consider an almost-periodic non-periodic continuous function $f$ on any connected Lie group with values in a Hilbert space.
\end{ex}
	
\begin{ex}\label{ex: families of manifolds}
	For each $n\in\Z_+$, let $\MM_*(n)$ denote the set\footnote{The logical problems of this definition can be avoided because any complete connected Riemannian manifold is equipotent to some subset of $\R$.} of isometry classes $[M,x]$ of pointed complete connected Riemannian $n$-manifolds $(M,x)$. The $C^\infty$ convergence \cite{Petersen1998} defines a Polish topology on $\MM_*(n)$ \cite[Theorem~1.1]{AlvBarralCandel:universal}. The corresponding space is denoted by $\MM_*^\infty(n)$, and its closure operator by $\Cl_\infty$. For any complete connected Riemannian manifold $M\equiv(M,g)$, a canonical continuous map $\iota:M\to\MM_*^\infty(n)$ is defined by $\iota(x)=[M,x]$. A concept of {\em weak aperiodicity\/} of $M$ was introduced in \cite{AlvBarralCandel:universal}. On the other hand, $M$ is called {\em almost periodic\/} if, for all $m\in\N$, $\epsilon>0$ and $x\in M$, there is a set $H$ of diffeomorphisms of $M$ so that $\sup|\nabla^kh^*g|<\epsilon$ for all $h\in H$ and $k\le m$, and $\{\,h(x)\mid h\in H\,\}$ is a net in $M$.  If $M$ is weakly aperiodic and almost periodic, then $\Cl_\infty(\iota(M))$ canonically becomes a compact minimal equicontinuous foliated space of dimension $n$, as follows from \cite[Theorem~1.2 and Lemma~12.5-(ii)]{AlvBarralCandel:universal} and \cite[Lemma~7.2 and Theorem~4.1]{Lessa2013} (see also \cite[Chapter~10, Theorem~3.3]{Petersen1998}, and \cite{Cheeger1970}).
\end{ex}

\begin{prob}
	In the Examples~\ref{ex: locally free action}--\ref{ex: families of manifolds}, understand the specific application of Theorems~\ref{mt: topological Molino} and~\ref{mt: growth}.
\end{prob}

\begin{prob}
	Use Theorem~\ref{mt: topological Molino} to classify particular classes of equicontinuous foliated spaces.
\end{prob}

\begin{quest}
	Is it possible to improve Theorem~\ref{mt: growth} for special types of structural local groups?
\end{quest}

\begin{quest}
	Is there any consequence of Theorems~\ref{mt: topological Molino} and~\ref{mt: growth} in usual foliation theory, assuming that the minimal sets are equicontinuous?
\end{quest}

The following questions refer to extensions of known properties of Riemannian foliations, where Theorem~\ref{mt: topological Molino} could play an important role.

\begin{quest}
	For compact minimal equicontinuous foliated spaces, does the leafwise heat flow of leafwise differential forms preserve transverse continuity at infinite time?
\end{quest}

\begin{quest}
	Is it possible to give useful extensions of tautness and tenseness to equicontinuous foliated spaces, and relate them to some kind of basic cohomology?
\end{quest}

\bibliographystyle{amsplain}


\providecommand{\bysame}{\leavevmode\hbox to3em{\hrulefill}\thinspace}
\providecommand{\MR}{\relax\ifhmode\unskip\space\fi MR }
\providecommand{\MRhref}[2]{%
  \href{http://www.ams.org/mathscinet-getitem?mr=#1}{#2}
}
\providecommand{\href}[2]{#2}

\end{document}